\documentclass[11pt]{amsart}
\usepackage[all]{xy}
\usepackage{yfonts}
\usepackage{amsmath}
\usepackage{amsfonts}
\usepackage{amssymb}
\usepackage{amsthm}
\usepackage{mathrsfs}
\usepackage{graphicx}
\usepackage{longtable}
\usepackage{tikz}
\usetikzlibrary{positioning}
\usepackage[
bookmarks=false,
colorlinks=true,
debug=true,
naturalnames=true,
pdfnewwindow=true,
citecolor=blue,
linkcolor=blue,
urlcolor = blue]{hyperref}
\usepackage{mathtools}
\numberwithin{equation}{section}

\newtheorem{Thm}{Theorem}[section]
\newtheorem{Cor}[Thm]{Corollary}
\newtheorem{Prop}[Thm]{Proposition}
\newtheorem{Lem}[Thm]{Lemma}
\newtheorem{Conj}[Thm]{Conjecture}

\theoremstyle{definition}
\newtheorem{Def}[Thm]{Definition}
\newtheorem{Rem}[Thm]{Remark}
\newtheorem{Ex}[Thm]{Example}

\newcommand{\Irr}{\mathop{\mathrm{Irr}}\nolimits}

\newcommand{\Z}{\mathbb{Z}}
\newcommand{\Q}{\mathbb{Q}}
\newcommand{\C}{\mathbb{C}}
\newcommand{\kk}{\Bbbk}

\newcommand{\zero}{\mathrm{zero}}

\newcommand{\Cc}{\mathscr{C}}

\newcommand{\fg}{\mathfrak{g}}

\newcommand{\fd}{\textfrak{d}}

\title[Classification of real and imaginary modules]
{Classification of real and imaginary modules of quantum affine algebras in monoidal categorifications of affine cluster algebras}

\date{\today}

\author[H.~Sakamoto]{Heizo Sakamoto}
\address{Graduate School of Mathematical Sciences, the University of Tokyo, 
3-8-1, Komaba, Meguro-ku, Tokyo, 153-8914, Japan}
\email{sakaheisw0327@g.ecc.u-tokyo.ac.jp}

\begin{document}
\begin{abstract}
Recently, Kashiwara-Kim-Oh-Park introduced a wide family of 
monoidal categories of finite-dimensional representations of 
quantum affine algebras, 
which provide monoidal categorifications of cluster algebras. 
In this paper, we prove that, for types $ADE$, 
some of these categories provide monoidal categorifications of 
cluster algebras of affine type.
Moreover, by means of the combinatorial theory of affine type cluster algebras, 
we give a complete classification of real and imaginary simple modules 
in these categories. 
In particular, we show that, in these cases, 
the conjecture asserting that real simple modules correspond exactly 
to cluster monomials holds. 
\end{abstract}

\maketitle

\tableofcontents

\section{Introduction}\label{sec : intro}
Quantum groups, introduced by Drinfeld ~\cite{Dr85} and Jimbo~\cite{Ji85}, 
are $q$-deformations of universal enveloping algebras of 
symmetrizable Kac-Moody Lie algebras. 
In particular, when the Lie algebra is taken to be an affine Lie algebra, 
the corresponding quantum group is called a \emph{quantum affine algebra}. 
The study of finite-dimensional representations of quantum affine algebras 
has been an active area of research since the 1990s, 
owing to its significant applications not only in mathematics but also in physics, 
particularly in the theory of integrable systems and solvable lattice models
(see, for example, \cite{Baxter, FR92, DFJMN}). 

Let us briefly recall the basics of quantum affine algebras and 
their representation theory. 
Let $q$ be an indeterminate and $\Bbbk$ be the algebraic closure of $\Q(q)$. 
Let $\mathfrak{g}$ be a finite-dimensional simple Lie algebra over $\C$, 
$I$ be the set of vertices of its Dynkin diagram, 
and $\Bbbk$-algebra $U_q(\hat{\mathfrak{g}})$ be the corresponding 
quantum affine algebra. 
Let $\mathscr{C}_{\mathfrak{g}}$ be the monoidal category 
of finite-dimensional representations of $U_q(\hat{\mathfrak{g}})$. 
We note that $\mathscr{C}_{\mathfrak{g}}$ is regarded as 
the monoidal category of finite-dimensional 
representations of the \emph{quantum loop algebra} $U_q(L\mathfrak{g})$. 
The simple objects in $\mathscr{C}_{\mathfrak{g}}$ are parametrized by 
monomials in indeterminates $Y_{i,a}$ for $i \in I$ and $a \in \Bbbk^{\times}$
\cite{CP95}. 
The simple module parametrized by a monomial $m$ is denoted by $L(m)$. 
For the simple modules $L(m_1), L(m_2)$, we say they commute strongly if 
$L(m_1)\otimes L(m_2)$ is simple. 
A simple module $L(m)$ is called \emph{real} if $L(m) \otimes L(m)$ is also simple, 
and called \emph{imaginary} otherwise. 

Determining whether a simple module is real or imaginary is an important problem. 
For example, from the representation–theoretic point of view, if a simple module 
$M$ is real, then for any simple module $N$, 
$M\otimes N$ has a simple head and a simple socle 
(\cite[Conjecture 3]{Le03}, \cite{headsimplicity}), and moreover, 
the irreducibility of $M\otimes N$ can be reduced to the computation of 
the invariant $\textfrak{d}$ \cite{KKOP20}. 
By these rich properties of real modules, 
knowing that a simple module is real brings significant advantages.
Moreover, the fact that $M$ is real can also be characterized by saying that 
the \emph{renormalized $R$-matrix} on $M \otimes M$ is a scalar multiple of the identity
\cite{headsimplicity}, 
so this is also a meaningful problem from the viewpoint of integrable systems.
In addition, due to the deep connection between 
the problem of determining real and imaginary modules 
and the \emph{monoidal categorification of cluster algebras} explained below, 
this problem is now regarded as one of the main topics 
in the representation theory of quantum affine algebras.


The study of the monoidal category $\mathscr{C}_{\mathfrak{g}}$ has made 
remarkable progress due to
the theory of 
monoidal categorification of cluster algebras introduced by 
Hernandez and Leclerc \cite{HL10}. 
A monoidal full subcategory $\mathcal{C} \subset \mathscr{C}_{\mathfrak{g}}$ 
is called a monoidal categorification of a cluster algebra $\mathscr{A}$ 
if there exists an algebra isomorphism $\iota$ between the Grothendieck ring $K(\mathcal{C})$ 
and $\mathscr{A}$ 
such that $\iota^{-1}$ sends cluster monomials to real simple modules.  
(This definition is weaker than the original one and follows \cite{KKOP20}. 
See \S\ref{sec : monoidal categorification}. )

In \cite{HL10}, Hernandez and Leclerc introduced a sequence of monoidal full subcategories
\[
\mathscr{C}_1\subset \mathscr{C}_2\subset \cdots \subset \mathscr{C}_l\subset\cdots
\]
of $\mathscr{C}_{\mathfrak{g}}$,
and proved that $\mathscr{C}_1$ provides a monoidal categorification of a cluster algebra of 
finite type 
when $\mathfrak{g}$ is of type $A_n$ and $D_4$. 
In \cite{HL16}, 
it was shown that there is a natural ring isomorphism 
$\iota_l : K(\mathscr{C}_l)\to \mathscr{A}_l$ for each $l\in \Z_{\geq 1}$ and 
certain cluster algebra $\mathscr{A}_l$. 
In most cases, the cluster algebras $\mathscr{A}_l$ are of infinite type. 
In \cite{Qi17}, Qin showed that when $\mathfrak{g}$ is simply-laced, 
each $\mathscr{C}_l$ provides a monoidal categorification of $\mathscr{A}_l$ by $\iota_l$. 

When the categorified cluster algebra is of finite type, 
every simple object in the category corresponds to a cluster monomial, 
since the cluster monomials form a basis of a finite type cluster algebra. 
However, in general, cluster monomials do not span the entire cluster algebra. 
Thus, it becomes a fundamental problem to determine which simple objects arise from cluster monomials.
It is believed that the notions of real and imaginary modules play a key role in this problem.
Namely, the following important conjecture is widely expected to hold$\colon$
\begin{Conj}[{\cite[Conjecture 5.7]{HL21}}, {\cite[Conjecture 5.2]{HL16}}]
\label{Conj : original real vs monomial intro}
    The isomorphism $\iota_l : K(\mathscr{C}_l)\to\mathscr{A}_l$ induces a bijection 
    between the set of real simple modules and the set of cluster monomials, i.e., 
    \[
    \{ \iota_l([L]) \mid L  \text{ is a real simple module in $\mathscr{C}_l$}\} 
    = \{ \text{cluster monomial in $\mathscr{A}_l$}\}. 
    \]
\end{Conj}

Since then, the theory of monoidal categorification has been extensively developed.
In particular, in \cite{KKOP24}, for any type $\mathfrak{g}$, 
infinitely many monoidal subcategories giving monoidal categorifications 
of cluster algebras were constructed.
More precisely, for each $\mathfrak{g}$, each pair of integers $-\infty < a \le b < \infty$,
and each \emph{admissible sequence} $\mathfrak{s}$ 
(see \S \ref{sec : monoidal categorification}), 
a monoidal category $\mathscr{C}_{\mathfrak{g}}^{[a,b],\mathfrak{s}}$ was defined 
and shown to provide a monoidal categorification of a cluster algebra.
In particular, if $n$ denotes the number of Dynkin nodes,
then by choosing $\mathfrak{s}$ appropriately,
$\mathscr{C}_{\mathfrak{g}}^{[1, (l+1)n], \mathfrak{s}}$ coincides with $\mathscr{C}_l$ 
for each $l \in \Z_{\geq 1}$. 

Since \cite{KKOP24} has significantly expanded the family of monoidal categorifications 
beyond $\mathscr{C}_l$, 
Conjecture~\ref{Conj : original real vs monomial intro} can be generalized as follows:
\begin{Conj}
\label{Conj : real vs monomial intro}
    Let $\mathcal{C}$ be a monoidal categorification of a cluster algebra $\mathscr{A}$ 
    and let $\iota : K(\mathcal{C}) \overset{\sim}{\to} \mathscr{A}$ be the corresponding 
    isomorphism. 
    Then, there is a bijection between the set of real simple modules 
    and the set of cluster monomials, i.e., 
    \[
    \{ \iota([L]) \mid L  \text{ is a real simple module in $\mathcal{C}$}\} 
    = \{ \text{cluster monomial in $\mathscr{A}$}\}. 
    \]
\end{Conj}


By the definition of monoidal categorification, 
it is clear that the modules corresponding to cluster monomials are real. 
However, the converse is wide open. 
When the cluster algebra is of finite type, 
the conjecture holds because the cluster monomials form a basis. 
In particular, every simple module in $\mathcal{C}$ is a real module.
However, the case in which a monoidal category $\mathcal{C}$ contains 
imaginary modules remains fully open. 

In \cite{HL10}, the combinatorial properties of cluster algebras of 
finite type, especially \emph{cluster expansion},  
played a crucial role in analyzing $\mathscr{C}_1$. 
For cluster algebras of finite type, 
the denominator vector (or $\mathbf{d}$-vector) 
gives a bijection between cluster variables and almost positive roots 
(roots that are either positive or the negatives of simple roots)\cite{FZ03}. 
Let $\Phi_{\geq -1}$ denote the set of almost positive roots and 
$x[\alpha]$ denote the cluster variable whose $\mathbf{d}$-vector is $\alpha \in \Phi_{\geq -1}$. 
For $\alpha, \beta\in \Phi_{\geq -1}$, 
we call them \emph{compatible} if $x[\alpha]$ and $x[\beta]$ are in the same cluster. 
Then, for any $\gamma$ in the root lattice, 
there is a unique expression 
\[
\gamma = \sum_{\alpha\in \Phi_{\geq -1}}m_{\alpha}\alpha
\]
where the $m_{\alpha}$ are nonnegative integers with 
$m_{\alpha}m_{\beta}=0$ whenever $\alpha$ and $\beta$ are not compatible
\cite{expansion}. 
This property is called cluster expansion. 

In the case of general cluster algebras, 
combinatorial properties are not as well understood as those in finite types. 
This is considered to be one of the major reasons why 
the conjecture~\ref{Conj : real vs monomial intro} remains unsolved in general. 
However, in recent years, combinatorial structures have also been developed 
for affine type cluster algebras by Reading and Stella in \cite{affine}. 
We briefly recall their results below. 
As an analog of almost positive roots of finite type, 
they introduced the set $\Phi_c^{\text{re}}$ for affine type 
and showed that the $\mathbf{d}$-vector gives a bijection between cluster variables and 
$\Phi_c^{\text{re}}$.
However, in this setting the cluster expansion does not hold: 
the $\mathbf{d}$-vector gives an injection from cluster monomials 
to the root lattice, but not a surjection.
To remedy this, although there is no cluster variable 
whose $\mathbf{d}$-vector equals the null root $\delta$, 
they formally considered the set
$\Phi_c\coloneq \Phi_c^{\text{re}}\cup \{\delta\}$
and by endowing $\Phi_c$ with a combinatorial relation called c-compatibility, 
they made it possible to establish a cluster expansion in the affine case. 
We call the cluster expansion without using $\delta$ the real cluster expansion,
and the one involving $\delta$ the imaginary cluster expansion. 
The real cluster expansion is characterized by the $\mathbf{d}$-vectors 
of cluster monomials,
whereas the relationship between the imaginary cluster expansion 
and the cluster algebra
is not yet understood. 
In \cite[\S 5]{affine}, Reading and Stella remark that 
\begin{quote}
    we expect that $\delta$ is the denominator vector of an important element 
    of the cluster algebra. 
\end{quote}

Our strategy is based on the idea that if we can find a monoidal categorification 
of an affine cluster algebra, 
the combinatorics of the affine cluster algebras would allow us 
to analyze the category in depth and settle the conjecture~\ref{Conj : real vs monomial intro} 
in the categories. 
Furthermore, since there exist ``imaginary'' roots in affine root systems, 
we expected that ``imaginary'' modules should appear in the categorified setting 
as counterparts to these roots. 
In this paper, we prove these expectations.

First, by choosing admissible sequences and intervals appropriately, 
we show that one can construct full subcategories of 
$\mathscr{C}_{\mathfrak{g}}$ that provide 
categorifications of cluster algebras of affine type. 
\begin{Thm}[Theorem \ref{Thm : affine categorification}]
\label{Thm : affine categorification intro}
    Let $\mathfrak{g}$ be a finite-dimensional simple Lie algebra 
    of type $A_n(n\geq2), D_n(n\geq4), E_n(n=6,7,8)$. 
    There exists an admissible sequence $\mathfrak{s}^1$ for 
    $\mathfrak{g}$ of type $A_n(n\geq2), D_n(n\geq4), E_n(n=6,7,8)$ 
    and $\mathfrak{s}^2$ for $\mathfrak{g}$ of type $D_4$ 
    such that the following categories provide monoidal categorifications of 
    the following cluster algebras of affine types : 
    \begin{center}
    \begin{tabular}{c|c}
        category & cluster algebra type \\ \hline
         $\mathscr{C}_{D_n}^{[1, 2n+1], \mathfrak{s}^1}(n\geq 4)$& $D_{n}^{(1)}$\\
         $\mathscr{C}_{D_4}^{[1, 11], \mathfrak{s}^2}$& $E_{6}^{(1)}$\\
         $\mathscr{C}_{E_n}^{[1, 2n+1], \mathfrak{s}^1}(n=6, 7, 8)$& $E_{n}^{(1)}$\\
         $\mathscr{C}_{A_2}^{[1, 11], \mathfrak{s}^1}$& $E_{8}^{(1)}$\\
         $\mathscr{C}_{A_3}^{[1, 10], \mathfrak{s}^1}$& $E_{6}^{(1)}$\\
         $\mathscr{C}_{A_4}^{[1, 13], \mathfrak{s}^1}$& $E_{8}^{(1)}$\\
         $\mathscr{C}_{A_n}^{[1, 2n+2], \mathfrak{s}^1}(n\geq5)$& $D_{n+1}^{(1)}$\\
    \end{tabular}
    \end{center}
\end{Thm}

Next, for the monoidal categories $\mathcal{C}$ appearing in 
Theorem~\ref{Thm : affine categorification intro} 
except for the cases where $\mathfrak{g}$ is of type $A_2$, $A_3$, and $A_4$, 
we determine whether each simple object in $\mathcal{C}$ is real or imaginary.

We show that, under the identification of $K(\mathcal{C})$ with 
a cluster algebra of affine type via the monoidal categorification, 
the $\mathbf{d}$-vector of each simple module can be described 
combinatorially in an explicit way 
(Theorem ~\ref{Thm : D_n G=dL}, \ref{Thm : D_4 G=dL}, \ref{Thm : E_n G=dL}, 
\ref{Thm : A_n G=dL}). 
Let $\mathcal{M}$ be the set of monomials that parametrizes the simple modules in 
$\mathcal{C}$. 
Let $F_1, \ldots, F_n$ be the monomials in $\mathcal{M}$ such that 
$L(F_i)$ is the simple module corresponding to the frozen variable $f_i$. 
Let $\mathcal{M}^{\prime}$ be the subset of $\mathcal{M}$ 
consisting of the monomials that are not divisible by any of the $F_i$. 
Then we prove that the $\mathbf{d}$-vector provides a bijection between 
the set $\{L(m) \mid m \in \mathcal{M}^{\prime}\}$ and the root lattice. 

Under this bijection, 
we write the simple module corresponding to an element $\gamma$ of the root lattice 
by $L(\gamma)$. 
Then, for the null-root $\delta$, 
we prove that the module $L(\delta)$ is a prime imaginary module
(Proposition \ref{Prop : D_n delta imaginary}, \ref{Prop : D_4 delta imaginary}, 
\ref{Prop : E_n delta imaginary}, \ref{Prop : A_n delta imaginary}). 
%
In other words, the prime imaginary module $L(\delta)$ in $K(\mathcal{C})$ provides 
one possible answer to Reading and Stella's question 
of what the important element is that has $\delta$ as a 
$\mathbf{d}$-vector.

Furthermore, we prove that if $\alpha, \beta \in \Phi^c$ are $c$-compatible,
then $L(\alpha)$ and $L(\beta)$ strongly commute. 
Using the affine-type cluster expansion, we then obtain the following result : 
\begin{Thm}[Corollary \ref{Cor : D_n using delta}, \ref{Cor : D_4 using delta}, 
\ref{Cor : E_n using delta}, \ref{Cor : A_n using delta}]
\label{Thm : using delta intro}
For a simple module $L(m)\in \mathcal{C}$, 
it is a real module if and only if the $\mathbf{d}$-vector of $L(m)$ has a 
real cluster expansion. 
\end{Thm}
Since the $\mathbf{d}$-vector of $L(m)$ can be described explicitly 
as above, 
this theorem provides a criterion to determine whether $L(m)$ is real or imaginary
for each $m \in \mathcal{M}$.
As a result, we obtain the following corollary : 
\begin{Cor}[Corollary \ref{Cor : D_n cluster monomial}, 
\ref{Cor : D_4 cluster monomial}, 
\ref{Cor : E_n clyster monomial}, 
\ref{Cor : A_n clyster monomial}]
\label{Cor : cluster monomial intro}
    For these cases, Conjecture ~\ref{Conj : real vs monomial intro} is true. 
\end{Cor}
As far as the author knows, this is the first example
in which the conjecture~\ref{Conj : real vs monomial intro} 
is proved within a monoidal categorification
whose module category contains imaginary modules. 

This paper is organized as follows : 
In Section \ref{Sec : qloop}, we review the basic facts about quantum loop algebras and recall the invariant $\fd$ introduced in \cite{KKOP20}.
In Section \ref{sec : monoidal categorification}, we summarize known results concerning the monoidal categorification of cluster algebras and prove Theorem~\ref{Thm : affine categorification intro}.
In Section \ref{Sec : affinecluster}, we recall the results of \cite{affine} on cluster algebras of affine type.
In Section \ref{sec : classification}, we prove 
Theorem~\ref{Thm : using delta intro} and 
Corollary~\ref{Cor : cluster monomial intro} for each type.

The monoidal categories treated up to this point are constructed from specific choices of admissible sequences.
By employing the theory of dual canonical bases, we can identify, 
for each imaginary module obtained so far, 
a corresponding imaginary module arising from monoidal categories 
associated with different admissible sequences.
This enables us to discover new imaginary modules.
In Section \ref{Sec : Qdatum}, we illustrate this phenomenon in the case of type $D_4$, 
and also discuss the relationship between the imaginary vectors 
discovered by Leclerc in \cite{Le03} and the imaginary modules studied 
in the present paper.

\subsection*{Acknowledgement}
I would like to express my sincere gratitude to my supervisor, 
Noriyuki Abe, for many valuable discussions and his continuous support. 
I am also deeply grateful to Hironori Oya for numerous helpful comments, 
including his advice on the discussion in Section ~\ref{Sec : Qdatum}.
This work was supported by the FoPM program at the University of Tokyo.

\section{Quantum loop algebras and their finite-dimensional representations}
\label{Sec : qloop}
\subsection{Quantum loop algebras}
Let $\fg$ be a complex finite-dimensional simple Lie algebra 
with the set of Dynkin indices $I = \{1, \ldots , n\}$
and Cartan matrix $C = (c_{ij})_{i, j \in I}$. 
Let $r$ be the maximum number of edges in the Dynkin diagram 
and $D = \mathrm{diag}(d_i \mid i \in I)$ be the unique diagonal matrix 
such that $d_{i} \in \{1,r\}$ for each $i \in I$ 
and $DC = (d_i c_{ij})_{i,j \in I}$ is symmetric. 

Let $q$ be an indeterminate and $\kk$
the algebraic closure of the rational function field $\Q(q)$. 
Set $q_i \coloneq q^{d_i}$. 
For integers $k \geq l \geq 0$, we set 
\[
[ k ]_{q_i} \coloneq \frac{q_i^k - q_i^{-k}}{q_i - q_i^{-1}}, \quad
[k]_{q_i}! \coloneq [ k ]_{q_i}[ k - 1 ]_{q_i}\cdots [ 2 ]_{q_i}[ 1 ]_{q_i}, \quad
\left[ \begin{matrix} k \\ l \end{matrix} \right]_{q_i} 
\coloneq \frac{[k]_{q_i}!}{[k-l]_{q_i}! [l]_{q_i}!}.
\]

\begin{Def}
\emph{The quantum loop algebra} $U_{q}(L\fg)$
is the $\kk$-algebra given by the following generators and relations : 

generators : \quad
$\quad\quad x_{i, k}^{\pm} \ (i \in I, k \in \Z), \quad
k_i^{\pm 1} \ (i \in I), \quad
h_{i, l} \ (i \in I, l \in \Z \setminus \{0\})$

relations : 
\begin{itemize}
    \item  $k_{i} k_{i}^{-1} = k_{i}^{-1} k_{i} = 1, 
    k_{i} k_{j} = k_{j} k_{i}$ for $i,j \in I$,
    \item $k_{i} x_{j, k}^{\pm} k_{i}^{-1} = q_{i}^{\pm c_{ij}} x_{j,k}^{\pm}$ 
    for $i,j \in I$ and $k \in \Z$,
    \item $[k_{i}, h_{j,l}] = [h_{i,l}, h_{j,m}] = 0$ 
    for $i,j \in I$ and $l,m \in \Z \setminus\{0\}$,
    \item $[h_{i, l}, x_{j, m}^{\pm}] = \pm \frac{1}{l}[l c_{ij}]_{q_i}x_{j, l + m}^{\pm}$
     for $i, j \in I, l \in \Z \setminus \{0\}$ and $m \in \Z$, 
    \item $x_{i, k+1}^{\pm} x_{j, l}^{\pm} - q_i^{\pm c_{ij}}x_{j, l}^{\pm} x_{i, k+1}^{\pm}
     = q_i^{\pm c_{ij}} x_{i, k}^{\pm}x_{j, l + 1}^{\pm} - x_{j, l + 1}^{\pm}x_{i, k}^{\pm}
    $ for $i, j \in I$ and $k, l \in \Z$, 
    \item $[x_{i, k}^{+}, x_{j, l}^{-}] 
    = \displaystyle \delta_{i, j}\frac{\phi_{i, k + l}^{+} - \phi_{i, k + l}^{-}}{q_i - q_i^{-1}}$
     for $i, j \in I$ and $k, l \in \Z$,
     \item $\displaystyle \sum_{\sigma \in \mathfrak{S}_{b}} \sum_{k=0}^{b}
(-1)^{k} \left[\begin{matrix} b \\ k\end{matrix} \right]_{q_i}
x_{i, r_{\sigma(1)}}^{\pm} \cdots x_{i, r_{\sigma(k)}}^{\pm} 
x_{j, s}^{\pm} x_{i, r_{\sigma(k + 1)}}^{\pm} \cdots x_{i, r_{\sigma(b)}}^{\pm} = 0
$, where $b = 1-c_{ij}$, for $i, j \in I$ with $i \neq j$ and $s, r_1, \ldots, r_b \in \Z$, 

\end{itemize}
where $\displaystyle\phi_{i}^{\pm}(z) \coloneq 
 \sum_{r = 0}^{\infty}  \phi_{i, r}^{\pm}z^{\pm r} \coloneq 
 k_{i}^{\pm 1}
\exp \left( \pm (q_{i}-q_{i}^{-1})
\sum_{l=1}^{\infty}h_{i, \pm l} z^{\pm l} \right)$.

\subsection{Finite-dimensional representations. }

A finite-dimensional $U_{q}(L\fg)$-module is said to be of type $\mathbf{1}$
if the element $k_{i}$ acts as a diagonal $\kk$-linear operator whose eigenvalues 
belong to the set $\{ q^{k} \mid k \in \Z\}$ for each $i \in I$. 
It is well-known that 
the study of finite-dimensional representations of $U_q(L\fg)$ reduces essentially to the study of the category $\Cc_{\fg}$ of type $\mathbf{1}$ finite-dimensional $U_{q}(L\fg)$-modules.
It is also well known that $U_q(L\fg)$ is a subquotient of the corresponding 
quantum affine algebra and the study of type $\mathbf{1}$ finite-dimensional modules 
of quantum affine algebra reduces to the study of $\Cc_{\fg}$. 
Since $U_{q}(L\fg)$ has a Hopf algebra structure, 
$\Cc_{\fg}$ is a monoidal category. 

Let $V \in \Cc_{\fg}$. Since the elements $\{ k_{i}^{\pm 1}, h_{i, l} \mid i \in I, l \in \Z \setminus\{0\}\}$
mutually commute, we have a decomposition 
\[
V = \bigoplus_{\gamma \in (\kk [\![ z ]\!] \times \kk[\![ z^{-1} ]\!])^{I}} V_{\gamma} 
\]
where, for each $\gamma = (\gamma_{i}^{+}(z), \gamma_{i}^{-}(z))_{i \in I}$, we define
$V_{\gamma}$ as the subspace of $V$ on which each coefficient of the series 
$\phi_{i}^{\pm}(z) - \gamma_{i}^{\pm}(z)$ acts nilpotently for any $i \in I$.
If $V_{\gamma} \neq 0$, it is called \emph{an $\ell$-weight space} and $\gamma$ is called 
the corresponding \emph{$\ell$-weight}.

For a simple module $L$ in  $\Cc_{\fg}$, there uniquely exists an $\ell$-weight space
$L_{\gamma_0}$ such that $x^{+}_{i,k} L_{\gamma_0}=0$ for all
$i \in I$ and $k \in \Z$. 
Such $\ell$-weight $\gamma_0$ is called \emph{the $\ell$-highest weight} of $L$ 
and we can classify the isomorphism classes of simple objects in $\Cc_{\fg}$
 by using the $\ell$-highest weights : 
 
\begin{Thm}[{\cite[Theorem 3.3]{CP95}, \cite[Theorem 12.2.6]{CP94}}] \label{Thm:CP}
If $\gamma= (\gamma_{i}^{+}(z), \gamma_{i}^{-}(z))_{i \in I}$ is an $\ell$-highest weight 
of a simple $U_{q}(L\fg)$-module in $\mathscr{C}_{\mathfrak{g}}$, for each $i\in I$ there exists a unique polynomial $P_{i}(z) \in 1+z\kk[z]$ for $i \in I$ such that
\begin{equation} \label{eq:ellhwt}
\gamma_{i}^{\pm}(z) = q_{i}^{\deg(P_i)} \frac{P_{i}(zq_{i}^{-1})}{P_{i}(zq_{i})}.
\end{equation}
Conversely, for any $(P_{i}(z))_{i \in I} \in (1+z\kk[z])^{I}$, we have a simple $U_{q}(L\fg)$-module 
in $\mathscr{C}_{\mathfrak{g}}$ whose $\ell$-highest weight $\gamma$ is given by {\rm (\ref{eq:ellhwt})}. 
\end{Thm} 
\end{Def}
This $I$-tuple of polynomials $(P_{i}(z))_{i \in I}$ associated with a simple module $L$ 
is called \emph{the Drinfeld polynomials of $L$}. 
A non-zero vector in the $\ell$-highest weight space $L_{\gamma}$ 
is called \emph{an $\ell$-highest weight vector of $L$}. Since $\dim_\kk L_\gamma=1$,
it is unique up to a scalar. 

Let $\Irr \Cc_{\fg}$ denote the classes of simple objects of $\Cc_{\fg}$. 
By theorem \ref{Thm:CP}, $\Irr \Cc_{\fg}$ is parametrized by $(1+z\kk[z])^{I}$. 

Let $Y_{i, a} (i \in I_0, a \in \kk^{\times})$ be an indeterminate and 
\[\mathcal{M}_{+} \coloneqq \biggl\{ m = \displaystyle\prod_{i \in I_0, a \in \kk^{\times}}
Y_{i, a}^{u_{i, a}}  \bigg |
\ u_{i, a} \in \Z_{\geq 0}, u_{i, a} = 0 \ \text{except for finitely many }(i, a)\biggr\}. \]
Then, there is a bijection 
\[\displaystyle\mathcal{M}_{+} \to (1 + z \Bbbk\lbrack z\rbrack)^I ; 
\prod_{i \in I_0, a \in \Bbbk^{\times}}
Y_{i, a}^{u_{i, a}} \mapsto (\prod_{a \in \Bbbk^{\times}}(1 - az)^{u_{i, a}})_{i \in I}. \] 
Using this bijection, $\Irr \Cc_{\fg}$ is also parametrized by $\mathcal{M}_{+}$, 
 and let $L(m)$ denote the simple module corresponding to the monomial $m$. 
 The monomial $m$ is called the \emph{highest weight monomial} 
 of the simple module $L(m)$. 
For $i \in I, a \in \kk^{\times}$ and $p \in \Z_{\geq 1}$, 
a simple module $L(\prod_{k = 0}^{p-1}Y_{i, aq_i^{2k}})$ is called 
a \emph{KR-module}, 
and in particular, $L(Y_{i, a})$ is called a \emph{fundamental module}.

 \begin{Prop}[{\cite[Corollary 3.5]{CP95}}] \label{Prop : subquot}
Let $m_1, \ldots, m_k \in \mathcal{M}_{+}$. 
Then, $L(m_1\cdots m_k)$ is isomorphic to a subquotient of $\bigotimes _{i=1}^{k}L(m_i)$. 
\end{Prop}

For each $a \in \kk$, we consider a $\kk$-algebra automorphism $\tau_a$ 
of $U_q(L_{\fg})$ given by 
\[
\tau_a (k_i) = k_i, \tau_a(x_{i, k}^{\pm}) = a^k x_{i, k}^{\pm}, 
\tau_a(h_{i, l}) = a^l h_{i, l}, 
\]
for $i \in I, k \in \Z, l \in \Z\setminus\{0\}$. 
This automorphism $\tau_a$ is a $q$-analog of the loop rotation $z \mapsto az$ for
the loop algebra $L\fg = \fg \otimes_\C \C[z^{\pm}]$.
For a finite-dimensional representation $V$ in $\mathscr{C}_{\mathfrak{g}}$, 
the twisted representation $\tau_a^*V$ by $\tau_a$ again belongs to $\Cc$.
Thus the assignment $V \mapsto \tau_a^*V$ defines a monoidal auto-equivalence $T_{a}$ of $\mathscr{C}_{\mathfrak{g}}$. 
If we denote by $m \mapsto m_a$ the monoid automorphism of $\mathcal{M}_{+}$ 
given by $Y_{i, b} \mapsto Y_{i, ab}$, 
we have $T_a(L(m)) \cong L(m_a)$ for any simple module $L(m) \in \mathscr{C_{\fg}}$.

As a monoidal category, 
 $\Cc_{\fg}$ is not braided i.e., 
 $V \otimes W$ is not isomorphic to $W \otimes V$ for general $V, W \in \Cc_{\fg}$. 
However, using the $q$-character, 
which is a type of character in the quantum loop algebra, we have 
the following fact. 

\begin{Prop}[{\cite[Theorem 3]{FR98}}] \label{q-ch}
Let $K(\Cc_{\fg})$ be the Grothendieck ring of the monoidal category $\Cc_{\fg}$. 
Then, there is an injective algebra homomorphism 
$\chi_{q} \colon K(\Cc_{\fg}) \to \Z[Y_{i,a}^{\pm 1} \mid i \in I, a \in \kk^\times ]$. 
In particular, although $\Cc_{\fg}$ is not braided, 
$K(\Cc_{\fg})$ is commutative. 
\end{Prop}

 For simple modules $L, M$ in $\Cc_\fg$, 
we say that $L$ and $M$ \emph{commute} if $L \otimes M \simeq M \otimes L$. 
Also, we say $L$ and $M$ \emph{strongly commute} if $L \otimes M$ is simple. 
If $L$ and $M$ strongly commute, 
then $[L \otimes M] = [M \otimes L ]$ in $K(\Cc_{\fg})$ and they are both simple, 
so $L$ and $M$ commute. 

We say that a simple module $L$ in $\Cc_\fg$ is \emph{real} if $L$ strongly commutes with itself, 
i.e., if $L \otimes L$ is simple. 
When a simple module $L$ is not real, then we say that $L$ is \emph{imaginary}. 
For example, it is known that KR-modules, in particular fundamental modules, 
are real modules
(see~e.g.~{\cite[Proposition 6.4]{FHOO22}}). 

We say that a simple module $L$ in $\Cc_\fg$ is \emph{prime} if 
there exists no nontrivial factorization $L \cong L_1\otimes L_2$ 
in $\mathscr{C}_{\fg}$.

\subsection{R-matrices and invariants}
Let $z$ be an indeterminate.
For any $U_{q}(L\fg)$-module $V \in \Cc_{\fg}$, we can consider its \emph{affinization} $V_{z}$.
This is the $\kk[z^{\pm 1}]$-module $V_{z} \coloneq V \otimes_{\kk} \kk[z^{\pm1}]$ endowed 
with the structure of 
a left $U_{q}(L\fg)$-module by
\[
k_{i}(v \otimes \varphi) = (k_{i}v) \otimes \varphi, \
x^{\pm}_{i, k} (v \otimes \varphi) = (x^{\pm}_{i,k}v)\otimes z^{k} \varphi, \
h_{i,l}(v \otimes \varphi) = (h_{i,l}v) \otimes z^{l}\varphi, 
\]
for $v \in V$ and $\varphi \in \kk[z^{\pm 1}]$.

Let $V, W \in \Cc_{\fg}$ be simple modules. We fix
their $\ell$-highest weight vectors
$v \in V, w \in W$.
It is known that there exists a unique isomorphism of $U_{q}(L\fg) \otimes_{\kk} \kk(z)$-modules
\[
R_{V, W}(z) \colon (V_z \otimes W) \otimes_{\kk[z^{\pm 1}]} \kk(z) \to (W \otimes V_z) \otimes_{\kk[z^{\pm 1}]} \kk(z) 
\]
satisfying $R_{V, W}(z) (v_z \otimes w) = w \otimes v_z$, where $v_z \coloneq v \otimes 1 \in V_z$ (see~\cite[Corollary 2.5]{AK97}). 
We call this isomorphism $R_{V,W}(z)$ \emph{the normalized R-matrix} between $V$ and $W$.
Let $d_{V, W}(z) \in \kk[z]$ denote the monic polynomial of the smallest degree
such that the image of $d_{V, W}(z) R_{V, W}(z)$ is contained in 
$W \otimes V_z \subset (W \otimes V_z) \otimes_{\kk[z^{\pm 1}]} \kk(z) $.
This polynomial $d_{V,W}(z)$ is called  \emph{the denominator} of $R_{V,W}(z)$.
Note that $d_{V,W}(z)$ does not depend on the choice of $\ell$-highest weight vectors $v$ and $w$.
%


\begin{Def}[cf.~{\cite[Section~3]{KKOP20}}] \label{Def:fd}
For each pair $(V, W)$
of simple modules in $\Cc_{\fg}$, we define the $\Z_{\geq 0}$-valued invariant $\fd(V, W)$ by  
\[
\fd(V, W) \coloneq \zero_{z=1}(d_{V, W}(z)) + \zero_{z=1}(d_{W,V}(z)),
\]
where $\zero_{z=a}(f(z))$ denotes the order of zeros of $f(z)$ at $z=a$
for $f(z) \in \kk[z]$ and $a \in \kk$. 
By the definition, we have $\fd(V, W) = \fd(W,V)$.
\end{Def} 

The invariant $\fd$ is related to the notion of strongly commuting.
\begin{Prop}[{\cite[Corollary 3.17]{KKOP20}}]\label{Prop : real delta}
Let $V$ and $W$ be simple modules in $\Cc_{\fg}$. 
Assume that one of them is real. 
Then $V$ and $W$ strongly commute if and only if $\fd(V, W) = 0$. 
\end{Prop}


\begin{Prop}[{\cite[Proposition 4.2]{KKOP20}}]\label{Prop : delta ineq sum}
    Let $L, M$ and $N$ be simple modules. 
    Then we have 
    \[
    \fd(S, L) \leq \fd(M, L) + \fd(N, L)
    \]
    for any simple subquotient $S$ of $M\otimes N$. 
\end{Prop}
\begin{Prop}[{\cite{H10, headsimplicity}}]\label{Prop : commuting family}
    Let $\{L_1, \ldots, L_n\}$ be a pairwise strongly commuting family of 
    real simple modules. 
    Then, $\bigotimes_{i=1}^n L_i^{\otimes m_i}$ is a simple module for 
    any $m_i \in \Z_{\geq 0}$. 
\end{Prop}
\begin{Prop}[{\cite[Theorem 4.11]{KKOP20}}]\label{Prop : delta inequality real}
    Let $V$ be a real module, $W$ be a simple module, 
    and $S$ be a simple subquotient of $V \otimes W$. 
    If $\fd(V, W) > 0$, 
    then $\fd(V, S) < \fd(V, W)$. 
\end{Prop}


\section{Monoidal categorification of cluster algebras}
\label{sec : monoidal categorification}
\subsection{Cluster algebras}\label{ssec : cluster}
In this subsection, we briefly recall the definition of cluster algebras. 
\begin{Def}[{\cite[Definition 7.1]{KKOP20}}]\label{Def : seed}
Let $K = K^{\text{ex}} \sqcup K^{\text{fr}}$ be a countable set 
which decomposes into subsets $K^{\text{ex}}$ of exchangeable indices 
and a subset $K^{\text{fr}}$ of frozen indices. 

    For a commutative ring $A$, 
    we say that a pair 
    $\mathcal{S} = 
    (\{x_i\in A\mid i \in K\}, \tilde{B}= (b_{ij})_{(i, j) \in K \times K^{\text{ex}}})$ 
    is a \emph{seed} in $A$ 
    if it satisfies the following conditions$\colon$ 
    \begin{itemize}
        \item $\{x_i\}_{i \in K}$ are algebraically independent over $\Z$ in $A$, 
        \item $\tilde{B} = (b_{ij})_{(i, j) \in K \times K^{\text{ex}}}$ is an integer-valued
 matrix and its principal part $B \coloneq (b_{ij})_{i, j \in K^{\text{ex}}}$ is skew-symmetric, 
 \item for each $j \in K^{\text{ex}}$, there exist finitely many $i \in K$
      such that $b_{ij} \neq 0$.
    \end{itemize}
    
    For a seed $\mathcal{S}$, we call 
    the matrix $\tilde{B}$ the \emph{exchange matrix} of $\mathcal{S}$, 
    the set $\{x_i\}_{i \in K}$ the \emph{cluster} of 
    $\mathcal{S}$ and its elements the \emph{cluster variables}. 
    An element of the form $x^{\bf{a}}$ for ${\bf{a}} \in \Z_{\geq 0}^{\oplus K}$ 
    is called a \emph{cluster monomial}, where 
    \[x^{\bf{c}} \coloneq \prod_{i \in K}x_i^{c_i} \quad \text{ for }  
    {\bf{c}} = (c_i)_{i \in K} \in \Z_{\geq 0}^{\oplus K}. \]
\end{Def}

Let $\mathcal{S} = (\{x_i\}_{i \in K}, \tilde{B})$ be a seed 
in a field $\mathfrak{K}$
 of characteristic $0$. 
 For each $k \in K^{\text{ex}}$, we define$\colon$
 \begin{itemize}
     \item $\mu_k(\tilde{B})_{ij} =
\begin{cases}
  -b_{ij} & \text{if}  \ i=k \ \text{or} \ j=k, \\
  b_{ij} + (-1)^{\delta(b_{ik} < 0)} \max(b_{ik} b_{kj}, 0) & \text{otherwise,}
\end{cases}
$
\item $\mu_k (x)_i  =
\begin{cases}
\displaystyle\frac{\prod_{j, b_{jk} > 0}x_j^{b_{jk}} 
+ \prod_{j, b_{jk} < 0}x_j^{-b_{jk}}}{x_k}, & \text{if} \ i=k, \\
x_i & \text{if} \ i\neq k. 
\end{cases}
$
 \end{itemize}
Then, the pair $\mu_k (\mathcal{S}) \coloneq 
(\{\mu_k (x)_i\}_{i \in K}, \mu_k (\tilde{B}))$
 is a new seed in $\mathfrak{K}$ and we call it the \emph{mutation} of $\mathcal{S}$ at $k$. 

The \emph{cluster algebra $\mathscr{A}(\mathcal{S})$ associated with a seed $\mathcal{S}$}
is the $\Z$-subalgebra of the field $\mathfrak{K}$ 
generated by all the cluster variables in the seeds obtained from $\mathcal{S}$
 by all possible successive mutations. 
 For a cluster algebra $\mathscr{A}(\mathcal{S})$, 
 the seed $\mathcal{S}$ is called the \emph{initial seed}.

\begin{Prop}[{\cite[Theorem 3.1]{FZ02}}]\label{Prop : Laurent}
    Let $\mathscr{A}(\mathcal{S})$ be a cluster algebra with 
    an initial seed 
    $\mathcal{S}
    = \bigl(\{x_i\}_{i \in K^{\mathrm{ex}}} \cup \{f_j\}_{j \in K^{\mathrm{fr}}}, \tilde{B}\bigr)$. 
    
    Then, $\mathscr{A}(\mathcal{S})$ is a subalgebra of 
    $\Z[x_i^{\pm 1}, f_j \mid i \in K^{\mathrm{ex}}, j \in K^{\mathrm{fr}}]$. 
\end{Prop}
 \begin{Prop}[{\cite[Theorem 3.1]{GLS13}}]\label{Prop : cl.var is irr}
     Let $\mathscr{A}$ be a cluster algebra. 
     Then, any cluster variable is an irreducible element in $\mathscr{A}$. 
 \end{Prop}

In the above definition, we can use a quiver instead of the matrix $\tilde{B}$. 
Let $\mathcal{Q}$ be a quiver satisfying the following conditions : 
\begin{enumerate}
    \item The set of vertices of $\mathcal{Q}$ is 
    the countable set $K = K^{\text{ex}} \sqcup K^{\text{fr}}$, 
    \item the quiver $\mathcal{Q}$ does not have any loops nor any $2$-cycles, 
    \item for each exchangeable vertex $v$ of $\mathcal{Q}$, 
    the number of arrows incident with $v$ is finite, 
    \item for each pair of frozen vertices of $\mathcal{Q}$, 
    there are no edges between them. 
\end{enumerate}
For a quiver $\mathcal{Q}$ satisfying the above conditions, 
let $\tilde{B}_{\mathcal{Q}} = (b_{ij})_{(i, j)\in K\times K^{\mathrm{ex}}}$ denote the matrix 
defined by 
\[
b_{ij} \coloneq (\text{the number of arrows from $i$ to $j$}) 
- (\text{the number of arrows from $j$ to $i$}). 
\]
Then, there is a bijection between the set of quivers satisfying the above conditions and 
the set of matrices satisfying the conditions of exchange matrices,  
via the correspondence $\mathcal{Q} \mapsto \tilde{B}_{\mathcal{Q}}$. 
By this bijection, we sometimes identify exchange matrices with quivers, 
and we call the quivers exchange quivers. 
A cluster algebra $\mathscr{A}$ is said to be of finite type if, for some seed, 
the full subquiver of its exchange quiver consisting of exchangeable nodes 
is a Dynkin quiver of symmetric finite type. 
Also, a cluster algebra $\mathscr{A}$ is said to be of affine type if, for some seed, 
the full subquiver of its exchange quiver consisting of exchangeable nodes 
is a Dynkin quiver of symmetric affine type. 

\subsection{Monoidal categorification}\label{ssec : categorification}
Let $\mathcal{C}$ be a 
full subcategory of $\Cc_{\fg}$ containing the trivial module and 
closed under taking tensor products, 
subquotients, and extensions. 

\begin{Def}[{\cite[Definition 7.2]{KKOP24}}]\label{Def : monoidal seed}
A \emph{monoidal seed} in $\mathcal{C}$ is a quadruple 
$\mathscr{S} = (\{M_i\}_{i \in K}, \tilde{B} ; K, K^{\text{ex}})$ consisting of a 
countable set $K$,  a subset $K^{\text{ex}} \subset K$, 
a pairwise strongly commuting family $\{M_i\}_{i \in K}$ of real simple modules in $\mathcal{C}$ 
and an integer-valued matrix $\tilde{B} = (b_{ij})_{(i, j) \in K \times K^{\text{ex}}}$
satisfying the conditions of an exchange matrix. 

For $i \in K$, we call $M_i$ the $i$-th \emph{cluster variable module} of $\mathscr{S}$. 
\end{Def}


\begin{Def}[{\cite[Definition 6.2, 6.3, 6.5, Proposition 6.4]{KKOP20}}, {\cite[Definition 7.6]{KKOP24}}]\label{Def : admissible}
    Let $\mathscr{S} = (\{M_i\}_{i \in K}, \tilde{B} ; K, K^{\text{ex}})$ be a monoidal seed. 
    We call $\mathscr{S}$ \emph{admissible} if it satisfies the following : 
    \begin{itemize}
        \item For each $k \in K^{\text{ex}}$, there exists a simple module $M_k^{\prime}$ with
        an exact sequence 
        \[
        0 \to \bigotimes_{b_{ik} > 0}M_i^{\otimes b_{ik}}
        \to M_k \otimes M_k^{\prime} 
        \to \bigotimes_{b_{ik} < 0}M_i^{\otimes (-b_{ik})} \to 0, 
        \]
        \item For each $k \in K^{\text{ex}}$, the quadruple 
        \[
        \mu_k(\mathscr{S}) \coloneq 
        (\{M_i\}_{i \in K \setminus \{k\}} \cup \{M_k^{\prime}\}, \mu_k(\tilde{B}) ; K, K^{\text{ex}})
        \]
        is a monoidal seed in $\mathcal{C}$, i.e., 
        $M_k^{\prime}$ is real and strongly commutes with every $M_i$ for $i \neq k$, 
    \end{itemize}

    If $\mathscr{S}$ is admissible, we say that the mutation $\mu_k(\mathscr{S})$ 
    of $\mathscr{S}$ at $k \in K^{\text{ex}}$ is a \emph{mutation}. 
    We say that a monoidal seed $\mathscr{S}$ is \emph{completely admissible} 
    if $\mathscr{S}$ admits successive mutations in all possible directions. 
\end{Def}
For a monoidal seed $\mathscr{S}$, we define
\[
[ \mathscr{S}] \coloneq (\{[M_i]\}_{i \in K}, \tilde{B}), 
\]
where $[M_i]$ is an isomorphism class of $M_i$ in $K(\mathcal{C})$. 

\begin{Def}[{\cite[Definition 6.7]{KKOP20}}]\label{Def : categorification}
    Let $\mathscr{A}$ be a cluster algebra and $\mathcal{C}$ be 
    a full subcategory of $\Cc_{\fg}$ containing the trivial module and 
    stable under taking tensor products, subquotients, and extensions. 
    $\mathcal{C}$ is called a \emph{monoidal categorification} 
    of $\mathscr{A}$ if the following holds : 
    \begin{itemize}
        \item There exists an isomorphism 
        $\iota : K(\mathcal{C}) \overset{\sim}{\to} \mathscr{A}$, 
        
        \item there exists a completely admissible monoidal seed 
        $\mathscr{S} = (\{M_i\}_{i \in K}, \tilde{B} ; K, K^{\text{ex}})$ in $\mathcal{C}$
        such that 
        $\iota ([\mathscr{S}]) = (\{\iota ([M_i])\}_{i \in K}, \tilde{B} ; K, K^{\text{ex}})$
        is an initial seed of $\mathscr{A}$. 
    \end{itemize}
\end{Def}
If $\mathcal{C}$ is a monoidal categorification of a cluster algebra $\mathscr{A}$ 
and $\iota : K(\mathcal{C}) \overset{\sim}{\to} \mathscr{A}$ is the isomorphism, 
then the inverse images of cluster monomials are real modules by the strongly commutative condition
of monoidal seeds. 
In what follows, when discussing the setting of monoidal categorifications,
we sometimes identify 
$K(\mathcal{C})$ with $\mathscr{A}$ via $\iota$,
and regard elements of $\mathscr{A}$ as elements of $K(\mathcal{C})$.

\begin{Rem}
    The original definition of monoidal categorification in \cite{HL10} is stronger than 
    this definition. See \cite[Introduction]{KKOP20}. 
    In \cite{KKOP20}, they introduce an invariant $\Lambda$ and 
    the notion of $\Lambda$-monoidal categorification. 
\end{Rem}


\subsection{Monoidal categories
$\mathscr{C}_{\fg}^0$ and 
$\mathscr{C}_{\fg}^{[a, b], \mathfrak{s}}$}\label{ssec : newcategory}
In this subsection, we explain the monoidal category 
$\mathscr{C}_{\fg}^0$ introduced in \cite{HL10} and 
$\mathscr{C}_{\fg}^{[a, b], \mathfrak{s}}$ 
introduced in \cite{KKOP24}. 
In this paper, we consider 
only the case of $\fg$ being simply-laced. 
In the general case, we refer the reader to 
\cite{KKO19} for $\mathscr{C}_{\fg}^0$ and to \cite{KKOP24} for 
$\mathscr{C}_{\fg}^{[a, b], \mathfrak{s}}$. 

In this subsection, we set $\fg$ to be a complex finite-dimensional simple Lie algebra of type ADE. 
We fix the label of Dynkin nodes of $A_n, D_n, E_n$ as follows$\colon$

\begin{tikzpicture}[every node/.style={inner sep=1.5pt}]
\node (0) at (-32,0){$A_n\ (n\geq1)$};
  \node (1) at (-30, 0){1};
  \node (2) [right=of 1] {2};
  \node (3) [right=of 2] {3};
  \node (4) [right=of 3] {$\cdots$};
  \node (5) [right=of 4] {$n\!-\!1$};
  \node (6) [right=of 5]{$n$};
  
  \draw (1)--(2)--(3)--(4)--(5)--(6);
\end{tikzpicture}

\begin{tikzpicture}[every node/.style={inner sep=1.5pt}]
\node (0) at (-32,0){$D_n\ (n\geq4)$};
  \node (1) at (-30, 0){1};
  \node (3) [right=of 1] {3};
  \node (2) [above=of 3] {2};
  \node (4) [right=of 3] {4};
  \node (5) [right=of 4] {$\cdots$};
  \node (6) [right=of 5] {$n\!-\!1$};
  \node (7) [right=of 6]{$n$};
  
  \draw (1)--(3)--(4)--(5)--(6)--(7);
  \draw (2)--(3);
\end{tikzpicture}

\begin{tikzpicture}[every node/.style={inner sep=1.5pt}]
\node (0) at (-32,0){$E_6$};
  \node (1) at (-30, 0){1};
  \node (2) [right=of 1] {2};
  \node (3) [right=of 2] {3};
  \node (4) [above=of 3] {4};
  \node (5) [right=of 3] {5};
  \node (6) [right=of 5] {6};

  \draw (1)--(2)--(3)--(5)--(6);
  \draw (3)--(4);
\end{tikzpicture}

\begin{tikzpicture}[every node/.style={inner sep=1.5pt}]
\node (0) at (-32,0){$E_7$};
  \node (1) at (-30, 0){1};
  \node (2) [right=of 1] {2};
  \node (3) [right=of 2] {3};
  \node (4) [above=of 3] {4};
  \node (5) [right=of 3] {5};
  \node (6) [right=of 5] {6};
  \node (7) [right=of 6] {7};

  \draw (1)--(2)--(3)--(5)--(6)--(7);
  \draw (3)--(4);
\end{tikzpicture}

\begin{tikzpicture}[every node/.style={inner sep=1.5pt}]
\node (0) at (-32,0){$E_8$};
  \node (1) at (-30, 0){1};
  \node (2) [right=of 1] {2};
  \node (3) [right=of 2] {3};
  \node (4) [above=of 3] {4};
  \node (5) [right=of 3] {5};
  \node (6) [right=of 5] {6};
  \node (7) [right=of 6] {7};
  \node (8) [right=of 7] {8};

  \draw (1)--(2)--(3)--(5)--(6)--(7)--(8);
  \draw (3)--(4);
\end{tikzpicture}

A function $\epsilon : I \to \{0, 1\}$ is called a \emph{parity function} 
if it satisfies the condition 
\[
\epsilon_i \equiv \epsilon_j + 1 \pmod{2} \quad \text{for any $i, j \in I$ with $c_{ij} = -1$. }
\]
For a parity function $\epsilon$, 
we define $\widehat{I} \coloneq \{ (i, p) \in I \times \Z \mid p - \epsilon_i \in 2\Z\}$. 
\begin{Def}[{\cite[3.7]{HL10}}]\label{Def : C0}
    We define the monoidal category $\mathscr{C}_{\fg}^0$ 
    as the smallest full subcategory of $\mathscr{C}_{\mathfrak{g}}$ 
    satisfying the following conditions : 
    \begin{itemize}
       \item it is closed under tensor products, subquotients, and extensions, 
       \item it contains the fundamental modules 
       $L(Y_{i, q^p})$ for all $(i, p) \in \widehat{I}$. 
   \end{itemize}
\end{Def}
In the study of the category $\Cc_{\fg}$, 
the subcategory $\mathscr{C}_{\fg}^0$ plays a crucial role thanks to the following fact. 
\begin{Prop}[{\cite{C02}}]\label{Prop : C0}
    Any simple module $V$ in $\Cc_{\fg}$ can be factorized into a commutative tensor product 
    of the form 
    \[
    V \cong T_{a_1}(V_1) \otimes T_{a_2}(V_2) \otimes \cdots \otimes T_{a_d}(V_d)
    \]
    with $V_k \in \Cc_{\fg}^0$ and $a_k \in \kk^{\times}$ for $1 \leq k \leq d$ 
    such that $a_k / a_l \not\in q^{2\Z}$ for all $k \neq l$. 
\end{Prop}
In this paper, we always work in the subcategory $\Cc_{\fg}^0 \subset \Cc_{\fg}$. 
From now on, for each $(i, p) \in I \times \Z$, we write $Y_{i, p} \coloneq Y_{i, q^p}$ 
for simplicity. 
Moreover, when there is no risk of confusion, we abbreviate $Y_{i, p}$ as $i_p$. 
For example, we write $L(Y_{1, 0}Y_{1, 2}Y_{2, 3})$ or 
$L(1_0 1_2 2_3)$ instead of 
$L(Y_{1, q^0}Y_{1, q^2}Y_{2, q^3})$.

\begin{Def}[{\cite[Definition 6.9]{KKOP24}}]\label{Def : admissible sequence}
    We say that an infinite sequence 
    $\mathfrak{s} \coloneq \bigl((\iota_k, p_k)\bigr)_{k \in \Z}$ in $I \times \Z$ 
    is \emph{an admissible sequence} if the following conditions are satisfied$\colon$
    \begin{enumerate}
        \item If $s^+<\infty$ then $p_{s^+} = p_s + 2$, 
        \item $p_t = p_s + 1$ if $c_{\iota_s \iota_t} = -1$ and $t^- < s < t < s^+$, 
        \item for every $k \in \Z$, 
        the product $s_{\iota_k}s_{\iota_{k + 1}} \cdots s_{\iota_{k + l - 1}}$ is
        equal to $w_0$ 
        where $l$ denotes the length of the longest element $w_0$ of the Weyl group. 
    \end{enumerate}
    Here $s^+ \coloneq 
    \min\bigl(\{s^{\prime} \mid  s < s^{\prime}, \iota_s = \iota_{s^{\prime}}\} 
    \cup \{ \infty \}\bigr)$ 
    and $s^- \coloneq 
    \max\bigl(\{s^{\prime} \mid  s > s^{\prime}, \iota_s = \iota_{s^{\prime}}\} 
    \cup \{ -\infty \}\bigr)$. 
\end{Def}

The admissible sequence can equivalently be described 
in terms of the $Q$-data and the $Q$-adapted reduced word defined below. 
For the definitions of the non-simply-laced case, 
see \cite[Definition 3.5]{FO21}. 

\begin{Def}\label{Def : Qdatum}
    Let $\Delta$ be the Dynkin diagram for $\mathfrak{g}$. 
    A function $\xi : I \to \Z$ is called a \emph{height  function} 
     on $\Delta$ if $|\xi(i)-\xi(j)|=1$ holds for any 
     $i, j\in I$ such that $c_{ij}=-1$. 
     We call the pair $\mathscr{Q}=(\Delta, \xi)$ a 
     \emph{Q-datum} for $\mathfrak{g}$. 
     \end{Def}

     For a Q-datum $\mathscr{Q}$, we call a vertex $i \in I$ a \emph{sink}
      of $\mathscr{Q}$ if $\xi(i)<\xi(j)$ for all $j\in I$ such that 
      $c_{ij}=-1$. 
      If $i\in I$ is a sink of $\mathscr{Q}$, 
      then we can construct a new height function $s_i\xi$ by 
      \[
      (s_i\xi)(j)= \xi(j)+2\delta_{ij}. 
      \]

     Let $w_0$ be the longest element of the Weyl group of $\mathfrak{g}$. 
     For a reduced expression $\underline{w_0} = s_{i_1}s_{i_2}\cdots s_{i_l}$, 
     we say that $\underline{w_0}$ is \emph{$\mathscr{Q}$-adapted} 
     if 
     \[
     \text{$i_k$ is a sink of }s_{i_{k-1}}s_{i_{k-2}}\cdots s_{i_1}\mathscr{Q} 
     \text{ for all } 1\leq k < l. 
     \]

\begin{Prop}[{\cite[Proposition 6.11]{KKOP24}}]
    Let $X$ be the set of admissible sequences $\mathfrak{s}$, and 
    $Y$ be the set of pairs $(\mathscr{Q}, \underline{w_0})$ of a 
    Q-datum $\mathscr{Q}$ and a $\mathscr{Q}$-adapted reduced expression 
    $\underline{w_0}$ of $w_0$. 
    Then there exists a bijection $\rho : X \to Y$ as follows : 

    For an admissible sequence $\mathfrak{s} = (i_s, p_s)_{s\in \Z}$, 
    the corresponding pair $\rho(\mathfrak{s}) = (\mathscr{Q}, \underline{w_0})$ 
    is given by 
    \[
    \xi : I \to \Z ; \quad 
    i\mapsto p_{{\vartheta}_i}
    \]
    where $\vartheta_i = \text{min}\{k\in \Z_{\geq 1}\mid i_k=i\}$ and 
    \[
    \underline{w_0} = s_{i_1}\cdots s_{i_l}. 
    \]
\end{Prop}
\begin{Def}[{\cite[Definition 6.16]{KKOP24}}]\label{Def : category}
    Let $- \infty \leq a \leq b \leq \infty$ be integers or $\pm \infty$ and 
    $\mathfrak{s} = \bigl((\iota_k, p_k)\bigr)_{k \in \Z}$ be an admissible sequence. 
    We define the monoidal category $\mathscr{C}_{\fg}^{[a, b], \mathfrak{s}}$ 
    as the smallest full subcategory of $\mathscr{C}_{\fg}$ 
    satisfying the following conditions : 
    \begin{itemize}
        \item it is closed under tensor products, subquotients, and extensions, 
        \item it contains the fundamental modules 
        $L(Y_{\iota_k, p_k})$ for all $a \leq k \leq b$ 
        and the trivial module. 
    \end{itemize}
\end{Def}

For example, for any admissible sequence $\mathfrak{s}$, 
the category $\mathscr{C}_{\fg}^{[-\infty, \infty], \mathfrak{s}}$ coincides with  
$\Cc_{\fg}^0$ with respect to one of the parity functions. 

\begin{Prop}[{\cite[$\S$6.3]{PBW24}}]\label{Prop : polynomial}
    For an admissible sequence $\mathfrak{s} = \bigl((\iota_k, p_k)\bigr)_{k \in \Z}$, 
    an interval $[a, b]$ and a monomial $m \in \mathcal{M}^+$, 
    a simple module $L(m)$ is contained in $\mathscr{C}_{\fg}^{[a, b], \mathfrak{s}}$ 
    if and only if $m$ is a monomial of $\{Y_{\iota_k, p_k} \mid a \leq k \leq b\}$. 
    Moreover, the Grothendieck ring $K(\mathscr{C}_{\fg}^{[a, b], \mathfrak{s}})$ is 
    the polynomial algebra generated by $\{[L(Y_{\iota_k, p_k})] \mid a \leq k \leq b\}$. 
\end{Prop}
For an admissible sequence $\mathfrak{s} = \bigl((\iota_k, p_k)\bigr)_{k \in \Z}$ 
and an interval $[a, b]$, 
let $\mathcal{M}_{+}^{[a, b], \mathfrak{s}}$ be a set of monomials 
generated by $\{Y_{\iota_k, p_k} \mid a \leq k \leq b\}$. 
Then, $\Irr \mathscr{C}_{\fg}^{[a, b], \mathfrak{s}}$ is parametrized 
by $\mathcal{M}_{+}^{[a, b], \mathfrak{s}}$. 

Let $\mathfrak{s} = (i_k, p_k)_{k\in\Z}$ be an admissible sequence and 
$\rho(\mathfrak{s}) = (\mathscr{Q}, \underline{w_0})$. 
The Q-datum $\mathscr{Q}$ can be regarded as a Dynkin quiver by orienting 
the edges between adjacent nodes in the direction of decreasing height. 
Let $\mathrm{Rep}(\mathscr{Q})$ denote the category of finite-dimensional representations 
of this Dynkin quiver over $\mathbb{C}$.
Then, we can naturally identify $\{(i_k, p_k)\}_{1\leq k\leq l}$ as 
the set of nodes of the Auslander-Reiten quiver of $\mathrm{Rep}(\mathscr{Q})$, 
where $l$ is the length of $w_0$. 
(see \cite{FO21}). 

In particular, for a Q-datum $\mathscr{Q}$, 
the set $\{(i_k, p_k)\}_{1\leq k\leq l}$ is independent of 
the choice of $\mathscr{Q}$-adapted reduced word $\underline{w_0}$. 
Therefore, the following definition is well-defined. 

\begin{Def}[\cite{HL15}]\label{Def : C_Q}
    Let $\mathscr{Q}$ be a Q-datum. 
    Let $\underline{w_0}$ be a $\mathscr{Q}$-adapted reduced word 
    and $\mathfrak{s} = \rho^{-1}(\mathscr{Q}, \underline{w_0})$. 
    We define the category $\mathscr{C}_{\mathscr{Q}}$ by 
    $\mathscr{C}_{\mathfrak{g}}^{[1, l], \mathfrak{s}}$. 
\end{Def}

\begin{Prop}[{\cite{F22}}]\label{Prop : C_Q equiv}
    Let $\mathscr{Q}_1, \mathscr{Q}_2$ be Q-data. 
    Then, $\mathscr{C}_{\mathscr{Q}_1}$ and $\mathscr{C}_{\mathscr{Q}_2}$ 
    are equivalent as monoidal categories. (see also \S~\ref{Sec : Qdatum}. )
\end{Prop}

\subsection{Examples of monoidal categorification}\label{ssec : ex of mon cat}
In this subsection, we suppose that $\fg$ is of type ADE. 
\begin{Def}\label{Def : GLS}
    Let $\mathfrak{s} = \bigl( (\iota_k, p_k)\bigr)_{k \in \Z}$ be an admissible sequence. 
    For $-\infty \leq a \leq b \leq \infty$, 
    we define the quiver $Q_{\text{GLS}}^{[a, b], \mathfrak{s}}$ 
    as follows : 

    Edges : $\{s \in \Z \mid a \leq s \leq b\}$

    Arrows : \begin{itemize}
        \item $s \longrightarrow t$ \quad if $t^- < s < t < s^+$ and $c_{\iota_s \iota_t} = -1$
        \item $s \longrightarrow s^-$
    \end{itemize}
\end{Def}

\begin{Thm}[{\cite[Theorem 8.1]{KKOP24}}]\label{Thm : KKOP categorification}
    Let $\mathfrak{s} = \bigl((\iota_k, p_k)\bigr)_{k \in \Z}$ be an admissible sequence and 
    $-\infty \leq a \leq b < \infty$. 
    We define a monoidal seed 
    $\mathscr{S} = (\{M_i\}_{i \in K}, \tilde{B} ; K, K^{\text{ex}})$ 
    in $\mathscr{C}_{\fg}^{[a, b], \mathfrak{s}}$ by 
    \begin{itemize}
        \item $K \coloneq [a, b]$, \quad
    $K^{\text{ex}} \coloneq \{s \in [a, b] \mid a \leq s^-\}, \quad
    K^{\text{fr}} \coloneq \{s \in [a, b] \mid s^- < a\}$, 
    \item $\tilde{B} \coloneq 
    \tilde{B}_{Q_{\text{GLS}}^{[a, b], \mathfrak{s}}} = (b_{ij})_{(i, j) \in K \times K^{ex}}$
     \quad (see \S\ref{ssec : cluster}), 
     \item $M_k \coloneq L(Y_{\iota_k, p_k}Y_{\iota_{k^+}, p_{k^+}}\cdots 
     Y_{\iota_{(b(\iota_k)^-)^-}, p_{(b(\iota_k)^-)^-}}Y_{\iota_{b(\iota_k)^-}, p_{b(\iota_k)^-}})$ 
     \quad for $k \in [a, b]$, 
    \end{itemize}
    where $b(i)^- = \max(\{s \leq b\mid \iota_s = i\} \cup \{-\infty\})$. 
    Then, $\mathscr{S}$ is a completely admissible monoidal seed. 
    
    Let $\mathcal{S} = 
    \bigl(\{x_i\}_{i \in K^{ex}}\cup \{f_i\}_{i \in K^{fr}}, 
    \tilde{B}_{Q_{\text{GLS}}^{[a, b], \mathfrak{s}}}\bigr)$ 
    be a seed in a rational function field 
    $\Q(x_i, f_j \mid i \in K^{ex}, j \in K^{fr})$, 
    and let $\mathscr{A}(\mathcal{S})$ 
    denote the associated cluster algebra. 
    
    Then, $\mathscr{C}_{\fg}^{[a, b], \mathfrak{s}}$ is a monoidal categorification 
    of $\mathscr{A}(\mathcal{S})$. 
\end{Thm}

\begin{Ex}\label{Ex : Cs}
We set admissible sequences $\mathfrak{s}_{D_n}^1$ for $n\geq 4$, $\mathfrak{s}_{D_4}^2$, 
$\mathfrak{s}_{E_n}^1$ for $n=6, 7, 8$, $\mathfrak{s}_{A_n}^1$ for $n\geq 2$ 
as follows. 
Note that an admissible sequence $\mathfrak{s}$ is uniquely determined 
by the entries $\mathfrak{s}_0, \dots, \mathfrak{s}_{l-1}$, so we describe only these terms here.
\begin{itemize}
    \item $D_{2n}\ (n\geq 2) : l = 2n(2n-1)$
    \begin{align*}
        (\mathfrak{s}_{D_{2n}}^1)_{1+k+2np}\coloneq
        \begin{cases}
            (2n-1-2k, 2p) \quad&\text{ if } 0\leq k\leq n-2, \\
            (4n-2-2k, 2p+1) \quad &\text{ if } n-1\leq k \leq 2n-2, \\
            (1, 2p+1) &\text{ if } k=2n-1, 
        \end{cases}
    \end{align*}
    for $0 \leq k < 2n, 0 \leq p < 2n-1$. 
    \item $D_{2n-1}\ (n\geq 3) : l = (2n-1)(2n-2)$
    \begin{align*}
        (\mathfrak{s}_{D_{2n-1}}^1)_{1+k+(2n-1)p}\coloneq
        \begin{cases}
            (2n-2-2k, 2p) \quad &\text{ if } 0 \leq k \leq n-2, \\
            (1, 2p) \quad &\text{ if } k=n-1, \\
            (4n-1-2k, 2p+1)\quad &\text{ if } n\leq k\leq 2n-2, 
        \end{cases}
    \end{align*}
    for $0 \leq k <2n-1, 0 \leq p<2n-2$. 
    \item $D_4 : l=12$
    \begin{align*}
        (\mathfrak{s}_{D_4}^2)_{1+k+4p}\coloneq
        \begin{cases}
            (\iota_{1+k}, 2p+1) \quad &\text{ if } k = 0, 1, 2, \\
            (\iota_4, 2p+2) \quad &\text{ if } k=3, 
        \end{cases}
    \end{align*}
    for $0 \leq k <4, 0\leq p < 3$, where 
    $\iota_1, \iota_2, \iota_3, \iota_4 = 1, 2, 4, 3$. 
    \item $E_n\ (n=6, 7, 8)$
    \begin{align*}
        (\mathfrak{s}_{E_n}^1)_{1+k+np}\coloneq
        \begin{cases}
            (\iota_{1+k}, 2p) \quad &\text{ if } 0 \leq k\leq [\frac{n}{2}]-1, \\
            (\iota_{1+k}, 2p+1) \quad &\text{ if } [\frac{n}{2}]\leq k \leq n-1, 
        \end{cases}
    \end{align*}
    for $0 \leq k <  n$. Here, 
    \[
    \begin{cases}
        n=6 : l=36, 0\leq p < 6, \iota_1, \ldots, \iota_6 = 4, 2, 5, 1, 3, 6, \\
        n=7 : l=63, 0\leq p<9, \iota_1, \ldots, \iota_7 = 1, 3, 6, 2, 4, 5, 7, \\
        n=8 : l=120, 0\leq p<15, \iota_1, \ldots, \iota_8 = 8, 1, 3, 6, 2, 4, 5, 7. 
    \end{cases}
    \]
    \item $A_2 : l=3$
    \[
    (\mathfrak{s}_{A_2}^1)_1, (\mathfrak{s}_{A_2}^1)_2, (\mathfrak{s}_{A_2}^1)_3  
    \coloneq (2, 0), (1, 1), (2, 2). 
    \]
    \item $A_3 : l=6$
    \[
    (\mathfrak{s}_{A_3}^1)_1, \ldots, (\mathfrak{s}_{A_3}^1)_6 \coloneq 
    (2, 0), (1, 1), (3, 1), (2, 2), (1, 3), (3, 3). 
    \]
    \item $A_4 : l=10$
    \[
    (\mathfrak{s}_{A_4}^1)_1, \ldots, 
    (\mathfrak{s}_{A_3}^1)_{10}\coloneq 
    (1, 0), (2, 1), (3, 2), (1, 2), (4, 3), (2, 3), (3, 4), (1, 4), (2, 5), (1, 6). 
    \]
    \item $A_{2n-1}\ (n\geq 3) : l=n(2n-1)$
    \begin{align*}
        (\mathfrak{s}_{A_{2n-1}}^1)_{1+k+(2n-1)p} \coloneq
        \begin{cases}
        (2n-2, 2p+1) \quad &\text{ if } k=0, \\
        (2k, 2p+1) \quad &\text{ if } 1 \leq k \leq n-2, \\
        (2k-2n+3, 2p+2) \quad &\text{ if }n-1\leq k \leq 2n-2, 
        \end{cases}
    \end{align*}
    for $0\leq k<2n-1, 0\leq p<n$. 
    \item $A_{2n} : l=n(2n+1)$
    \begin{align*}
        (\mathfrak{s}_{A_{2n-1}}^1)_{1+k+2np} \coloneq
        \begin{cases}
            (2, 2p) \quad &\text{ if } k=0, 0\leq p\leq n, \\
            (2n+1-2k, 2p+1)\quad &\text{ if } 1\leq k\leq n, 0\leq p\leq n-1, \\
            (4n+2-2k, 2p+2)\quad &\text{ if } n+1\leq k\leq 2n-1, 0\leq p\leq n-1, \\
            (2k-1, 2p+1) \quad &\text{ if } 1\leq k\leq n-1, p=n
        \end{cases}
    \end{align*}
    for $0\leq k<2n, 0\leq p\leq n$ and $1+k+2np\leq n(2n+1)$. 
\end{itemize}

For the above admissible sequences, 
when we set intervals as follows, 
the corresponding categories are categorified by the following initial monoidal seeds. 
In the figures below, the cluster variable modules corresponding to the frozen variables 
are enclosed in squares. 
The meaning of the numbers enclosed in circles in the figures below will be explained later in 
Theorem~\ref{Thm : affine categorification}.

\begin{align*}
 &\mathscr{C}_{D_{2n}}^{[1, 4n+1], \mathfrak{s}_{D_{2n}}^1} (n\geq 2): \\
&\raisebox{2.3em}{\scalebox{0.7}{\xymatrix@C=0.7ex@R=1ex{
 && \boxed{L(1_1 1_3)} \ar[ddr] 
 && L(1_3),\tikz[baseline=(c.base)] \node[draw,circle,inner sep=1pt, scale=0.9] (c) {n\!+\!2};\ar[ll]\\
 && \boxed{L(2_1 2_3)} \ar[dr] 
 && L(2_3),\tikz[baseline=(c.base)] \node[draw,circle,inner sep=1pt, scale=0.9] 
 (c) {n\!+\!1};\ar[ll]\\
 &\boxed{L(3_0 3_2)}  && L(3_2) \ar[ll] \ar[uur] \ar[ur]\ar[dr]\\
 && \boxed{L(4_1 4_3)} \ar[ur]\ar[dr] 
 && L(4_3),\tikz[baseline=(c.base)] \node[draw,circle,inner sep=1pt, scale=1.2] (c) {n};\ar[ll]\\
 &\boxed{L(5_0 5_2)}  && L(5_2) \ar[ll] \ar[ur]\\
 & \vdots & \vdots &\vdots&\vdots \\
 && \boxed{L\bigl((2n-4)_1 (2n-4)_3\bigr)} \ar[dr] 
 && L\bigl((2n-4)_3\bigr)
 ,\tikz[baseline=(c.base)] \node[draw,circle,inner sep=1pt] (c) {4};\ar[ll]\\
 &\boxed{L\bigl((2n-3)_0 (2n-3)_2 \bigr)} 
&& L\bigl((2n-3)_2\bigr) \ar[ll] \ar[ur]\ar[dr]\\
 && \boxed{L\bigl((2n-2)_1 (2n-2)_3\bigr)} \ar[ur]\ar[dr] 
 && L\bigl((2n-2)_3\bigr)
 ,\tikz[baseline=(c.base)] \node[draw,circle,inner sep=1pt] (c) {3};\ar[ll]\ar[dr]\\
 &\boxed{L\bigl((2n-1)_0 (2n-1)_2 (2n-1)_4\bigr)} 
&& L\bigl( (2n-1)_2 (2n-1)_4\bigr) \ar[ll] \ar[ur]\ar[dr] 
&&L( (2n-1)_{4}),\tikz[baseline=(c.base)] \node[draw,circle,inner sep=1pt] (c) {1};\ar[ll] \\
 && \boxed{L\bigl((2n)_1 (2n)_3\bigr)} \ar[ur] && L\bigl((2n)_3\bigr)
 ,\tikz[baseline=(c.base)] \node[draw,circle,inner sep=1pt] (c) {2};\ar[ll]\ar[ur]\\
}}}
\end{align*}
\begin{align*}
&\mathscr{C}_{D_{2n-1}}^{[1, 4n-1], \mathfrak{s}_{D_{2n-1}}^1} (n\geq 3):\\
&\raisebox{2.3em}{\scalebox{0.7}{\xymatrix@C=0.7ex@R=1ex{
 & \boxed{L(1_0 1_2)}  && L(1_2)\ar[ll]\ar[ddr]\\
 & \boxed{L(2_0 2_2)}  && L(2_2)\ar[ll]\ar[dr]\\
 &&\boxed{L(3_1 3_3)} \ar[uur] \ar[ur] \ar[dr] && L(3_3)
 ,\tikz[baseline=(c.base)] \node[draw,circle,inner sep=1pt,scale=1.2] (c) {n};\ar[ll] \\
 & \boxed{L(4_0 4_2)} && L(4_2)\ar[ll]\ar[ur]\ar[dr]\\
 &&\boxed{L(5_1 5_3)} \ar[ur] && L(5_3)
 ,\tikz[baseline=(c.base)] \node[draw,circle,inner sep=1pt,scale=0.9] (c) {n-1};\ar[ll]\\
 & \vdots & \vdots &\vdots&\vdots \\
 && \boxed{L\bigl((2n-5)_1 (2n-5)_3\bigr)} \ar[dr] && L\bigl((2n-5)_3\bigr)
 ,\tikz[baseline=(c.base)] \node[draw,circle,inner sep=1pt] (c) {4};\ar[ll]\\
 &\boxed{L\bigl((2n-4)_0 (2n-4)_2 \bigr)} 
&& L\bigl((2n-4)_2\bigr) \ar[ll] \ar[ur]\ar[dr]\\
 && \boxed{L\bigl((2n-3)_1 (2n-3)_3\bigr)} \ar[ur]\ar[dr] 
 && L\bigl((2n-3)_3\bigr)
 ,\tikz[baseline=(c.base)] \node[draw,circle,inner sep=1pt] (c) {3};\ar[ll]\ar[dr]\\
 &\boxed{L\bigl((2n-2)_0 (2n-2)_2 (2n-2)_4\bigr)}
&& L\bigl( (2n-2)_2 (2n-2)_4\bigr) \ar[ll] \ar[ur]\ar[dr] 
&&L( (2n-2)_{4}),\tikz[baseline=(c.base)] \node[draw,circle,inner sep=1pt] (c) {1};\ar[ll] \\
 && \boxed{L\bigl((2n-1)_1 (2n-1)_3\bigr)} \ar[ur] && L\bigl((2n-1)_3\bigr)
 ,
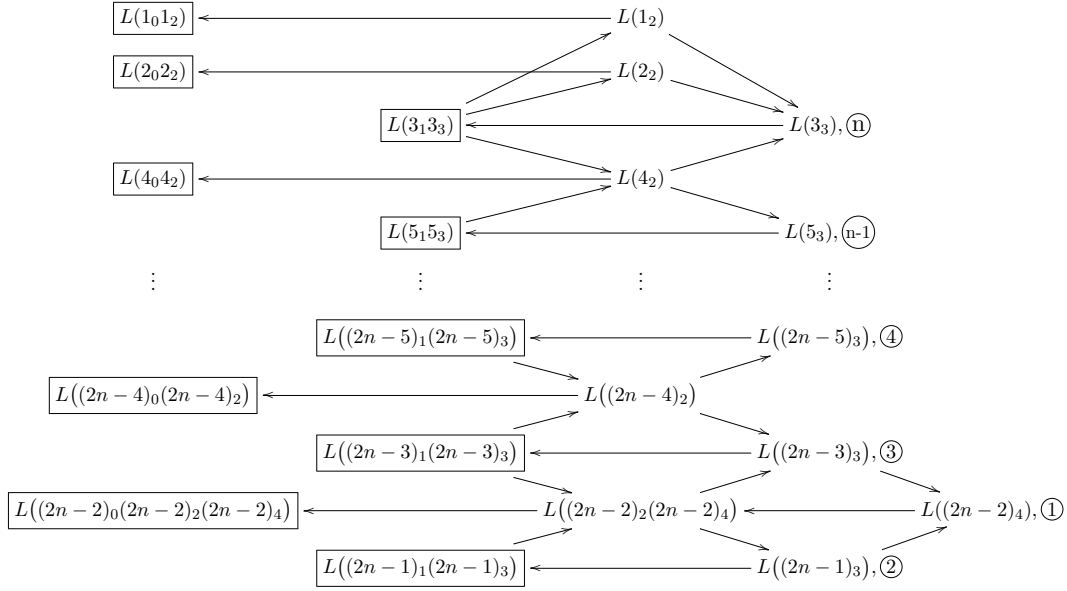
\tikz[baseline=(c.base)] \node[draw,circle,inner sep=1pt] (c) {2};\ar[ll]\ar[ur]\\
}}}
\end{align*}
\begin{align*}
&\mathscr{C}_{D_4}^{[1, 11], \mathfrak{s}_{D_4}^2}:
&\raisebox{2.3em}{\scalebox{0.7}
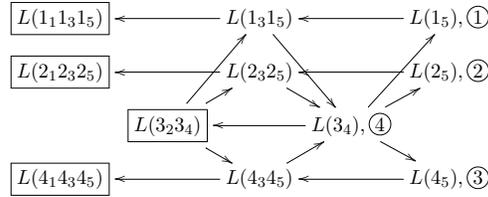
{\xymatrix@C=0.7ex@R=1ex{
 & \boxed{L(1_1 1_3 1_5)} && L(1_3 1_5)\ar[ll]\ar[ddr]
 &&L(1_5)
 ,\tikz[baseline=(c.base)] \node[draw,circle,inner sep=1pt] (c) {1};\ar[ll]\\
 & \boxed{L(2_1 2_3 2_5)}  && L(2_3 2_5)\ar[ll]\ar[dr]
 &&L(2_5)
 ,\tikz[baseline=(c.base)] \node[draw,circle,inner sep=1pt] (c) {2};\ar[ll]\\
 &&\boxed{L(3_2 3_4)} \ar[uur] \ar[ur] \ar[dr] 
 && L(3_4)
 ,\tikz[baseline=(c.base)] \node[draw,circle,inner sep=1pt] (c) {4};\ar[ll]\ar[uur]\ar[ur]\ar[dr]\\
 & \boxed{L(4_1 4_3 4_5)}  && L(4_3 4_5)\ar[ll]\ar[ur]
 &&L(4_5)
 ,\tikz[baseline=(c.base)] \node[draw,circle,inner sep=1pt] (c) {3};\ar[ll]\\
 }}}
&\mathscr{C}_{E_6}^{[1, 13], \mathfrak{s}_{E_6}^1}:
&\raisebox{2.3em}{\scalebox{0.68}{\xymatrix@C=0.7ex@R=1ex{
 && \boxed{L(1_1 1_3)} \ar[dr] && L(1_3)
 ,\tikz[baseline=(c.base)] \node[draw,circle,inner sep=1pt] (c) {3};\ar[ll]\\
 & \boxed{L(2_0 2_2)}  && L(2_2)\ar[ll]\ar[ur]\ar[dr]\\
 && \boxed{L(3_1 3_3)} \ar[ur]\ar[dr]\ar[ddr] && L(3_3)
 ,\tikz[baseline=(c.base)] \node[draw,circle,inner sep=1pt] (c) {2};\ar[ll]\ar[dr]\\
 &\boxed{L\bigl(4_0 4_2 4_4\bigr)} 
&& L\bigl( 4_2 4_4\bigr) \ar[ll] \ar[ur]
&&L( 4_{4})
,\tikz[baseline=(c.base)] \node[draw,circle,inner sep=1pt] (c) {1};\ar[ll] \\
& \boxed{L(5_0 5_2)}  && L(5_2)\ar[ll]\ar[uur]\ar[dr]\\
&& \boxed{L\bigl(6_1 6_3\bigr)} \ar[ur] && L\bigl(6_3\bigr)
,
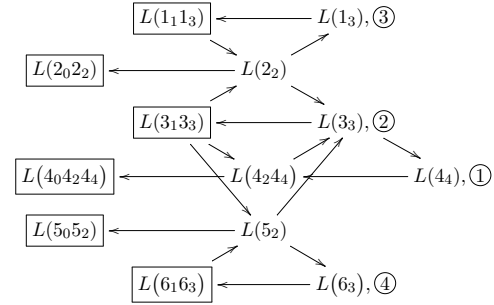
\tikz[baseline=(c.base)] \node[draw,circle,inner sep=1pt] (c) {4};\ar[ll]\\
}}}
\end{align*}
\begin{align*}
&\mathscr{C}_{E_7}^{[1, 15], \mathfrak{s}_{E_7}^1}:
&\raisebox{2.3em}{\scalebox{0.7}{\xymatrix@C=0.7ex@R=1ex{
 & \boxed{L(1_0 1_2 1_4)}  && L(1_2 1_4)\ar[ll]\ar[dr]
 &&L(1_4)
 ,\tikz[baseline=(c.base)] \node[draw,circle,inner sep=1pt] (c) {1};\ar[ll]\\
 && \boxed{L(2_1 2_3)} \ar[ur]\ar[dr] && L(2_3)
 ,\tikz[baseline=(c.base)] \node[draw,circle,inner sep=1pt] (c) {2};\ar[ll]\ar[ur]\\
 & \boxed{L(3_0 3_2)}  && L(3_2)\ar[ll]\ar[ur]\ar[dr]\ar[ddr]\\
 &&\boxed{L\bigl(4_1 4_3\bigr)} \ar[ur] 
&& L\bigl( 4_3\bigr)
,\tikz[baseline=(c.base)] \node[draw,circle,inner sep=1pt] (c) {3};\ar[ll]\\
&& \boxed{L(5_1 5_3)} \ar[uur]\ar[dr] && L(5_3)
,\tikz[baseline=(c.base)] \node[draw,circle,inner sep=1pt] (c) {4};\ar[ll]\\
& \boxed{L\bigl(6_0 6_2\bigr)}  && L\bigl(6_2\bigr)\ar[ll]\ar[ur]\ar[dr]\\
&&\boxed{L(7_1 7_3)} \ar[ur]&& L(7_3)
,
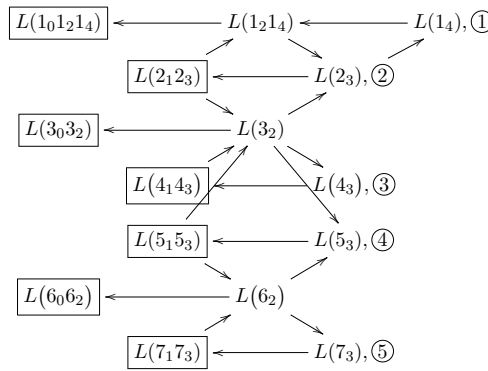
\tikz[baseline=(c.base)] \node[draw,circle,inner sep=1pt] (c) {5};\ar[ll]\\
}}}
&\mathscr{C}_{E_8}^{[1, 17], \mathfrak{s}_{E_8}^1}:
&\raisebox{2.3em}{\scalebox{0.68}
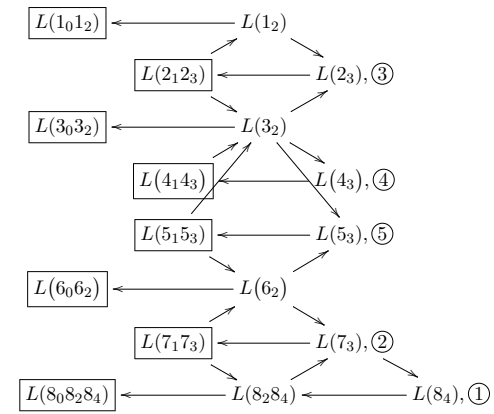
{\xymatrix@C=0.7ex@R=1ex{
 & \boxed{L(1_0 1_2)}  && L(1_2)\ar[ll]\ar[dr]\\
 && \boxed{L(2_1 2_3)} \ar[ur]\ar[dr] && L(2_3)
 ,\tikz[baseline=(c.base)] \node[draw,circle,inner sep=1pt] (c) {3};\ar[ll]\\
 & \boxed{L(3_0 3_2)} && L(3_2)\ar[ll]\ar[ur]\ar[dr]\ar[ddr]\\
 &&\boxed{L\bigl(4_1 4_3\bigr)} \ar[ur] 
&& L\bigl( 4_3\bigr)
,\tikz[baseline=(c.base)] \node[draw,circle,inner sep=1pt] (c) {4};\ar[ll]\\
&& \boxed{L(5_1 5_3)} \ar[uur]\ar[dr] && L(5_3)
,\tikz[baseline=(c.base)] \node[draw,circle,inner sep=1pt] (c) {5};\ar[ll]\\
& \boxed{L\bigl(6_0 6_2\bigr)}  && L\bigl(6_2\bigr)\ar[ll]\ar[ur]\ar[dr]\\
&&\boxed{L(7_1 7_3)} \ar[ur]\ar[dr]&& L(7_3)
,\tikz[baseline=(c.base)] \node[draw,circle,inner sep=1pt] (c) {2};\ar[ll]\ar[dr]\\
&\boxed{L(8_0 8_2 8_4)} &&L(8_2 8_4) \ar[ll]\ar[ur] &&L(8_4)
,\tikz[baseline=(c.base)] \node[draw,circle,inner sep=1pt] (c) {1};\ar[ll]\\
}}}
\end{align*}
\begin{align*}
 &\mathscr{C}_{A_2}^{[1, 11], \mathfrak{s}_{A_2}^1}:\\
&\raisebox{1em}{\scalebox{0.625}
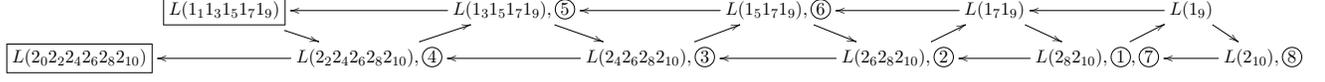
{\xymatrix@C=0.01ex@R=1ex{
 && \boxed{L(1_1 1_3 1_5 1_7 1_9)}\ar[dr] 
 && L(1_3 1_5 1_7 1_9)
 ,\tikz[baseline=(c.base)] \node[draw,circle,inner sep=1pt] (c) {5};\ar[ll]\ar[dr]
 && L(1_5 1_7 1_9)
 ,\tikz[baseline=(c.base)] \node[draw,circle,inner sep=1pt] (c) {6};\ar[ll]\ar[dr]
 && L(1_7 1_9) \ar[ll]\ar[dr]
 && L(1_9) \ar[ll]\ar[dr]\\
 & \boxed{L(2_0 2_2 2_4 2_6 2_8 2_{10})} 
 && L(2_2 2_4 2_6 2_8 2_{10})
 ,\tikz[baseline=(c.base)] \node[draw,circle,inner sep=1pt] (c) {4};\ar[ll]\ar[ur]
 && L(2_4 2_6 2_8 2_{10})
 ,\tikz[baseline=(c.base)] \node[draw,circle,inner sep=1pt] (c) {3};\ar[ll]\ar[ur]
 && L(2_6 2_8 2_{10})
 ,\tikz[baseline=(c.base)] \node[draw,circle,inner sep=1pt] (c) {2};\ar[ll]\ar[ur]
 && L(2_8 2_{10})
 ,\tikz[baseline=(c.base)] \node[draw,circle,inner sep=1pt] (c) {1};
 ,\tikz[baseline=(c.base)] \node[draw,circle,inner sep=1pt] (c) {7};\ar[ll]\ar[ur]
 && L(2_{10})
 ,\tikz[baseline=(c.base)] \node[draw,circle,inner sep=1pt] (c) {8};\ar[ll]\\
}}}
\end{align*}
\begin{itemize}
    \item $A_3 : l=6, (\mathfrak{s}_{A_3}^1)_1, \ldots, (\mathfrak{s}_{A_3}^1)_l = 
(2, 0), (1, 1), (3, 1), (2, 2), (1, 3), (3, 3), $
\end{itemize}
\begin{align*}
&\mathscr{C}_{A_3}^{[1, 10], \mathfrak{s}_{A_3}^1}:
&\raisebox{2.3em}{\scalebox{0.7}
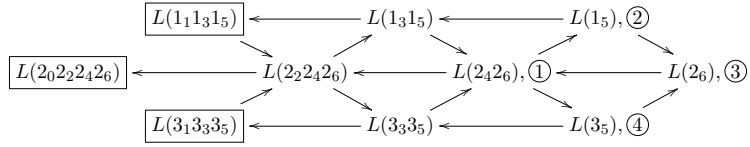
{\xymatrix@C=0.7ex@R=1ex{
 && \boxed{L(1_1 1_3 1_5)} \ar[dr] && L(1_3 1_5)\ar[ll]\ar[dr]
 &&L(1_5)
 ,\tikz[baseline=(c.base)] \node[draw,circle,inner sep=1pt] (c) {2};\ar[ll]\ar[dr]\\
 & \boxed{L(2_0 2_2 2_4 2_6)} && L(2_2 2_4 2_6)\ar[ll]\ar[ur]\ar[dr]
 &&L(2_4 2_6)
 ,\tikz[baseline=(c.base)] \node[draw,circle,inner sep=1pt] (c) {1};\ar[ll]\ar[ur]\ar[dr] &&L(2_6)
 ,\tikz[baseline=(c.base)] \node[draw,circle,inner sep=1pt] (c) {3};\ar[ll]\\
 && \boxed{L(3_1 3_3 3_5)} \ar[ur] && L(3_3 3_5)\ar[ll]\ar[ur]
 &&L(3_5)
 ,\tikz[baseline=(c.base)] \node[draw,circle,inner sep=1pt] (c) {4};\ar[ll]\ar[ur]\\
}}}
\end{align*}
\begin{align*}
&\mathscr{C}_{A_4}^{[1, 13], \mathfrak{s}_{A_4}^1}:
&\raisebox{2.3em}{\scalebox{0.7}{\xymatrix@C=0.7ex@R=1ex{
 & \boxed{L(1_0 1_2 1_4 1_6)} 
 && L(1_2 1_4 1_6)
 ,\tikz[baseline=(c.base)] \node[draw,circle,inner sep=1pt] (c) {3};\ar[ll]\ar[dr]
 && L(1_4 1_6)
 ,\tikz[baseline=(c.base)] \node[draw,circle,inner sep=1pt] (c) {4};\ar[ll]\ar[dr]
 && L(1_6 1_8)
 ,\tikz[baseline=(c.base)] \node[draw,circle,inner sep=1pt] (c) {8};\ar[ll]\ar[dr]\\
 && \boxed{L(2_1 2_3 2_5 2_7)} \ar[ur]
 && L(2_3 2_5 2_7)
 ,\tikz[baseline=(c.base)] \node[draw,circle,inner sep=1pt] (c) {2};\ar[ll]\ar[ur]\ar[dr]
 && L(2_5 2_7)
 ,\tikz[baseline=(c.base)] \node[draw,circle,inner sep=1pt] (c) {1};\ar[ll]\ar[ur]\ar[dr] 
 &&L(2_7)\ar[ll]\\
 &&& \boxed{L(3_2 3_4 3_6)} \ar[ur]
 && L(3_4 3_6)
 ,\tikz[baseline=(c.base)] \node[draw,circle,inner sep=1pt] (c) {7};\ar[ll]\ar[ur]\ar[dr]
 &&L(3_6)
 ,\tikz[baseline=(c.base)] \node[draw,circle,inner sep=1pt] (c) {5};\ar[ll]\ar[ur]\\
 &&&& \boxed{L\bigl(4_3 4_5\bigr)} \ar[ur] && L\bigl(4_5\bigr)
 ,\tikz[baseline=(c.base)] \node[draw,circle,inner sep=1pt] (c) {6};\ar[ll]\ar[ur]\\
}}}
\end{align*}
\begin{align*}
 &\mathscr{C}_{A_{2n-1}}^{[1, 4n], \mathfrak{s}_{A_{2n-1}}^1}(n\geq3):\\
&\raisebox{2.3em}{\scalebox{0.7}{\xymatrix@C=0.7ex@R=1ex{
 && \boxed{L(1_2 1_4)} \ar[dr] && L(1_4)
 ,\tikz[baseline=(c.base)] \node[draw,circle,inner sep=1pt] (c) {2};\ar[ll]\ar[dr]\\
 &\boxed{L(2_1 2_3 2_5)}  && L(2_3 2_5) \ar[ll] \ar[ur]\ar[dr]
 &&L(2_5)
 ,\tikz[baseline=(c.base)] \node[draw,circle,inner sep=1pt] (c) {1};\ar[ll]\\
 && \boxed{L(3_2 3_4)} \ar[ur]\ar[dr] && L(3_4)
 ,\tikz[baseline=(c.base)] \node[draw,circle,inner sep=1pt] (c) {3};\ar[ll]\ar[ur]\\
 &\boxed{L(4_1 4_3)}  && L(4_3) \ar[ll] \ar[ur]\\
 & \vdots & \vdots &\vdots&\vdots \\
 && \boxed{L\bigl((2n-5)_2 (2n-5)_4\bigr)} \ar[dr] && L\bigl((2n-5)_4\bigr)
 ,\tikz[baseline=(c.base)] \node[draw,circle,inner sep=1pt] (c) {n-1};\ar[ll]\\
 &\boxed{L\bigl((2n-4)_1 (2n-4)_3 \bigr)} 
&& L\bigl((2n-4)_3\bigr) \ar[ll] \ar[ur]\ar[dr]\\
 && \boxed{L\bigl((2n-3)_2 (2n-3)_4\bigr)} \ar[ur]\ar[dr] 
 && L\bigl((2n-3)_4\bigr)
 ,\tikz[baseline=(c.base)] \node[draw,circle,inner sep=1pt] (c) {n+1};\ar[ll]\ar[dr]\\
 &\boxed{L\bigl((2n-2)_1 (2n-2)_3 (2n-2)_5\bigr)} 
&& L\bigl( (2n-2)_3 (2n-2)_5\bigr) \ar[ll] \ar[ur]\ar[dr] 
&&L( (2n-2)_5)
,\tikz[baseline=(c.base)] \node[draw,circle,inner sep=1pt] (c) {n};\ar[ll] \\
 && \boxed{L\bigl((2n-1)_3 (2n-1)_4\bigr)} \ar[ur] && L\bigl((2n-1)_4\bigr)
 ,\tikz[baseline=(c.base)] \node[draw,circle,inner sep=1pt] (c) {n+2};\ar[ll]\ar[ur]\\
}}}
\end{align*}
\begin{align*}
 &\mathscr{C}_{A_{2n}}^{[1, 4n+2], \mathfrak{s}_{A_{2n}}^1}(n\geq3):\\
&\raisebox{2.3em}{\scalebox{0.68}
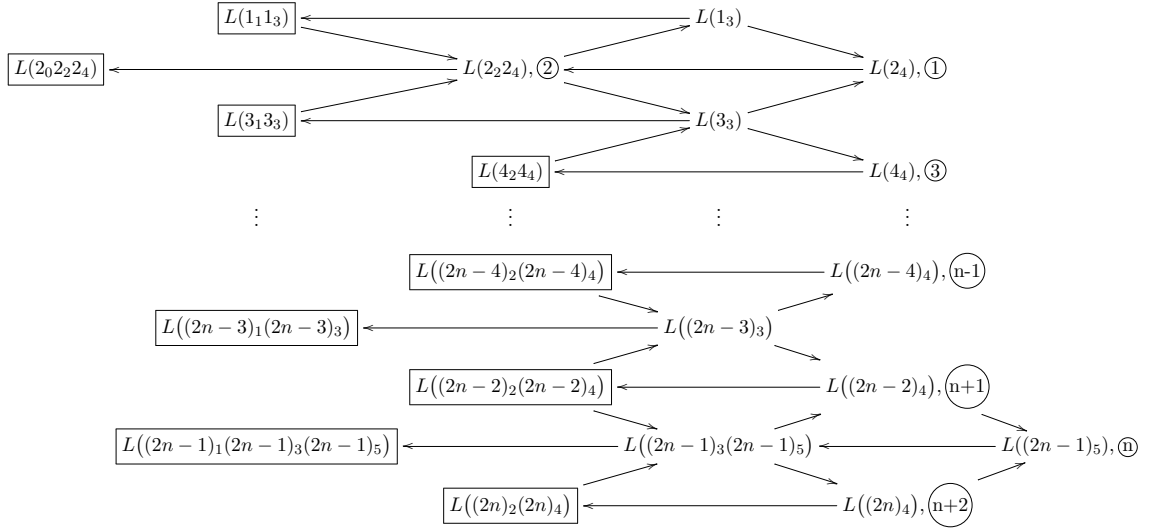
{\xymatrix@C=0.01ex@R=0.9ex{
 && \boxed{L(1_1 1_3)} \ar[dr] && L(1_3)\ar[ll]\ar[dr]\\
 &\boxed{L(2_0 2_2 2_4)}  && L(2_2 2_4),\tikz[baseline=(c.base)] \node[draw,circle,inner sep=1pt] (c) {2};
 \ar[ll] \ar[ur]\ar[dr]
 &&L(2_4),\tikz[baseline=(c.base)] \node[draw,circle,inner sep=1pt] (c) {1};
 \ar[ll]\\
 && \boxed{L(3_1 3_3)} \ar[ur] && L(3_3)\ar[ll]\ar[ur]\ar[dr]\\
 &&&\boxed{L(4_2 4_4)} \ar[ur] && L(4_4)
 ,\tikz[baseline=(c.base)] \node[draw,circle,inner sep=1pt] (c) {3};
 \ar[ll] \\
 &  & \vdots &\vdots&\vdots &\vdots\\
 &&& \boxed{L\bigl((2n-4)_2 (2n-4)_4\bigr)} \ar[dr] && L\bigl((2n-4)_4\bigr)
 ,\tikz[baseline=(c.base)] \node[draw,circle,inner sep=1pt] (c) {n-1};\ar[ll]\\
 &&\boxed{L\bigl((2n-3)_1 (2n-3)_3 \bigr)} 
&& L\bigl((2n-3)_3\bigr) \ar[ll] \ar[ur]\ar[dr]\\
 &&& \boxed{L\bigl((2n-2)_2 (2n-2)_4\bigr)} \ar[ur]\ar[dr] 
 && L\bigl((2n-2)_4\bigr),
 \tikz[baseline=(c.base)] \node[draw,circle,inner sep=1pt] (c) {n+1};
 \ar[ll]\ar[dr]\\
 &&\boxed{L\bigl((2n-1)_1 (2n-1)_3 (2n-1)_5\bigr)} 
&& L\bigl( (2n-1)_3 (2n-1)_5\bigr) \ar[ll] \ar[ur]\ar[dr] 
&&L( (2n-1)_{5}),\tikz[baseline=(c.base)] \node[draw,circle,inner sep=1pt] (c) {n};
\ar[ll] \\
 &&& \boxed{L\bigl((2n)_2 (2n)_4\bigr)} \ar[ur] && L\bigl((2n)_4\bigr)
 ,\tikz[baseline=(c.base)] \node[draw,circle,inner sep=1pt] (c) {n+2};
 \ar[ll]\ar[ur]\\
}}}
\end{align*}
If there is no confusion, 
we simply write $\mathfrak{s}^1, \mathfrak{s}^2$ instead of 
$\mathfrak{s}_{\mathfrak{g}}^1, \mathfrak{s}^2_{D_4}$. 
\end{Ex}
\begin{Prop}[{\cite[Theorem 4.21]{KKOP24}}]\label{Prop : frozenKR}
    Let $\mathfrak{s} = \bigl((\iota_k, p_k)\bigr)_{k \in \Z}$ be an admissible sequence and 
    $-\infty < a \leq b < \infty$. 
    Let  $\mathscr{S} = (\{M_i\}_{i \in K}, \tilde{B} ; K, K^{\text{ex}})$ 
    be the monoidal seed of $\mathscr{C}_{\fg}^{[a, b], \mathfrak{s}}$ 
    defined in \ref{Thm : KKOP categorification}. 
    Then, for every $s \in K^{\text{fr}}$, 
    the KR module $M_s$ strongly commutes with 
    every simple module in $\mathscr{C}_{\fg}^{[a, b], \mathfrak{s}}$. 
\end{Prop}
\begin{Prop}\label{Prop : prime module}
    Let $\mathfrak{s} = \bigl((\iota_k, p_k)\bigr)_{k \in \Z}$ be an admissible sequence and 
    $-\infty < a \leq b < \infty$. 
    If a simple module $L \in\Irr \mathscr{C}_{\fg}^{[a, b], \mathfrak{s}}$ 
corresponds to a cluster variable, 
    then it is a prime module. 
\end{Prop}
\begin{proof}
Let $L(m)$ be a simple module in 
$\mathscr{C}_{\fg}^{[a, b], \mathfrak{s}}$ corresponding to a cluster variable, 
where $m \in \mathcal{M}_{+}^{[a, b], \mathfrak{s}}$. 
    Let $m_1, m_2 \in \mathcal{M}_+$ be monomials satisfying $L(m) \cong L(m_1)\otimes L(m_2)$. 
    Then, by Proposition~\ref{Prop : polynomial}, 
    $m_1$ and $m_2$ are in $\mathcal{M}_{+}^{[a, b], \mathfrak{s}}$. 
    So $[L(m)]$ has a factorization $[L(m_1)][L(m_2)]$ in 
    $K(\mathscr{C}_{\fg}^{[a, b], \mathfrak{s}})$. 
    Since $[L(m)]$ is an irreducible element in $K(\mathscr{C}_{\fg}^{[a, b], \mathfrak{s}})$ 
    by Proposition~\ref{Prop : cl.var is irr}, 
    $[L(m_1)]$ or $[L(m_2)]$ is the trivial module. 
\end{proof}

\subsection{Monoidal categorifications of cluster algebras of affine type}

\begin{Thm}\label{Thm : affine categorification}
    The categories listed in Example~\ref{Ex : Cs} are monoidal categorifications of 
    cluster algebras of affine type. 
    \begin{center}
    \begin{longtable}{c|c}
        category & cluster algebra type \\ \hline
         $\mathscr{C}_{D_n}^{[1, 2n+1], \mathfrak{s}^1}(n\geq 4)$& $D_{n}^{(1)}$\\
         $\mathscr{C}_{D_4}^{[1, 11], \mathfrak{s}^2}$& $E_{6}^{(1)}$\\
         $\mathscr{C}_{E_n}^{[1, 2n+1], \mathfrak{s}^1}(n=6, 7, 8)$& $E_{n}^{(1)}$\\
         $\mathscr{C}_{A_2}^{[1, 11], \mathfrak{s}^1}$& $E_{8}^{(1)}$\\
         $\mathscr{C}_{A_3}^{[1, 10], \mathfrak{s}^1}$& $E_{6}^{(1)}$\\
         $\mathscr{C}_{A_4}^{[1, 13], \mathfrak{s}^1}$& $E_{8}^{(1)}$\\
         $\mathscr{C}_{A_n}^{[1, 2n+2], \mathfrak{s}^1}(n\geq5)$& $D_{n+1}^{(1)}$\\
    \end{longtable}
    \end{center}
\end{Thm}
\begin{proof}
   By performing the mutations in the order depicted in 
   Example~\ref{Ex : Cs}, 
   we obtain affine-type seeds as stated in the theorem.
\end{proof}


\section{Cluster algebras of affine type}\label{Sec : affinecluster}
In this section, we consider a cluster algebra $\mathscr{A} = \mathscr{A}(\mathcal{S})$, 
where 
$\mathcal{S} = (\{x_i\}_{i \in K^{\text{ex}}}\cup \{f_j\}_{j \in K^{\text{fr}}}, 
\tilde{B} = (b_{ij})_{(i, j)\in K\times K^{\text{ex}}})$ 
is a seed such that the quiver corresponding to the principal part $B$ of $\tilde{B}$ 
is an acyclic Dynkin quiver of untwisted affine type. 
We refer to such a seed as an \emph{acyclic seed of untwisted affine type}. 

\subsection{Properties of cluster algebras of affine type}
Let $(\mathfrak{h}, \Pi, \Pi^{\vee})$ be a realization of the Cartan matrix 
associated with this Dynkin quiver, 
where $\mathfrak{h}$ is a complex vector space, 
$\Pi = \{\alpha_0, \ldots, \alpha_n\} \subset \mathfrak{h}^{*}$ is a set of \emph{simple roots}, 
and $\Pi^{\vee} = \{h_0, \ldots, h_n\} \subset \mathfrak{h}$ is a set of \emph{simple coroots}
(see \cite{Kac83}). 
We set the \emph{affine root system} $\Phi \subset \mathfrak{h}^*$, 
the \emph{root lattice} $Q \coloneq \displaystyle \bigoplus_{i=0}^n \Z\alpha_i$, 
and the \emph{positive root lattice} 
$Q^{+} \coloneq \displaystyle \bigoplus_{i=0}^n \Z_{\geq 0}\alpha_i$. 
For $\beta_1, \beta_2\in Q$, we write $\beta_1 \geq \beta_2$ if $\beta_1-\beta_2 \in Q^+$. 
We also define the set of \emph{positive roots} by $\Phi^+ \coloneq \Phi \cap Q^+$. 
In this paper, we only consider untwisted affine root systems. 
For an untwisted affine Dynkin diagram of type $X_n^{(1)}$ for $X=A, \ldots, G$, 
we fix the index $0$ so that the subgraph on $\{1, \ldots, n\}$ forms 
the finite Dynkin diagram of type $X_n$. 
Let $\Phi_{\mathrm{fin}}$ be the finite root system whose simple roots are 
$\{\alpha_1, \ldots, \alpha_n\}$, viewed as a subsystem of $\Phi$. 
Let $\delta \in \Phi^+$ be the \emph{null root}, characterized by the condition
\[\{\lambda \in Q \mid \langle h_i, \lambda\rangle = 0\text{ for every }0\leq i\leq n\} 
= \Z\delta. 
\]
For $0 \leq i \leq n$, we set the \emph{simple reflection} $s_i \in GL(\mathfrak{h}^*)$ by 
\[s_i \lambda = \lambda - \langle h_i, \lambda\rangle\alpha_i \quad 
\text{ for $\lambda$ in $\mathfrak{h}^*$}
\]
and define the \emph{Weyl group} by $W \coloneq \langle s_i \mid 0\leq i\leq n\rangle$. 

For the exchange quiver $\tilde{B}$, 
we set a Coxeter element $c = s_{i_0}\cdots s_{i_n} \in W$ satisfying  
$\{i_0, \ldots, i_n\} = \{0, \ldots, n\}$ and $k<l$ implies $b_{i_k i_l}\geq 0$. 
We call such a Coxeter element \emph{Coxeter element associated with $\tilde{B}$} . 
Since $\Phi$ is closed under the action of the Weyl group, 
we can consider the action of $c$ on $\Phi$. 
\begin{Def}[{\cite[Definition 3.1]{affine}}]\label{Def : Phi_c}
    Let $\Phi$ be a root system of untwisted affine type and $c$ be a Coxeter element. 
    We define 
    \begin{align*}
    \Lambda_c^{\text{re}} &= \{\beta, \delta - \beta \mid \beta \in \Phi_{\text{fin}}^+ 
    \text{ and the $c$-orbit of $\beta$ is finite}\}, \\
    \Phi_c^{\text{re}} &= -\Pi \cup 
    \{\beta \in \Phi^+ \mid \text{the $c$-orbit of $\beta$ is infinite}\} \cup 
    \Lambda_c^{\text{re}}, \\
    \Lambda_c &= \Lambda_c^{\mathrm{re}}\cup \{\delta\}, 
    \Phi_c = \Phi_c^{\text{re}} \cup \{\delta\}. 
    \end{align*}
\end{Def}
We note that $\Phi_c \cap \Z\delta = \{\delta\}$ since every element of $\Z\delta$ is fixed 
by the Weyl group action and, in particular, the $c$-orbit is finite. 
Also, we note that $\Lambda_c^{\text{re}} = \Lambda_{c^{-1}}^{\text{re}}$, 
$\Phi_c^{\text{re}} = \Phi_{c^{-1}}^{\text{re}}$, and 
$\Phi_c = \Phi_{c^{-1}}$. 

\begin{Prop}[{\cite[Proposition 3.1, Lemma 3.6]{coxeter}}]\label{Prop :  -alpha_i in Phi}
    Let $\Phi$ be a root system of untwisted affine type and $c$ be a Coxeter element. 
    Then, for any $\alpha_i \in \Pi$ and $n \in \Z_{\geq 0}$, 
    $n\delta -\alpha_i \in \Phi_c^{\text{re}}$. 
\end{Prop}

\begin{Def}[{\cite{FZ03}}]\label{Def : dvector}
    Let $\mathcal{S} = (\{x_i\}_{i \in K^{\text{ex}}}\cup \{f_j\}_{j \in K^{\text{fr}}}, 
    \tilde{B} = (b_{ij})_{(i, j)\in K\times K^{\text{ex}}})$ be a seed such that 
    $K^{\mathrm{ex}} = \{0, \ldots, n\}$ and 
    the quiver associated with the principal part of $\tilde{B}$ is 
    an acyclic Dynkin quiver of affine type. 
    Let $\mathscr{A}(\mathcal{S})$ be the corresponding cluster algebra and 
    $Q$ be the associated affine root lattice. 
    For a cluster variable $y$ that is not one of the frozen variables 
    $\{f_j\}_{j\in K^{\text{fr}}}$, 
    Proposition~\ref{Prop : Laurent} implies that 
    there exists a unique sequence ${\mathbf{d}} = (d_0, \ldots, d_n)\in \Z^{n+1}$ 
    and a polynomial $P(x_i, f_j\mid i \in K^{\mathrm{ex}}, j\in K^{\text{fr}})$ such that : 
    \begin{itemize}
        \item $y = x^{-{\bf{d}}}\times P(x_i, f_j)$, 
        \item $P(x_i, f_j)$ is not divisible by any $x_i$. 
    \end{itemize}
    Then we define the \emph{$\mathbf{d}$-vector} of $y$ by 
    \[
    \mathbf{d}(y) \coloneq \sum_{i = 0}^{n}d_i \alpha_i \in Q. 
    \]
\end{Def}
For example, for the cluster variable $x_i$, we have $\mathbf{d}(x_i) = -\alpha_i$.  

For a cluster algebra of finite type, it is well known that there is a bijection between 
the set of cluster variables and the set of almost positive roots 
through $\mathbf{d}$-vector (see \cite{FZ03}). 
For a cluster algebra of affine type, there is the following analogous fact. 
\begin{Prop}[{\cite[Theorem 1.2]{affine}}]\label{Prop : dvectorbij}
    Let $\mathcal{S} = (\{x_i\}_{i \in K^{\text{ex}}}\cup \{f_j\}_{j \in K^{\text{fr}}}, \tilde{B})$ be an acyclic seed of untwisted affine type 
    and $\mathscr{A}(\mathcal{S})$ be the associated cluster algebra. 
    Let $Q$ be the corresponding root lattice, $W$ be the Weyl group, and 
    $c\in W$ be a Coxeter element associated with $\tilde{B}$. 
    Then, $\mathbf{d}$-vectors provide a bijection between 
    $\{\text{cluster variables}\}\setminus \{f_j\mid j \in K^{\mathrm{fr}}\}$ 
    to $\Phi_c^{\mathrm{re}}$. 
\end{Prop}
For $\alpha \in \Phi_c^{\text{re}}$, 
the cluster variable whose $\mathbf{d}$-vector is $\alpha$ is denoted by $x[\alpha]$. 
\begin{Def}[{\cite{affine}}]\label{Def : c-compatible}
    Let $\Phi$ be a root system of untwisted affine type and $c$ be a Coxeter element. 
    Two distinct roots $\alpha, \beta \in \Phi_c$ are said to be \emph{$c$-compatible} 
    if they have the following properties : 
    
    \begin{itemize}
        \item If $\alpha, \beta \in \Phi_c^{\text{re}}$, 
        then the cluster variables $x[\alpha], x[\beta]$ are 
        in the same cluster. 
        \item If one of the roots is $\delta$, 
        then the other belongs to  $\Lambda_c^{\text{re}}$. 
    \end{itemize}
\end{Def}
\begin{Rem}
    In \cite{affine}, the authors define $c$-compatibility 
    by using $c$-compatibility degree, and 
    Definition~\ref{Def : c-compatible} is derived as a consequence. 
    See \cite[\S 4, 5]{affine}. 
\end{Rem}

\begin{Prop}[{\cite[Theorem 5.5]{affine}}]\label{Prop : pairwise compatible}
    Let $\alpha_1, \ldots, \alpha_k\in \Phi_c^{\text{re}}$. 
    If they are mutually $c$-compatible, 
    then cluster variables 
    $x[\alpha_1], \ldots, x[\alpha_k]$ are in the same cluster. 
\end{Prop}
\begin{Def}[{\cite[Definition 6.1]{affine}}]\label{Def : cluster expansion}
    Let $\Phi$ be a root system of untwisted affine type and $c$ be a Coxeter element. 
    For an element $\gamma$ of the root lattice $Q$, 
    a \emph{$c$-cluster expansion} of $\gamma$ is an expression 
    \[
    \gamma = \sum_{\alpha \in \Phi_c}m_{\alpha}\alpha
    \]
    where the $m_{\alpha}$ are 
    nonnegative integers with $m_{\alpha}m_{\beta} = 0$ 
    whenever $\alpha$ and $\beta$ are distinct and not $c$-compatible. 
\end{Def}
\begin{Prop}[{\cite[Theorem 6.2]{affine}}]\label{Prop : cluster expansion}
    Let $\Phi$ be a root system of untwisted affine type and $c$ be a Coxeter element. 
    Then, every element $\gamma$ in $Q$ admits a unique $c$-cluster expansion. 
\end{Prop}
For a $c$-cluster expansion $\gamma = \sum_{\alpha \in \Phi_c}m_{\alpha}\alpha$, 
we call it a \emph{real $c$-cluster expansion} if $m_{\delta}=0$. 
\begin{Rem}
    In \cite{affine}, the cluster expansion is defined 
    over $Q \otimes_{\mathbb{Z}} \mathbb{R}$ instead of $Q$, 
    allowing the coefficients $m_\alpha$ to range 
    over the nonnegative real numbers rather than nonnegative integers. 
    However, from the proof given in \cite{affine}, 
    one can see that Proposition~\ref{Prop : cluster expansion} still holds 
    under the setting of Definition~\ref{Def : cluster expansion}. 
\end{Rem}
\begin{Prop}[{\cite[Proposition 6.8, Proposition 6.2]{affine}}]
\label{Prop : imaginary cluster}
    Let $\gamma \in Q$ and $\gamma = \sum_{\alpha \in \Phi_c}m_{\alpha}\alpha$ 
    be the cluster expansion. 
    If $\gamma \in \sum_{\alpha \in \Lambda_c}\Z_{\geq 0}\alpha$, 
    then $m_{\alpha} = 0$ for any $\alpha \not\in \Lambda_c$. 
\end{Prop}

\begin{Prop}\label{Prop : m_delta>0}
    Let $\gamma \in Q$ and $\gamma = \sum_{\alpha \in \Phi_c}m_{\alpha}\alpha$ 
    be the cluster expansion. 
    Then, $m_{\delta} > 0$ if and only if 
    $\gamma \in \delta + \sum_{\alpha \in \Lambda_c}\Z_{\geq 0}\alpha$. 
\end{Prop}
\begin{proof}
    First, we assume $m_{\delta} > 0$. 
    By the definition of $c$-compatibility, 
    $m_{\alpha} > 0$ only if $\alpha \in \Lambda_c$. 
    Since $m_{\delta} > 0$, we obtain that 
    $\gamma -\delta \in \sum_{\alpha \in \Lambda_c}\Z_{\geq 0}\alpha$. 

    Next, we assume $\gamma \in \delta + \sum_{\alpha \in \Lambda_c}\Z_{\geq 0}\alpha$. 
    Let $\gamma - \delta = \sum_{\alpha \in \Phi_c}m^{\prime}_{\alpha}\alpha$ be the 
    cluster expansion of $\gamma - \delta$. 
    By Proposition~\ref{Prop : imaginary cluster}, 
    it can be written as 
    $\gamma - \delta = \sum_{\alpha \in \Lambda_c}m^{\prime}_{\alpha}\alpha$. 
    Since each element of $\Lambda_c\setminus\{\delta\}$ is $c$-compatible with 
    $\delta$, 
    $\gamma= (m^{\prime}_{\delta} + 1)\delta + 
    \sum_{\alpha \in \Lambda_c^{\mathrm{re}}}m^{\prime}_{\alpha}\alpha$ 
    is the cluster expansion of $\gamma$. 
\end{proof}

\begin{Prop}[{\cite[\S 7]{ClusterIV}}]\label{Prop : dvector mutation}
    Let $\mathcal{S}_0 = (\{x_i\}_{i \in K^{\text{ex}}}\cup \{f_j\}_{j \in K^{\text{fr}}}, \tilde{B}_0)$ be an acyclic seed of untwisted affine type 
    and $\mathscr{A}(\mathcal{S}_0)$ be the associated cluster algebra. 
    For a cluster variable $y \in \mathscr{A}(\mathcal{S}_0)$, 
    we write $\mathbf{d}(y) = \sum_{i=0}^{n}d_i(y)\alpha_i$. 
    For a seed $(\{y_i\}_{i \in K^{\text{ex}}}\cup \{f_j\}_{j \in K^{\text{fr}}}, 
    \tilde{B} = (b_{ij})_{(i, j) \in K\times K^{\mathrm{ex}}}$ and 
    $0 \leq j, k, l \leq n$, 
    \begin{align*}
    d_j(\mu_k(y_l)) = 
    \begin{cases}
    d_j(y_l) &\text{ if } k\neq l, \\
    -d_j(y_k) + \mathrm{max}(\sum_{i, b_{ik}>0}b_{ik}d_j(y_i), \sum_{i, b_{ik}<0}-b_{ik}d_j(y_i)
    &\text{ if } k=l. 
    \end{cases}
    \end{align*}
\end{Prop}

In the following sections, we consider the setting that 
a quiver $B$ is a Dynkin quiver of type $D_n^{(1)}, E_n^{(1)}$ and 
of sink-source type i.e. each node of the quiver is a sink or a source. 
For a Dynkin diagram of these types, 
there are two choices of the orientation of the sink-source quiver. 
The corresponding Coxeter elements are inverses of each other. 
Since we have $\Lambda_c^{\text{re}} = \Lambda_{c^{-1}}^{\text{re}}$ for a Coxeter element $c$, 
the set $\Lambda_c^{\mathrm{re}}$ is independent of the choice of the orientation 
of the sink-source quiver.
We give the explicit description of $\Lambda_{c}^{\text{re}}$. 

We fix the label of simple roots of $D_n^{(1)}\ (n\geq 4)$ as 

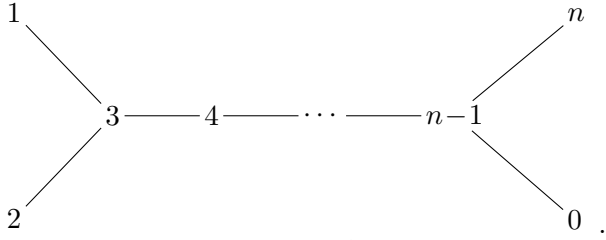
\begin{tikzpicture}[every node/.style={inner sep=1.5pt}]
  \node (1) at (-30, 0){1};
  \node (3) [below right=of 1] {3};
  \node (2) [below left=of 3] {2};
  \node (4) [right=of 3] {4};
  \node (5) [right=of 4] {$\cdots$};
  \node (6) [right=of 5] {$n\!-\!1$};
  \node (7) [above right=of 6]{$n$};
  \node (8) [below right=of 6]{$0$};

  \draw (1)--(3)--(4)--(5)--(6)--(7);
  \draw (2)--(3);
  \draw (6)--(8);
\end{tikzpicture}
. 

Then $\delta = \alpha_1 +\alpha_2 +2\sum_{i=3}^{n-1}\alpha_i +\alpha_n+\alpha_0$. 

\begin{Prop}[{\cite{affine}}]\label{Prop : D_n Lambda_c}
    Let $B$ be an acyclic Dynkin quiver of type $D_n^{(1)}$ 
    and assume that $B$ is of sink-source type. 
    Then, $\Lambda_{c}^{\text{re}}$ can be described as a disjoint union of 
    three components 
    \begin{align*}
        I_1 &= \{\sum_{i=k}^{l}\beta_i, \delta-\sum_{i=k}^{l}\beta_i\mid 
        1 \leq k\leq l\leq n-3\}, \\
        I_2 &= \{\beta_{n-2}, \delta-\beta_{n-2}\}, \\
        I_3 &= \{\beta_{n-1}, \delta-\beta_{n-1}\}, 
    \end{align*}
    where 
    \begin{align*}
        \beta_i &= \alpha_{n-2i}+\alpha_{n+1-2i} \quad (1 \leq i<[(n-1)/2]), \\
        \beta_{[(n-1)/2]} &= \alpha_1+\alpha_2+\alpha_3, \\
        \beta_j &= \alpha_{2j+3-n}+\alpha_{2j+4-n}\quad ([(n-1)/2]<j\leq n-3), \\
        \beta_{n-2} &= \alpha_1+\alpha_3+\alpha_4+\cdots +\alpha_{n}, \\
        \beta_{n-1} &= \alpha_2+\alpha_3+\alpha_4+\cdots +\alpha_{n}. \\
    \end{align*}
    If $\gamma_1, \gamma_2\in \Lambda_{c}^{\text{re}}$ are in distinct components, 
    then they are $c$-compatible. 
\end{Prop}

In this paper, we fix the label of simple roots of $E_6^{(1)}$ as 

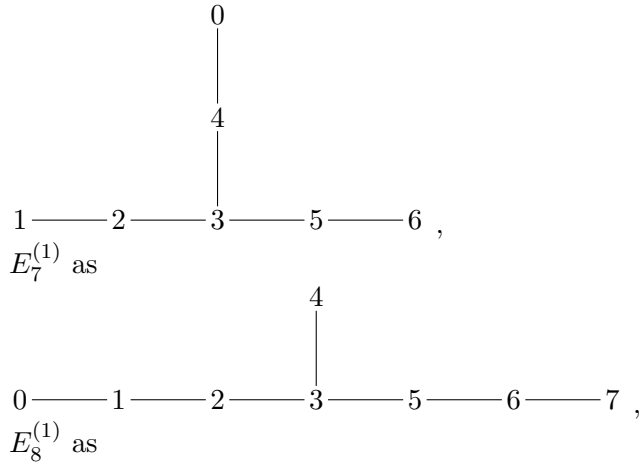
\begin{tikzpicture}[every node/.style={inner sep=1.5pt}]
  \node (1) at (-30, 0){1};
  \node (2) [right=of 1] {2};
  \node (3) [right=of 2] {3};
  \node (4) [above=of 3] {4};
  \node (5) [right=of 3] {5};
  \node (6) [right=of 5] {6};
  \node (7) [above=of 4]{0};

  \draw (1)--(2)--(3)--(5)--(6);
  \draw (3)--(4)--(7);
\end{tikzpicture}
, 

$E_7^{(1)}$ as 

\begin{tikzpicture}[every node/.style={inner sep=1.5pt}]
  \node (0) at (-30, 0){0};
  \node (1) [right=of 0] {1};
  \node (2) [right=of 1] {2};
  \node (3) [right=of 2] {3};
  \node (4) [above=of 3] {4};
  \node (5) [right=of 3] {5};
  \node (6) [right=of 5] {6};
  \node (7) [right=of 6] {7};

  \draw (0)--(1)--(2)--(3)--(5)--(6)--(7);
  \draw (3)--(4);
\end{tikzpicture}
, 

$E_8^{(1)}$ as 

\begin{tikzpicture}[every node/.style={inner sep=1.5pt}]
  \node (1) at (-30, 0){1};
  \node (2) [right=of 1] {2};
  \node (3) [right=of 2] {3};
  \node (4) [above=of 3] {4};
  \node (5) [right=of 3] {5};
  \node (6) [right=of 5] {6};
  \node (7) [right=of 6] {7};
  \node (8) [right=of 7] {8};
  \node (0) [right=of 8] {0};

  \draw (1)--(2)--(3)--(5)--(6)--(7)--(8)--(0);
  \draw (3)--(4);
\end{tikzpicture}
. 

\begin{Prop}[{\cite{affine}}]\label{Prop : E_n Lambda_c}
    Let $B$ be an acyclic Dynkin quiver of type $E_n^{(1)}$ 
    and assume that $B$ is of sink-source type. 
        Then, $\Lambda_{c}^{\text{re}}$ can be described as a disjoint union of 
    three components 
    \begin{align*}
        I_1 &= \{\sum_{i=k}^{l}\beta_i, \delta-\sum_{i=k}^{l}\beta_i\mid 
        1 \leq k\leq l\leq n-4\}, \\
        I_2 &= \{\sum_{i=k}^{l}\beta_i, \delta-\sum_{i=k}^{l}\beta_i\mid 
        n-3 \leq k\leq l\leq n-2\}, \\
        I_3 &= \{\beta_{n-1}, \delta-\beta_{n-1}\}, 
    \end{align*}
    where 
    \begin{itemize}
        \item $n=6$
    
    \begin{align*}
        \beta_1 &= \alpha_2+\alpha_3+\alpha_5+\alpha_6, \\
        \beta_2 &= \alpha_1+\alpha_2+\alpha_3+\alpha_4, \\
        \beta_3 &= \alpha_1+\alpha_2+\alpha_3+\alpha_5, \\
        \beta_4 &= \alpha_3+\alpha_4+\alpha_5+\alpha_6, \\
        \beta_5 &= \alpha_2+2\alpha_3+\alpha_4+\alpha_5, \\
    \end{align*}

    \item $n=7$

    \begin{align*}
        \beta_1 &= \alpha_1+\alpha_2+\alpha_3+\alpha_5, \\
        \beta_2 &= \alpha_3+\alpha_4+\alpha_5+\alpha_6+\alpha_7, \\
        \beta_3 &= \alpha_2+\alpha_3+\alpha_5+\alpha_6, \\
        \beta_4 &= \sum_{i=1}^6\alpha_i, \\
        \beta_5 &= \alpha_2+2\alpha_3+\alpha_4+\alpha_5, \\
        \beta_6 &= 
        \alpha_1+2\alpha_2+2\alpha_3+\alpha_4+\alpha_5+\alpha_6+\alpha_7, \\
    \end{align*}

    \item $n=8$

    \begin{align*}
        \beta_1 &= \alpha_1+\alpha_2+\alpha_3+\alpha_5+\alpha_6+\alpha_7, \\
        \beta_2 &= \alpha_2+2\alpha_3+\alpha_4+\alpha_5, \\
        \beta_3 &= \sum_{i=1}^6\alpha_i, \\
        \beta_4 &= \alpha_2+\alpha_3+\alpha_5+\alpha_6+\alpha_7+\alpha_8, \\
        \beta_5 &= \alpha_2+2\alpha_3+\alpha_4+2\alpha_5+2\alpha_6+\alpha_7, \\
        \beta_6 &= \alpha_1+\alpha_2+2\alpha_3+\alpha_4+2\alpha_5+
        \alpha_6+\alpha_7+\alpha_8, \\
        \beta_7 &= \alpha_1+2\alpha_2+3\alpha_3+2\alpha_4+2\alpha_5+2\alpha_6+2\alpha_7+
        \alpha_8. 
    \end{align*}
    \end{itemize}
    If $\gamma_1, \gamma_2\in \Lambda_{c}^{\text{re}}$ are in distinct components, 
    then they are $c$-compatible. 
    
\end{Prop}

\subsection{Some lemmas in the context of monoidal categorification}
Let $\mathscr{A} = \mathscr{A}(\mathcal{S}_0)$ be a cluster algebra, 
where 
$\mathcal{S}_0 = (\{x_i\}_{i \in K^{\text{ex}}}\cup \{f_j\}_{j \in K^{\text{fr}}}, 
\tilde{B}_0 = (b_{ij})_{(i, j)\in K\times K^{\text{ex}}})$ 
is an acyclic seed of untwisted affine type. 
Let $\mathcal{C}$ be a monoidal subcategory of $\mathscr{C}_{\mathfrak{g}}$ 
which is a monoidal categorification of $\mathscr{A}$, 
and $\iota : K(\mathcal{C}) \overset{\sim}{\longrightarrow}\mathscr{A}$ be 
the isomorphism. 
Let $\mathcal{M}$ be a set of highest weight monomials 
which parametrizes $\Irr \mathcal{C}$. 
We write the highest weight monomials corresponding to initial cluster variables 
$x_i, f_j$ by $M_i, F_j$ respectively.

\begin{Lem}[{\cite[Proposition 2.2]{HL10}}]\label{Lem : positivity}
    For $m \in \mathcal{M}$, 
    $\iota([L(m)])$ is a Laurent polynomial with non-negative coefficients. 
\end{Lem}
\begin{proof}
    For $m \in \mathcal{M}$, by Proposition~\ref{Prop : Laurent}, we can write as 
    \[
    [L(m)] = \frac{P(x_i, f_j)}
    {\prod_{i \in K^{\text{ex}}}x_i^{d_i}}
    \]
    for some $d_i\in \Z_{\geq 0}$, where 
    $P = \sum_{\mathbf{a, b}} c_{{\mathbf{a, b}}}x^{{\mathbf{a}}}f^{{\mathbf{b}}}$ is a polynomial 
    in $x_i$ ($i\in K^{\text{ex}}$) and $f_j$ ($j \in K^{\text{fr}}$). 
    Here ${\mathbf{a}} \in \Z_{\geq 0}^{K^{\text{ex}}}, 
    {\mathbf{b}} \in \Z_{\geq 0}^{K^{\text{fr}}}$, 
    and $c_{{\mathbf{a, b}}} \in \Z$. 
    It is enough to show that $c_{{\mathbf{a, b}}}$ 
    is non-negative for each $\mathbf{a}\in \Z_{\geq 0}^{K^{\text{ex}}}$ 
    and $\mathbf{b}\in \Z_{\geq 0}^{K^{\text{fr}}}$. 
    Multiplying both sides by the denominator, we see that $P$ is equal to 
    \[
    \biggl[L(m)\otimes  \bigotimes_{i\in K^{\text{ex}}}L(M_i)^{\otimes{d_i}} \biggr]. 
    \]
    Moreover, since 
    $\{x_i\}_{i\in K^{\text{ex}}}\cup \{f_j\}_{j \in K^{\text{fr}}}$ 
    is a cluster, 
    $x^{{\mathbf{a}}}f^{{\mathbf{b}}}$ is a class of simple modules 
    for each $\mathbf{a}\in \Z_{\geq 0}^{K^{\text{ex}}}$ and 
    $\mathbf{b}\in \Z_{\geq 0}^{j \in K^{\text{fr}}}$. 
    This implies that $c_{{\mathbf{a, b}}}$ is equal to the multiplicity of 
    the simple module $x^{{\mathbf{a}}}f^{{\mathbf{b}}}$ 
    as a composition factor of the module $P$, so it is non-negative. 
\end{proof}
\begin{Lem}\label{Lem : dvector inequality}
    For $l \in \Z_{\geq0}$ and $m_1, \ldots, m_l\in \mathcal{M}$, 
    \[
    \sum_{k=1}^{l}\mathbf{d}\circ\iota([L(m_k)]) \geq 
    \mathbf{d}\circ \iota\biggl(\biggl[L\biggl(\prod_{k=1}^{l}m_k\biggr)\biggr]\biggr). 
    \]
\end{Lem}
\begin{proof}
    By the definition of $\mathbf{d}$-vector, 
    $(LHS) = \mathbf{d}\circ\iota([\bigotimes_{k=1}^{l}L(m_k)])$. 
    By Proposition~\ref{Prop : subquot}, we can write 
    $[\bigotimes_{k=1}^{l}L(m_k)] 
    = [L(\prod_{k=1}^{l}m_k)] + 
    \sum_{m \in \mathcal{M}}a_m[L(m)]$  for $a_m \in \Z_{\geq 0}$. 
    For $i \in K^{\text{ex}}$, 
    if two polynomials $P_1, P_2\in \Z[x_i, f_j]$ have non-negative coefficients and 
    they are not divisible by $x_i$, 
    then $P_1 + P_2$ is also not divisible by $x_i$. 
    Combining this fact and Lemma~\ref{Lem : positivity}, 
    we have $\mathbf{d}\circ\iota([L(\prod_{k=1}^{l}m_k)] + 
    \sum_{m \in \mathcal{M}}a_m[L(m)]) \geq \mathbf{d}\circ\iota([L(\prod_{k=1}^{l}m_k)])$. 
\end{proof}
\begin{Lem}\label{Lem : mutation monomial}
    Let $\mathcal{S}=(\{y_i = [L(m_i)]\}_{i \in K}, 
    \tilde{B} = (b_{ij}))$ be a cluster of $\mathscr{A}$. 
    Then, for $k \in K^{\text{ex}}$, 
    the $k$-th exchangeable cluster variable of 
    $\mu_k(\mathcal{S})$ is either 
    \[
    \Bigl[L\Bigl(\frac{\prod_{b_{ik}>0}m_i^{b_{ik}}}{m_k}\Bigr)\Bigr] \text{ or } 
    \Bigl[L\Bigl(\frac{\prod_{b_{ik}<0}m_i^{-b_{ik}}}{m_k}\Bigr)\Bigr]. 
    \]
    In particular, if one of 
    $\frac{\prod_{b_{ik}>0}m_i^{b_{ik}}}{m_k}, 
    \frac{\prod_{b_{ik}<0}m_i^{-b_{ik}}}{m_k}$ 
    is not in $\mathcal{M}$, 
    then the highest weight monomial of the $k$-th exchangeable cluster variable of 
    $\mu_k(\mathcal{S})$ is the other. 
\end{Lem}
\begin{proof}
    We write the highest weight monomial of 
    the $k$-th exchangeable cluster variable of $\mu_k(\mathcal{S})$ 
    as $m_k^{\prime}$. 
    By Proposition~\ref{Prop : subquot}, $L(m_k m_k^{\prime})$ is a subquotient 
    of $L(m_k)\otimes L(m_k^{\prime})$. 
    Since $[L(m_k)\otimes L(m_k^{\prime})]
    =[L(\prod_{b_{ik}>0}m_i^{b_{ik}})]+[L(\prod_{b_{ik}<0}m_i^{-b_{ik}})]$ 
    holds by Definition~\ref{Def : admissible}, 
    we have 
    \[
    m_k m_k^{\prime} = \prod_{b_{ik}>0}m_i^{b_{ik}} \text{ or }
     \prod_{b_{ik}<0}m_i^{-b_{ik}}. 
    \]
\end{proof}

\section{Classification of real modules and imaginary modules}
\label{sec : classification}
In this section, we completely classify the real modules and imaginary modules of 
the categories 
$\mathscr{C}_{D_{n}}^{[1, 2n+1], \mathfrak{s}^1}(n\geq 4), 
\mathscr{C}_{D_{4}}^{[1, 11], \mathfrak{s}^2}, 
\mathscr{C}_{E_{n}}^{[1, 2n+1], \mathfrak{s}^1}(n=6,7,8), 
$ and $\mathscr{C}_{A_{n}}^{[1, 2n+2], \mathfrak{s}^1}(n\geq 5)$ 
in the proof of Theorem~\ref{Thm : affine categorification}. 
By the classifications, we can solve Conjecture~\ref{Conj : real vs monomial intro} 
for these categories.

\subsection{$\mathscr{C}_{D_{n}}^{[1, 2n+1], \mathfrak{s}^1}(n\geq 4)$ 
(The module category of type $D_{n}$ as a categorification of 
the $D_{n}^{(1)}$-type cluster algebra)}

Let $\mathfrak{g}$ be of type $D_{n}$ and $\mathcal{C}$ be the monoidal category $\mathscr{C}_{D_{n}}^{[1, 2n+1], \mathfrak{s}^1}$ 
for $n \geq 4$, 
where the admissible sequence $\mathfrak{s}^1$ is as in Example~\ref{Ex : Cs}. 
By Theorem~\ref{Thm : affine categorification}, 
this category admits a monoidal categorification of the cluster algebra of type $D_{n}^{(1)}$. 

As the initial monoidal seed, 
we take the seed constructed by Theorem~\ref{Thm : affine categorification}. 
By Lemma~\ref{Lem : mutation monomial}, 
this seed is explicitly given as follows : 
\begin{align*}
&n=4\\
&\raisebox{2.3em}{\scalebox{1}
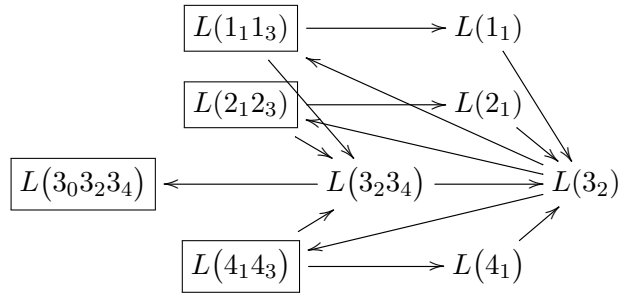
{\xymatrix@C=0.7ex@R=1ex{
 && \boxed{L(1_1 1_3)}\ar[rr]\ar[ddr]  && L(1_1)\ar[ddr]\\
 && \boxed{L(2_1 2_3)}\ar[rr]\ar[dr]  && L(2_1)\ar[dr]\\
 &\boxed{L\bigl(3_0 3_2 3_4\bigr)} 
&& L\bigl(3_2 3_4\bigr) \ar[ll] \ar[rr]
&&L(3_2)\ar[uulll]\ar[ulll]\ar[dlll]\\
 && \boxed{L\bigl(4_1 4_3\bigr)} \ar[ur]\ar[rr] && 
 L\bigl(4_1\bigr)\ar[ur]\\
}}}
\end{align*}
\begin{align*}
&n=2k\ (k\geq 3)\\
&\raisebox{2.3em}{\scalebox{0.7}
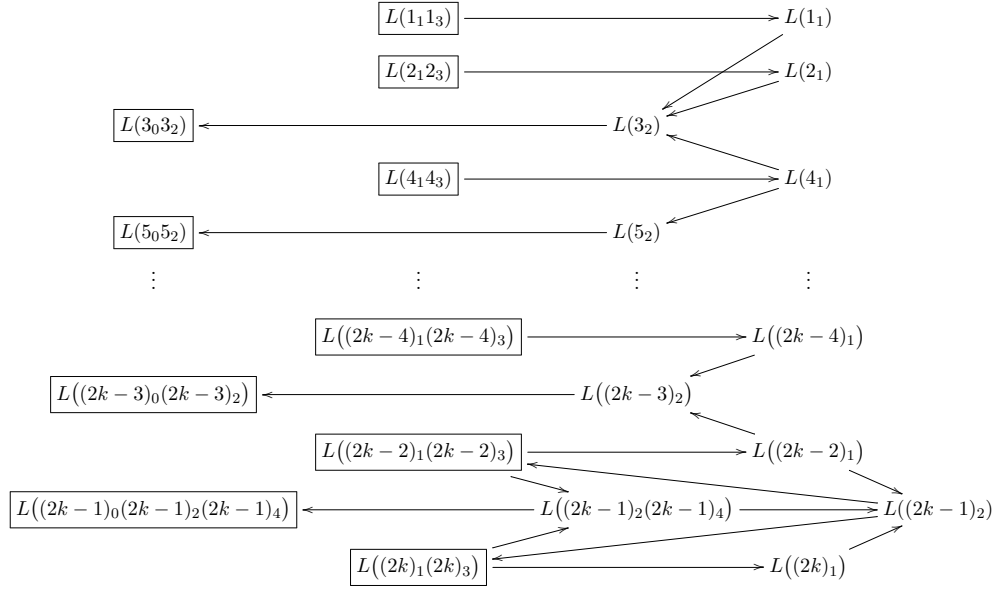
{\xymatrix@C=0.7ex@R=1ex{
 && \boxed{L(1_1 1_3)}\ar[rr]  && L(1_1)\ar[ddl]\\
 && \boxed{L(2_1 2_3)}\ar[rr]  && L(2_1)\ar[dl]\\
 &\boxed{L(3_0 3_2)} 
 && L(3_2) \ar[ll] \\
 && \boxed{L(4_1 4_3)} \ar[rr] && L(4_1)\ar[ul]\ar[dl]\\
 &\boxed{L(5_0 5_2)} 
 && L(5_2) \ar[ll]\\
 & \vdots & \vdots &\vdots&\vdots \\
 && \boxed{L\bigl((2k-4)_1 (2k-4)_3\bigr)} \ar[rr] && L\bigl((2k-4)_1\bigr)\ar[dl]\\
 &\boxed{L\bigl((2k-3)_0 (2k-3)_2 \bigr)} 
&& L\bigl((2k-3)_2\bigr) \ar[ll] \\
 && \boxed{L\bigl((2k-2)_1 (2k-2)_3\bigr)} \ar[rr]\ar[dr]
 && L\bigl((2k-2)_1\bigr)\ar[ul]\ar[dr]\\
 &\boxed{L\bigl((2k-1)_0 (2k-1)_2 (2k-1)_4\bigr)} 
&& L\bigl( (2k-1)_2 (2k-1)_4\bigr) \ar[ll] \ar[rr]
&&L( (2k-1)_2)\ar[ulll]\ar[dlll]\\
 && \boxed{L\bigl((2k)_1 (2k)_3\bigr)} \ar[ur]\ar[rr] && 
 L\bigl((2k)_1\bigr)\ar[ur]\\
}}}
\end{align*}
\begin{align*}
&n=2k-1\ (k\geq 3)\\
&\raisebox{2.3em}{\scalebox{0.7}
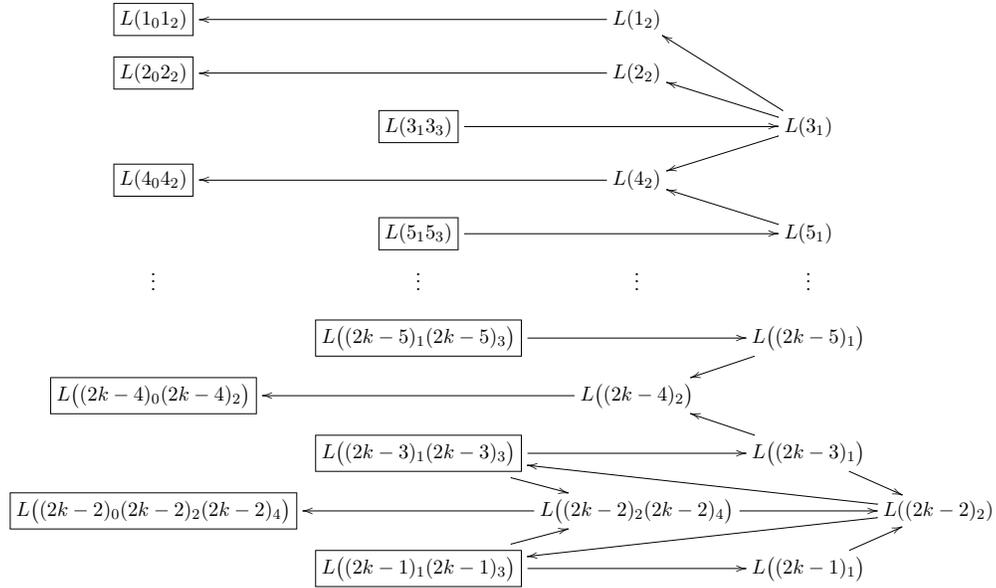
{\xymatrix@C=0.7ex@R=1ex{
 & \boxed{L(1_0 1_2)}
 && L(1_2)\ar[ll]\\
 & \boxed{L(2_0 2_2)}
 && L(2_2)\ar[ll]\\
 &&\boxed{L(3_1 3_3)} \ar[rr] && L(3_1) \ar[uul]\ar[ul]\ar[dl] \\
 & \boxed{L(4_0 4_2)} 
 && L(4_2)\ar[ll]\\
 &&\boxed{L(5_1 5_3)} \ar[rr] && L(5_1) \ar[ul]\\
 & \vdots & \vdots &\vdots&\vdots \\
 && \boxed{L\bigl((2k-5)_1 (2k-5)_3\bigr)} \ar[rr] && L\bigl((2k-5)_1\bigr)\ar[dl]\\
 &\boxed{L\bigl((2k-4)_0 (2k-4)_2 \bigr)} 
&& L\bigl((2k-4)_2\bigr) \ar[ll] \\
 && \boxed{L\bigl((2k-3)_1 (2k-3)_3\bigr)} \ar[rr]\ar[dr]
 && L\bigl((2k-3)_1\bigr)\ar[ul]\ar[dr]\\
 &\boxed{L\bigl((2k-2)_0 (2k-2)_2 (2k-2)_4\bigr)} 
&& L\bigl( (2k-2)_2 (2k-2)_4\bigr) \ar[ll] \ar[rr]
&&L( (2k-2)_2)\ar[ulll]\ar[dlll]\\
 && \boxed{L\bigl((2k-1)_1 (2k-1)_3\bigr)} \ar[ur]\ar[rr] && 
 L\bigl((2k-1)_1\bigr)\ar[ur]\\
}}}
\end{align*}
We set $I_{\text{fin}} = \{1, \ldots, n\}$ and $I_{\text{fin}} = I_0 \sqcup I_1$, 
where 
\[I_0 \coloneq 
\begin{cases}
    \{1, 2, 4, 6, \ldots, 2k-2, 2k\} \quad \text{for }n=2k\\
    \{3, 5, 7, \ldots, 2k-3, 2k-1\} \quad \text{for }n=2k-1
\end{cases}
\]
and 
\[
I_1 \coloneq 
\begin{cases}
    \{3, 5, 7, \ldots, 2k-3, 2k-1\} \quad \text{for }n=2k\\
    \{1, 2, 4, 6, \ldots, 2k-4, 2k-2\} \quad \text{for }n=2k-1
\end{cases}.
\]

Let $\mathcal{M} \coloneq \mathcal{M}_+^{[1, 2n+1], \mathfrak{s}^1}$ be the set of monomials in 
the indeterminates
\[
\{i_1, i_3\mid i \in I_0\}\cup 
\{j_0, j_2\mid j \in I_1\setminus \{n-1\}\}\cup 
\{(n-1)_0, (n-1)_2, (n-1)_4\}, 
\]
which parametrizes $\Irr \mathcal{C}$. 
For $i \in I_{\text{fin}}$, we define $F_i \in \mathcal{M}$ by 
\[
F_i = 
\begin{cases}
    i_1i_3 \quad &\text{if } i \in I_0\\
    i_0i_2 \quad &\text{if } i \in I_1\setminus \{n-1\}\\
    (n-1)_0(n-1)_2(n-1)_4 \quad &\text{if } i = n-1. 
\end{cases}
\]

We define the subset $\mathcal{M}^{\prime}$ of $\mathcal{M}$ by 
\[
\mathcal{M}^{\prime} \coloneq \{m \in \mathcal{M}\mid 
\text{
$m$ is not divisible by every $F_i$ for $i \in I_{\text{fin}}$
}
\}. 
\]
Let $K=[0, 2n], K^{\text{ex}}=[0,n]$, and $K^{\text{fr}}=[n+1, 2n]$. 
Let $\mathcal{S}_0 = (\{x_i\}_{i \in [0, 2n]}, \tilde{B}_0)$ 
be a seed, where $x_0, \ldots, x_{2n}$ are indeterminates and 
$\tilde{B}_0$ is the exchange matrix associated with the above quiver. 
We write the $n$ frozen variables $x_{n+1}, \ldots, x_{2n}$ 
as $f_1, \ldots, f_{n}$, respectively.
Then by Theorem~\ref{Thm : affine categorification}, there is an isomorphism 
\begin{align*}
  \iota : K(\mathcal{C}) \overset{\sim}{\longrightarrow}\mathscr{A} \quad; 
  &\quad [L(i_1)] \mapsto x_i  &&\text{ for } i \in I_0\\
  &\quad [L(j_2)] \mapsto x_j &&\text{ for } j \in I_1\\
  &\quad [L\bigl((n-1)_2(n-1)_4\bigr)] \mapsto x_0 &\\
  &\quad [L(F_i)] \mapsto f_i &&\text{ for } i \in I_{\text{fin}}. 
\end{align*}
We sometimes identify $K(\mathcal{C})$ with $\mathscr{A}$ by $\iota$. 

First, we combinatorially determine the $\mathbf{d}$-vector of each simple representation 
in $\mathscr{A}$.
Let $Q = \bigoplus_{i=0}^{n}\Z\alpha_i$ be the root lattice of $D_{n}^{(1)}$. 

For each $m \in \mathcal{M}$, there is a unique expression of the form
\begin{align*}
    m = &\prod_{i\in I_{\text{fin}}}F_i^{a_i} \times \prod_{i \in I_0}i_1^{p_i}i_3^{q_i} \times 
    \prod_{j \in I_1\setminus \{n-1\}}j_2^{p_j}j_0^{q_j} \\
    &\times \bigl((n-1)_2(n-1)_4\bigr)^{p}\bigl((n-1)_0(n-1)_4\bigr)^{q}
    \bigl((n-1)_0(n-1)_2\bigr)^{r}\\
    &\times (n-1)_0^s (n-1)_2^t (n-1)_4^u, 
\end{align*}
such that each of the following tuples contains at most one nonzero entry :  

\vspace{0.5em}
$(p_i, q_i)$ for each $i \in I_{\text{fin}}\setminus \{n\!-\!1\}$, 
$(p, q, r)$, $(s, t, u)$, $(p, s)$, $(q, t)$, $(r, u)$. 

\vspace{0.5em}

Then, we define a map $G : \mathcal{M} \to Q$ by 
\begin{itemize}
    \item $n=4$
    \begin{align*}
    G(m) = &\sum_{i=1,2,4}
    \bigl(p_i(-\alpha_i) + q_i\alpha_i\bigr) \\
    &+ p(-\alpha_0) 
    + q(\alpha_0 + \alpha_1+\alpha_2+2\alpha_3+\alpha_4) 
    + r\alpha_0\\
    &+ s(\alpha_0 + \alpha_3) + t(-\alpha_3) 
    + u(\alpha_1 + \alpha_2 + \alpha_3+\alpha_4), 
\end{align*}
\item $n>4$
\begin{align*}
    G(m) = &\sum_{i \in I_{\text{fin}}\setminus\{n\!-\!1\}}
    \bigl(p_i(-\alpha_i) + q_i\alpha_i\bigr) \\
    &+ p(-\alpha_0) + q(\alpha_0 + \alpha_{n-2} + 2\alpha_{n-1} + \alpha_{n}) 
    + r\alpha_0\\
    &+ s(\alpha_0 + \alpha_{n-1}) + t(-\alpha_{n-1}) 
    + u(\alpha_{n-2} + \alpha_{n-1} + \alpha_{n}). 
\end{align*}
\end{itemize}

Let $G' : \mathcal{M}' \to Q$ denote the restriction of $G$ to $\mathcal{M}'$. 
By definition, for $m \in \mathcal{M}^{\prime}$, 
$G^{\prime}(m) = G(\prod_{i \in {I_{\text{fin}}}}F_i^{a_i} \times m)$ for any $a_i$. 
\begin{Ex}
For $n>4$, 
    \begin{align*}
    G\bigl(1_1 1_3^4 2_1^3 (n-1)_0(n-1)_2^2(n-1)_4^3\bigr) 
    &= G\bigl(F_1 F_{n-1} \times 1_3^3 2_1^3 \times (n-1)_2(n-1)_4 \times (n-1)_4\bigr)\\
    &=3\alpha_1 + 3(-\alpha_2) + (-\alpha_0) + (\alpha_{n-2} + \alpha_{n-1} + \alpha_{n}). 
    \end{align*}
\end{Ex}
In this subsection, for $i \in I_{\mathrm{fin}}\setminus\{n-1\}$ and $a \in \Z$, we write 
    \[
    i^a \coloneq 
   \begin{cases}
        i_3^a \quad &\text{ if } i\in I_0, a\geq 0, \\
        i_1^{-a} \quad &\text{ if } i \in I_0, a < 0, \\
        i_0^a \quad &\text{ if } i \in I_1\setminus\{n\!-\!1\}, a \geq 0, \\
        i_2^{-a} \quad &\text{ if } i\in I_1\setminus\{n\!-\!1\}, a < 0. 
    \end{cases}
    \]
    for simplicity. 
    Then, 
    $G^{\prime}\biggl(\prod_{i \in I_{\text{fin}}\setminus\{n-1\}}i^{a_i}\biggr) 
    = \sum_{i \in I_{\text{fin}}\setminus\{n-1\}}a_i \alpha_i$ for $a_i \in \Z$. 

    For $\gamma = \sum_{i=0}^{n}a_i \alpha_i \in Q$, 
    we set a monomial $m_1 \in \mathcal{M}^{\prime}$ as 
    \[
m_1 \coloneq 
\begin{cases}
    ((n-1)_0 (n-1)_2)^{a_0-a_{n-1}}(n-1)_0^{a_{n-1}} \quad &\text{if } (a_0, a_{n-1})\in A,\\
    ((n-1)_0 (n-1)_4)^{a_{n-1}-a_0}(n-1)_0^{2a_0-a_{n-1}} \quad &\text{if } (a_0, a_{n-1})\in B,\\
    ((n-1)_0 (n-1)_4)^{a_0}(n-1)_4^{a_{n-1}-2a_0} \quad &\text{if } (a_0, a_{n-1})\in C,\\
    ((n-1)_2 (n-1)_4)^{-a_0}(n-1)_4^{a_{n-1}} \quad &\text{if } (a_0, a_{n-1})\in D,\\
    ((n-1)_2 (n-1)_4)^{-a_0}(n-1)_2^{-a_{n-1}} \quad &\text{if } (a_0, a_{n-1})\in E,\\
    ((n-1)_0 (n-1)_2)^{a_0}(n-1)_2^{-a_{n-1}} \quad &\text{if } (a_0, a_{n-1})\in F, \\
\end{cases}
\]
where $A, \ldots, F$ are as follows$\colon$

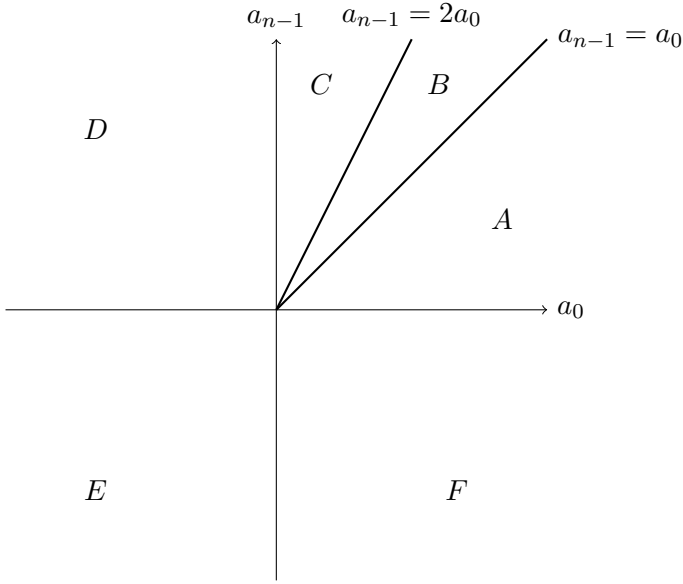
\begin{tikzpicture}[scale=1.2]

  \draw[->] (-3,0) -- (3,0) node[right] {$a_0$};
  \draw[->] (0,-3) -- (0,3) node[above] {$a_{n-1}$};

  \draw[thick]   (0,0) -- (1.5,3) node[above] {$a_{n-1}=2a_0$};
  \draw[thick]  (0, 0)  -- (3,3) node[right] {$a_{n-1}=a_0$};

  \node at (2.5,1){$A$};
  \node at (1.8,2.5) {$B$};
  \node at (0.5,2.5) {$C$};
  \node at (-2,2) {$D$};
  \node at (-2,-2) {$E$};
  \node at (2, -2) {$F$};

\end{tikzpicture}

Then, we can check that $\gamma$ can be written in the form 
\[
\gamma = G^{\prime}(m_1) + 
\sum_{\substack{i \in \{0, 1, \ldots, n\}\\i \neq 0, n-1}}b_i \alpha_i 
= G^{\prime}(m_1) + 
\sum_{\substack{i \in I_{\text{fin}}\\i \neq n-1}}b_i \alpha_i,  
\]
where $b_i \in \Z$ for $i \in I_{\text{fin}}\setminus\{n\!-\!1\}$. 
Then, by construction, the map 
\[
\sum_{i=0}^{n}a_i \alpha_i \longmapsto m_1 \times 
\prod_{\substack{i \in I_{\text{fin}}\\i \neq n-1}}i^{b_i}
\]
is the inverse map of $G^{\prime}$. 
As a result, the following Proposition holds. 

\begin{Prop}\label{Prop : D_n bij}
    $G' : \mathcal{M}' \to Q$ is bijective. 
\end{Prop}

Since every element of $\mathscr{A}$ is a Laurent polynomial, 
we can define the $\mathbf{d}$-vector of any $y \in \mathscr{A}\setminus\{0\}$, 
not necessarily a cluster variable, in the same way as in Definition~\ref{Def : dvector}. 
In particular, we can consider the map 
\[
\mathbf{d}\circ \iota \circ [L(-)] : \mathcal{M} \to Q. 
\]
\begin{Thm}\label{Thm : D_n G=dL}
    $G = \mathbf{d}\circ \iota \circ [L(-)]$. 
\end{Thm}

\begin{proof}
Note that for $m\in \mathcal{M}$ and $i \in I_{\text{fin}}$, 
the equality $G(m) = G(F_i m)$ holds by the definition of $G$, and 
\begin{align*}
    \mathbf{d}\circ \iota \circ [L(m)] &= \mathbf{d}(f_i\iota([(L(m)]))\\
    &= \mathbf{d}\circ\iota([L(F_i)][L(m)])\\
    &= \mathbf{d}\circ\iota([L(F_im)])
\end{align*}
holds by Proposition~\ref{Prop : frozenKR}. 
We can check that the statement holds 
for the case $m = i^{\pm 1}$ for $i \in I_{\text{fin}}\setminus\{n-1\}$ by considering 
the initial seed $\mathcal{S}_0$ and $\mu_i(\mathcal{S}_0)$. 

For $m \in \mathcal{M}$ and $a_i \in \Z$ for $i \in I_{\text{fin}}\setminus\{n-1\}$, 
we have 
\[
\mathbf{d}\circ \iota \circ [L(m \prod_{\substack{i \in I_{\text{fin}}\\i \neq n-1}}i^{a_i})]
\leq \mathbf{d}\circ \iota \circ [L(m)] 
+ \sum_{\substack{i \in I_{\text{fin}}\\i \neq n-1}}a_i\alpha_i
\]
and 
\[
\mathbf{d}\circ \iota \circ [L(m \prod_{\substack{i \in I_{\text{fin}}\\i \neq n-1}}i^{a_i})]
+\sum_{\substack{i \in I_{\text{fin}}\\i \neq n-1}}a_i(-\alpha_i)
\geq \mathbf{d}\circ \iota \circ 
[L(m \prod_{\substack{i \in I_{\text{fin}}\\i \neq n-1}}F_i^{a_i})]
=\mathbf{d}\circ \iota \circ [L(m)]
\]
by Lemma~\ref{Lem : dvector inequality}. 
This shows 
\[
\mathbf{d}\circ \iota \circ [L(m \prod_{\substack{i \in I_{\text{fin}}\\i \neq n-1}}i^{a_i})]
=\mathbf{d}\circ \iota \circ [L(m)]
+\sum_{\substack{i \in I_{\text{fin}}\\i \neq n-1}}a_i\alpha_i. 
\]
By the definition of $G$, 
for $m \in \mathcal{M}$ and $a_i \in \Z$ for $i \in I_{\text{fin}}\setminus\{n-1\}$, 
\[G(m \prod_{\substack{i \in I_{\text{fin}}\\i \neq n-1}}i^{a_i}) = 
G(m) +\sum_{\substack{i \in I_{\text{fin}}\\i \neq n-1}}a_i\alpha_i
\]
also holds.

By the above fact, it is enough to show that 
for any $a, b, c \in \Z_{\geq 0}$, 
there exist integers $a_i \in \Z$ for $i \in I_{\text{fin}}\setminus\{n-1\}$ 
such that 
the claim of this theorem holds for 
\[m = (n-1)_0^{a}(n-1)_2^{b}(n-1)_4^{c}
\prod_{\substack{i \in I_{\text{fin}}\\i \neq n-1}}i^{a_i}. 
\]

By Lemma~\ref{Lem : mutation monomial}, we have that the following simple modules are in the 
same monoidal cluster : 
\begin{align*}
\underline{n=4}\\
\mathcal{S}_0 &\ni& &3_2,& &3_2 3_4,& &3_0 3_2 3_4, \\
\mu_0(\mathcal{S}_0) &\ni& &3_2,& &3_0 3_2,& &3_0 3_2 3_4, \\
\mu_3(\mathcal{S}_0) &\ni& &3_4 1_1 2_1 4_1,& &3_2 3_4,& &3_0 3_2 3_4, \\
\mu_3\mu_0(\mathcal{S}_0) &\ni& &3_0,& &3_0 3_2,& &3_0 3_2 3_4, \\
\mu_0\mu_3(\mathcal{S}_0) &\ni& &3_4 1_1 2_1 4_1,& &3_0,& &3_0 3_2 3_4, \\
\end{align*}
\begin{align*}
\underline{n>4}\\
\mathcal{S}_0 &\ni& &(n\!-\!1)_2,& &(n\!-\!1)_2 (n\!-\!1)_4,& 
&(n\!-\!1)_0 (n\!-\!1)_2 (n\!-\!1)_4, \\
\mu_{0}(\mathcal{S}_0) &\ni& &(n\!-\!1)_2,& &(n\!-\!1)_0 (n\!-\!1)_2,& 
&(n\!-\!1)_0 (n\!-\!1)_2 (n\!-\!1)_4, \\
\mu_{n\!-\!1}(\mathcal{S}_0) &\ni& &(n\!-\!1)_4 (n\!-\!2)_1 n_1,& &(n\!-\!1)_2 (n\!-\!1)_4,& 
&(n\!-\!1)_0 (n\!-\!1)_2 (n\!-\!1)_4, \\
\mu_{n\!-\!1}\mu_{0}(\mathcal{S}_0) &\ni& &(n\!-\!1)_0,& &(n\!-\!1)_0 (n\!-\!1)_2,& 
&(n\!-\!1)_0 (n\!-\!1)_2 (n\!-\!1)_4, \\
\mu_{0}\mu_{n\!-\!1}(\mathcal{S}_0) &\ni& &(n\!-\!1)_4 (n\!-\!2)_1 n_1,& &(n\!-\!1)_0,& 
&(n\!-\!1)_0 (n\!-\!1)_2 (n\!-\!1)_4. \\
\end{align*}

By Proposition~\ref{Prop : dvector mutation}, 
we can calculate $\mathbf{d}$-vectors and 
we can check that the claim of the theorem holds for 
the above simple modules. 
Since modules in the same cluster strongly commute mutually, 
the claim of the theorem also holds for 
the simple modules corresponding to a cluster monomial of the above clusters. 

By the above process, for any $a, b, c\in \Z_{\geq 0}$, 
we can realize a simple module whose form is 
\[
L((n-1)_0^{a}(n-1)_2^{b}(n-1)_4^{c}
\prod_{\substack{i \in I_{\text{fin}}\\i \neq n-1}}i^{a_i})
\]
for some $a_i \in \Z$ as a cluster monomial. 
This completes the proof.


%
\end{proof}
For simplicity, 
for $\beta \in Q$, we write $L(\beta) \coloneq L({G^{\prime}}^{-1}(\beta))$. 
Then, we have that the $\mathbf{d}$-vector of $\iota([L(\beta)])$ is $\beta$. 

\begin{Prop}\label{Prop : D_n delta imaginary}
For $l \in \Z_{>0}$, the irreducible module $L(l\delta)$ is imaginary. 
\end{Prop}
\begin{proof}
By Proposition~\ref{Prop :  -alpha_i in Phi}, 
$k\delta -\alpha_1\in \Phi^{\text{re}}_{c}$ for any $k \in \Z_{\geq 0}$. 
Assume that $L(l\delta)$ is real. 
For $p \in \Z_{\geq 0}$, if $\fd(L(l\delta), L(p\delta-\alpha_1))>0$ 
then 
\[
\fd(L(l\delta), 
L({G^{\prime}}^{-1}(l\delta){G^{\prime}}^{-1}(p\delta-\alpha_1))) < 
\fd(L(l\delta), 
L(p\delta-\alpha_1))
\]
by Proposition~\ref{Prop : subquot} and ~\ref{Prop : delta inequality real}. 
The description of ${G^{\prime}}^{-1}$ shows that 
${G^{\prime}}^{-1}(p\delta-\alpha_1){G^{\prime}}^{-1}(l\delta) 
= {G^{\prime}}^{-1}((p+l)\delta-\alpha_1)$, 
so we have 
\[
\fd(L(l\delta), L((p+l)\delta-\alpha_1)) < 
\fd(L(l\delta), L(p\delta-\alpha_1)). 
\]
By repeating this argument, we can show that there is some $N \in \Z_{>0}$ such that 
$\fd(L(l\delta), L(N\delta-\alpha_1)) = 0$. 
Since $L(l\delta)$ is real, 
Proposition~\ref{Prop : real delta} shows that 
$L(l\delta)$ and $L(N\delta-\alpha_1)$ commute strongly, 
so 
\[
L(l\delta) \otimes L(N\delta-\alpha_1)
\cong L((l+N)\delta-\alpha_1). 
\]
Since 
$(l+N)\delta -\alpha_1\in \Phi^{\text{re}}_{c}$, 
$L((l+N)\delta-\alpha_1)$ is a prime module 
by Proposition~\ref{Prop : dvectorbij} and Proposition~\ref{Prop : prime module}, 
but it contradicts the above equation. 
\end{proof}

\begin{Prop}\label{Prop : D_n compatible}
Let $\gamma_1, \gamma_2 \in \Phi_c$ be distinct roots and  
$m_1, m_2\in\Z_{\geq1}$. 
If $\gamma_1$ and $\gamma_2$ are $c$-compatible, 
then two irreducible modules $L(m_1\gamma_1)$, $L(m_2\gamma_2)$ commute strongly. 
\end{Prop}
\begin{proof}
When $\gamma_1, \gamma_2$ are both distinct from $\delta$, 
then the definition of $c$-compatibility shows that 
$L(\gamma_1)$ and $L(\gamma_2)$ are in the same monoidal cluster, 
so $L(\gamma_1)$ and $L(\gamma_2)$ commute strongly. 
Since $L(\gamma_1)$ and $L(\gamma_2)$ are both real modules, 
$L(m_1\gamma_1) \cong L(\gamma_1)^{\otimes m_1}$ and 
$L(m_2\gamma_2) \cong L(\gamma_2)^{\otimes m_2}$ hold. 
So in this case, 
$L(m_1\gamma_1)$ and $L(m_2\gamma_2)$ commute strongly. 

So, it is enough to consider the case that one of the $\gamma_1, \gamma_2$ is 
equal to $\delta$. 
Without loss of generality, 
we set that $\gamma_1 = \delta$. 
Then, $\gamma_2 \in \Lambda_c^{\text{re}}$ by the definition of 
$c$-compatibility. 

Let $I_1, I_2, I_3 \subset \Lambda_c^{\mathrm{re}}$ be subsets 
defined in Proposition~\ref{Prop : D_n Lambda_c}. 
If $\gamma_2 \not\in I_2$ in Proposition~\ref{Prop : D_n Lambda_c}, 
then $\gamma_2$ is $c$-compatible with $\beta_{n-2}, \delta - \beta_{n-2}$. 
Therefore, $\fd(L(\gamma_2), L(\beta_{n-2}))=0$ and 
$\fd(L(\gamma_2), L(\delta-\beta_{n-2}))=0$ hold. 
Since ${G^{\prime}}^{-1}(\beta_{n-2})^{m_1}
{G^{\prime}}^{-1}(\delta - \beta_{2n-2})^{m_1} 
= {G^{\prime}}^{-1}(m_1\delta)$ by the description of ${G^{\prime}}^{-1}$, 
$L(m_1\delta)$ is a subquotient of 
$L(\beta_{n-2})^{\otimes m_1}\otimes L(\delta-\beta_{n-2})^{\otimes m_1}$ and 
\[
\fd(L(\gamma_2), L(m_1\delta))\leq m_1\fd(L(\gamma_2), L(\beta_{n-2}))
+m_1\fd(L(\gamma_2), L(\delta-\beta_{n-2}))=0
\]
by Proposition~\ref{Prop : delta ineq sum}. 
It shows $L(m_1\delta)$ and $L(\gamma_2)$ commute strongly. 
Since $L(\gamma_2)$ is real, 
$L(m_2\gamma_2) \cong L(\gamma_2)^{\otimes m_2}$ holds. 
Then we have that $L(m_1\delta)$ and $L(m_2\gamma_2)$ commute strongly. 

If $\gamma_2 \in I_2$, then $\gamma_2 \not\in I_3$. 
By replacing $\beta_{n-2}$ with $\beta_{n-1}$ in the above argument, 
the statement also follows in this case.
\end{proof}

\begin{Thm}\label{Thm : D_n tensor factorization}
    Let $m \in \mathcal{M}^{\prime}$ and 
    $G^{\prime}(m) = \sum_{\alpha \in \Phi_c} m_{\alpha}\alpha$ be a 
    $c$-cluster expansion. 
    Then, 
    \begin{align*}
    L(m) &\cong \bigotimes_{\alpha \in \Phi_c}L(m_{\alpha}\alpha)\\
    &\cong \bigotimes_{\alpha \in \Phi_c^{\text{re}}}
    L(\alpha)^{\otimes m_{\alpha}} \otimes L(m_{\delta}\delta). 
    \end{align*}
\end{Thm}
\begin{proof}
The second isomorphism holds because $L(\alpha)$ is real and 
$L(m\alpha) \cong L(\alpha)^{\otimes m}$
for $\alpha \in \Phi_c^{\mathrm{re}}$ and $m \in \Z_{\geq 0}$. 
We prove the first isomorphism. 

By Proposition~\ref{Prop : D_n compatible}, 
$\bigotimes_{\alpha \in \Phi_c}L(m_{\alpha}\alpha)$ is simple. 
Then we have 
\[\bigotimes_{\alpha \in \Phi_c}L(m_{\alpha}\alpha)\cong 
L(\prod_{\alpha \in \Phi_c}{G^{\prime}}^{-1}(m_{\alpha}\alpha)).
\]
First, we show that 
$\prod_{\alpha \in \Phi_c}{G^{\prime}}^{-1}(m_{\alpha}\alpha) 
\in \mathcal{M}^{\prime}$. 
Since each frozen variable $f_i \in K(\mathcal{C})$ is an irreducible element 
and $K(\mathcal{C})$ is a polynomial ring 
by Proposition~\ref{Prop : polynomial}, 
$f_i \in K(\mathcal{C})$ is a prime element. 
Since ${G^{\prime}}^{-1}(m_{\alpha}\alpha)$ is not divided by 
each $F_i$, 
$[L(m_{\alpha}\alpha)]$ is not divided by each $f_i$. 
Then it follows that each $f_i$ does not divide 
$[L(\prod_{\alpha \in \Phi_c}{G^{\prime}}^{-1}(m_{\alpha}\alpha))]$, 
so we have $\prod_{\alpha \in \Phi_c}{G^{\prime}}^{-1}(m_{\alpha}\alpha) \in 
\mathcal{M}^{\prime}$ by Proposition~\ref{Prop : frozenKR}. 

From the fact above, 
$G^{\prime}(\prod_{\alpha \in \Phi_c}{G^{\prime}}^{-1}(m_{\alpha}\alpha))$ 
is well-defined. 
Then, we have 
\begin{align*}
    G^{\prime}(\prod_{\alpha \in \Phi_c}{G^{\prime}}^{-1}(m_{\alpha}\alpha)) 
    &= \mathbf{d}\circ \iota ([L(\prod_{\alpha \in \Phi_c}{G^{\prime}}^{-1}(m_{\alpha}\alpha))]\\
    &=\mathbf{d}\circ \iota([\bigotimes_{\alpha \in \Phi_c}L(m_{\alpha}\alpha)]\\
    &=\sum_{\alpha \in \Phi_c}\mathbf{d}\circ \iota ([L(m_{\alpha}\alpha)])\\
    &=\sum_{\alpha \in \Phi_c} m_{\alpha}\alpha\\
    &= G^{\prime}(m). 
\end{align*}
It shows that 
$L(\prod_{\alpha \in \Phi_c}{G^{\prime}}^{-1}(m_{\alpha}\alpha))
\cong L(m)$, and thus the claim is proved. 
\end{proof}
By Proposition~\ref{Prop : frozenKR}, 
the following corollary holds. 
\begin{Cor}\label{Cor : D_n factorization}
    Let $m \in \mathcal{M}$ 
    and  $m = \prod_{i=1}^{n}F_i^{a_i} \times m^{\prime}$ for 
    $a_i \in \Z_{\geq 0}, m^{\prime}\in \mathcal{M}^{\prime}$. 
    Let $G^{\prime}(m^{\prime}) = \sum_{\alpha \in \Phi_c} m_{\alpha}\alpha$ be a $c$-cluster expansion. 
    Then, we have 
    \begin{align*}
    L(m)&\cong \bigotimes _{i=1}^{n}L(F_i)^{\otimes a_i} \otimes 
    \bigotimes_{\alpha \in \Phi_c}L(m_{\alpha}\alpha)\\
    &\cong \bigotimes _{i=1}^{n}L(F_i)^{\otimes a_i} \otimes
    \bigotimes_{\alpha \in \Phi_c^{\text{re}}}
    L(\alpha)^{\otimes m_{\alpha}} \otimes L(m_{\delta}\delta). 
    \end{align*}
\end{Cor}
\begin{Cor}\label{Cor : D_n using delta}
   Let $m \in \mathcal{M}$. 
   Then, $L(m)$ is a real module if and only if 
   $G(m) \in Q$ has a real $c$-cluster expansion. 
   It is also equivalent to $G(m) \not\in \delta + \sum_{\alpha\in \Lambda_c}\Z_{\geq 0}\alpha$. 
\end{Cor}
\begin{proof}
    If $G(m) = \sum_{\alpha \in \Phi_c} m_{\alpha}\alpha$ is 
    a real $c$-cluster expansion, 
    then 
    \[
    L(m) \cong \bigotimes _{i=1}^{n}L(F_i)^{\otimes a_i} \otimes 
    \bigotimes_{\alpha \in \Phi_c^{\text{re}}}L(m_{\alpha}\alpha)
    \]
    by Corollary~\ref{Cor : D_n factorization}. 
    Since $F_i(1\leq i\leq n)$, $L(m_{\alpha}\alpha)$ are mutually 
    strongly commuting and they are real, 
    $\bigotimes _{i=1}^{n}L(F_i)^{\otimes a_i} \otimes 
    \bigotimes_{\alpha \in \Phi_c^{\text{re}}}L(m_{\alpha}\alpha)$ 
    is a real module by Proposition~\ref{Prop : commuting family}. 

    If $G(m) = \sum_{\alpha \in \Phi_c} m_{\alpha}\alpha$ is 
    an imaginary $c$-cluster expansion, 
    then $[L(m_{\delta}\delta)]$ divides $[L(m)]$ by 
    Corollary~\ref{Cor : D_n factorization}. 
    Since $L(m_{\delta}\delta)$ is imaginary 
    by Proposition~\ref{Prop : D_n delta imaginary}, 
    we have $L(m)$ is imaginary. 

    The last part of this Proposition is a consequence of Proposition~\ref{Prop : m_delta>0}. 
\end{proof}
Since we can calculate both the map $G$ and the set 
$\delta + \sum_{\alpha\in \Lambda_c}\Z_{\geq 0}\alpha$ explicitly, 
Corollary~\ref{Cor : D_n using delta} gives a complete classification 
of real and imaginary modules in this category in terms of the highest weight monomials. 
As a corollary, we can prove that Conjecture~\ref{Conj : real vs monomial intro} 
holds for this case. 
\begin{Cor}\label{Cor : D_n cluster monomial}
    Let $m\in \mathcal{M}$. 
    Then, $L(m)$ is real if and only if $[L(m)]$ is a cluster monomial. 
\end{Cor}
\begin{proof}
    If $[L(m)]$ is a cluster monomial, 
    then it is clear that $L(m)$ is real by the definition of 
    monoidal categorification. 

    We assume that $L(m)$ is a real module. 
    Then, $G(m) = \sum_{\alpha \in \Phi_c} m_{\alpha}\alpha$ is 
    a real $c$-cluster expansion. 
    So we have 
    \[
    L(m) \cong \bigotimes _{i=1}^{n}L(F_i)^{\otimes a_i} \otimes
    \bigotimes_{\alpha \in \Phi_c^{\text{re}}}
    L(\alpha)^{\otimes m_{\alpha}}. 
    \]
    By Proposition~\ref{Prop : pairwise compatible}, 
    $\{L(F_i)\mid 1\leq i\leq n\}\cup 
    \{L(\alpha)\mid m_{\alpha}\neq 0\}$ are in the same monoidal cluster. 
    This shows that $[L(m)]$ is a cluster monomial in 
    $K(\mathcal{C})=\mathscr{A}$. 
\end{proof}

\begin{Ex}\label{Ex : Dn primeimaginary}
    By the proof of Proposition~\ref{Prop : D_n bij}, 
    \[
    {G^{\prime}}^{-1}(\delta)= 
    \begin{cases}
    3_03_4 \quad &\text{if }n=4, \\
    (n\!-\!1)_0(n\!-\!1)_41_32_3(n\!-\!2)_3
    \displaystyle\prod_{i=3, 5, \ldots, n-3}i_0^2
    \prod_{i=4, 6, \ldots, n-4}i_3^2 &\text{if }n=2k(k\geq3), \\
    (n\!-\!1)_0(n\!-\!1)_41_02_0(n\!-\!2)_3
    \displaystyle\prod_{i=3, 5, \ldots, n-4}i_3^2
    \prod_{i=4, 6, \ldots, n-3}i_0^2 &\text{if }n=2k-1(k\geq3). \\
    \end{cases}
    \]
    By Proposition~\ref{Prop : D_n delta imaginary}, 
    $L({G^{\prime}}^{-1}(\delta))$ is imaginary. 
    By Corollary~\ref{Cor : D_n using delta}, 
    if $L(m)$ is imaginary,  
    then $m$ is divisible by ${G^{\prime}}^{-1}(\delta)$. 
    So, $L(\delta)$ is a prime imaginary module. 
\end{Ex}


\subsection{$\mathscr{C}_{D_{4}}^{[1, 11], \mathfrak{s}^2}$ 
(The module category of type $D_{4}$ as a categorification of 
the $E_6^{(1)}$-type cluster algebra)}

Let $\mathfrak{g}$ be of type $D_4$ and $\mathcal{C}$ be the monoidal category 
$\mathscr{C}_{D_{4}}^{[1, 11], \mathfrak{s}^2}$ , 
where the admissible sequence $\mathfrak{s}^2$ is of Example~\ref{Ex : Cs}. 
By Theorem~\ref{Thm : affine categorification}, 
this category admits a monoidal categorification of the cluster algebra of type $D
E_{6}^{(1)}$. 

As the initial monoidal seed, 
we take the seed constructed by Theorem~\ref{Thm : affine categorification}. 
By Lemma~\ref{Lem : mutation monomial}, 
this seed (ignoring the arrows between frozen nodes) is explicitly given as follows : 

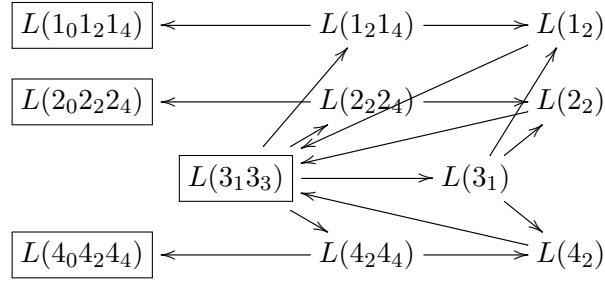
\begin{align*}
&\raisebox{2.3em}{\scalebox{1}{\xymatrix@C=0.7ex@R=1ex{
 & \boxed{L(1_0 1_2 1_4)} && L(1_2 1_4)\ar[ll]\ar[rr]
 &&L(1_2)\ar[ddlll]\\
 & \boxed{L(2_0 2_2 2_4)} && L(2_2 2_4)\ar[ll]\ar[rr] && 
 L(2_2) \ar[dlll]\\
 &&\boxed{L(3_1 3_3)} \ar[uur]\ar[ur]\ar[dr]\ar[rr] 
 && L(3_1) \ar[uur]\ar[ur]\ar[dr] \\
 & \boxed{L(4_0 4_2 4_4)}  && L(4_2 4_4)\ar[ll]\ar[rr]
 &&L(4_2)\ar[ulll]\\
}}}
\end{align*}

We set $I_{\text{fin}} = \{1, \ldots, 4\}$. 

Let $\mathcal{M} \coloneq \mathcal{M}_+^{[1, 11], \mathfrak{s}^2}$ 
be the set of monomials in 
the indeterminates
\[
    \{3_1, 3_3\}\cup 
\{i_0, i_2, i_4\mid i =1, 2, 4\}
\]
which parametrizes $\Irr \mathcal{C}$. 
For $i \in I_{\text{fin}}$, we define $F_i \in \mathcal{M}$ by 
\[
F_i = 
\begin{cases}
    3_13_3 \quad &\text{if } i =3\\
    i_0i_2i_4 \quad &\text{if } i =1, 2, 4
\end{cases}
\]

We define the subset $\mathcal{M}^{\prime}$ of $\mathcal{M}$ by 
\[
\mathcal{M}^{\prime} \coloneq \{m \in \mathcal{M}\mid 
\text{
$m$ is not divisible by every $F_i$ for $i \in I_{\text{fin}}$
}
\}. 
\]
Let $K=[0, 10], K^{\text{ex}}=[0,6]$, and $K^{\text{fr}}=[7, 10]$. 
Let $\mathcal{S}_0 = (\{x_i\}_{i \in [0, 10]}, \tilde{B}_0)$ 
be a seed, where $x_0, \ldots, x_{10}$ are indeterminates and 
$\tilde{B}_0$ is the exchange matrix associated with the above quiver. 
We write the $4$ frozen variables $x_7, \ldots, x_{10}$ 
as $f_1, \ldots, f_{4}$, respectively. 
Then by Theorem~\ref{Thm : affine categorification}, there is an isomorphism 
\begin{align*}
  \iota : K(\mathcal{C}) \overset{\sim}{\longrightarrow}\mathscr{A} \quad; 
  &\quad [L(1_2 1_4)] \mapsto x_1  \\
  &\quad [L(1_2)] \mapsto x_2 \\
  &\quad [L(3_1)] \mapsto x_3\\
  &\quad [L(2_2)] \mapsto x_4  \\
  &\quad [L(4_2)] \mapsto x_5  \\
  &\quad [L(4_2 4_4)] \mapsto x_6  \\
  &\quad [L(2_2 2_4)] \mapsto x_0  \\
  &\quad [L(F_i)] \mapsto f_i &&\text{ for } i \in I_{\text{fin}}. 
\end{align*}
We sometimes identify $K(\mathcal{C})$ with $\mathscr{A}$ by $\iota$. 

First, we combinatorially determine the $\mathbf{d}$-vector of each simple representation 
in $\mathscr{A}$.
Let $Q = \bigoplus_{i=0}^{6}\Z\alpha_i$ be the root lattice of $E_{6}^{(1)}$. 

For each $m \in \mathcal{M}$, there is a unique expression of the form as follows : 
\begin{align*}
    m = &\prod_{i\in I_{\text{fin}}}F_i^{a_i} \times 
    3_1^{p}3_3^q\\
    &\times \prod_{i=1, 2, 4}(i_2i_4)^{p_i}(i_0i_4)^{q_i}(i_0i_2)^{r_i}\\
    &\times \prod_{i=1, 2, 4}i_0^{s_i}i_2^{t_i}i_4^{u_i}
\end{align*}
such that each of the following tuples contains at most one nonzero entry :  

$(p, q)$, $(p_i, q_i, r_i), (s_i, t_i, u_i), 
(p_i, s_i), (q_i, t_i), (r_i, u_i)$ for each $i =1, 2, 4$
. 

Then, we define a map $G : \mathcal{M} \to Q$ by 

\begin{align*}
    G(m) = &p(-\alpha_3)+q\alpha_3\\
    &+ p_1(-\alpha_1) + q_1(\alpha_1 + 2\alpha_2 + \alpha_3) 
    + r_1\alpha_1\\
    &+ s_1(\alpha_1 + \alpha_2) + t_1(-\alpha_2) 
    + u(\alpha_2 + \alpha_3)\\
    &+ p_2(-\alpha_0) + q_2(\alpha_0 + \alpha_3 + 2\alpha_4) 
    + r_2\alpha_0\\
    &+ s_2(\alpha_0 + \alpha_4) + t_2(-\alpha_4) 
    + u_2(\alpha_3 + \alpha_4)\\
    &+ p_4(-\alpha_6) + q_2(\alpha_3 + 2\alpha_5 + \alpha_6) 
    + r_4\alpha_6\\
    &+ s_4(\alpha_5 + \alpha_6) + t_4(-\alpha_6) 
    + u_4(\alpha_3 + \alpha_5). \\ 
\end{align*}

Let $G' : \mathcal{M}' \to Q$ denote the restriction of $G$ to $\mathcal{M}'$. 
By definition, for $m \in \mathcal{M}^{\prime}$, 
$G^{\prime}(m) = G(\prod_{i \in {I_{\text{fin}}}}F_i^{a_i} \times m)$ for any $a_i$. 

In this subsection, 
we write 
    \[
    3^a \coloneq 
    \begin{cases}
        3_3^a \quad &\text{ if } a\geq 0, \\
        3_1^{-a} \quad &\text{ if } a < 0 \\
    \end{cases}
    \]
    for simplicity. 
    Then, 
    $G^{\prime}(3^{a}) 
    = a\alpha_3$ for $a \in \Z$. 

    For $\gamma = \sum_{i=0}^{6}a_i \alpha_i \in Q$, 
    we set monomials $m_1, m_2, m_4 \in \mathcal{M}^{\prime}$ as follows$\colon$

    For $i =1, 2, 4$, we define 
\[
m_i \coloneq 
\begin{cases}
    (i_0 i_2)^{x-y}i_0^{y} \quad &\text{if } (x, y)\in A_i,\\
    (i_0 i_4)^{y-x}i_0^{2x-y} \quad &\text{if }(x, y)\in B_i,\\
    (i_0 i_4)^{x}i_4^{y-2x} \quad &\text{if }(x, y)\in C_i,\\
    (i_2 i_4)^{-x}i_4^{y} \quad &\text{if }(x, y)\in D_i,\\
    (i_2 i_4)^{-x}i_2^{-y} \quad &\text{if }(x, y)\in E_i,\\
    (i_0 i_2)^{x}i_2^{-y} \quad &\text{if }(x, y)\in F_i, \\
\end{cases}
\]
where $A_i, \ldots, F_i$ are as follows$\colon$

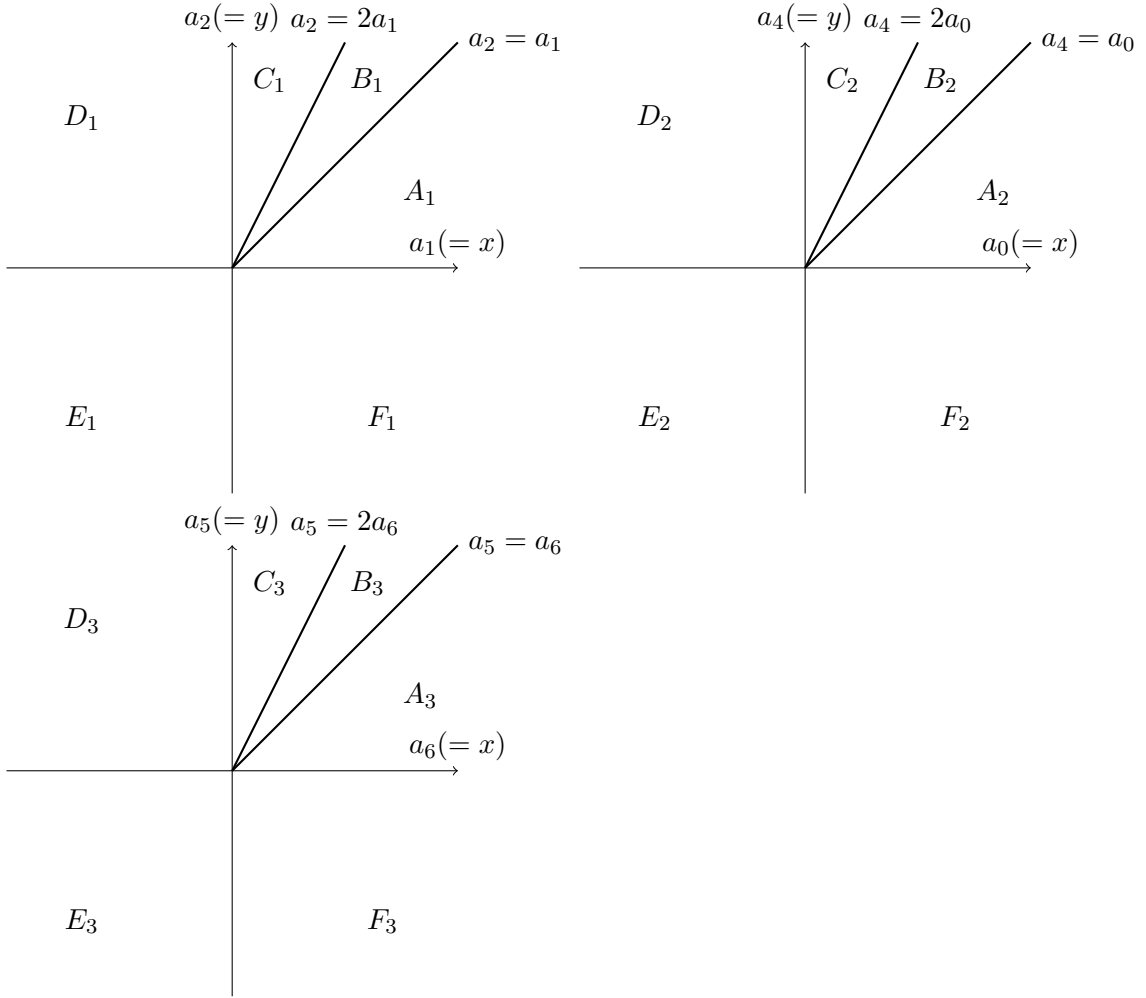
\begin{tikzpicture}[scale=1]

  \draw[->] (-3,0) -- (3,0) node[above] {$a_1(=x)$};
  \draw[->] (0,-3) -- (0,3) node[above] {$a_2(=y)$};

  \draw[thick]   (0,0) -- (1.5,3) node[above] {$a_2=2a_1$};
  \draw[thick]  (0, 0)  -- (3,3) node[right] {$a_2=a_1$};

  \node at (2.5,1){$A_1$};
  \node at (1.8,2.5) {$B_1$};
  \node at (0.5,2.5) {$C_1$};
  \node at (-2,2) {$D_1$};
  \node at (-2,-2) {$E_1$};
  \node at (2, -2) {$F_1$};

\end{tikzpicture}
\begin{tikzpicture}[scale=1]

  \draw[->] (-3,0) -- (3,0) node[above] {$a_0(=x)$};
  \draw[->] (0,-3) -- (0,3) node[above] {$a_4(=y)$};

  \draw[thick]   (0,0) -- (1.5,3) node[above] {$a_4=2a_0$};
  \draw[thick]  (0, 0)  -- (3,3) node[right] {$a_4=a_0$};

  \node at (2.5,1){$A_2$};
  \node at (1.8,2.5) {$B_2$};
  \node at (0.5,2.5) {$C_2$};
  \node at (-2,2) {$D_2$};
  \node at (-2,-2) {$E_2$};
  \node at (2, -2) {$F_2$};

\end{tikzpicture}

\begin{tikzpicture}[scale=1]

  \draw[->] (-3,0) -- (3,0) node[above] {$a_6(=x)$};
  \draw[->] (0,-3) -- (0,3) node[above] {$a_5(=y)$};

  \draw[thick]   (0,0) -- (1.5,3) node[above] {$a_5=2a_6$};
  \draw[thick]  (0, 0)  -- (3,3) node[right] {$a_5=a_6$};

  \node at (2.5,1){$A_3$};
  \node at (1.8,2.5) {$B_3$};
  \node at (0.5,2.5) {$C_3$};
  \node at (-2,2) {$D_3$};
  \node at (-2,-2) {$E_3$};
  \node at (2, -2) {$F_3$};

\end{tikzpicture}

Then, we can check that $\gamma$ can be written in the form 
\[
\gamma = G^{\prime}(m_1 m_2 m_4) + b\alpha_3,  
\]
where $b \in \Z$. 
Then, by construction, the map 
\[
\sum_{i=0}^{6}a_i \alpha_i \longmapsto m_1 m_2 m_4 \times 
3^b
\]
is the inverse map of $G^{\prime}$. 
As a result, the following Proposition holds.

\begin{Prop}\label{Prop : D_4 bij}
    $G' : \mathcal{M}' \to Q$ is bijective. 
\end{Prop}

\begin{Thm}\label{Thm : D_4 G=dL}
    $G = \mathbf{d}\circ \iota \circ [L(-)]$. 
\end{Thm}
\begin{proof}
Note that for $m\in \mathcal{M}$ and $i \in I_{\text{fin}}$, 
the equalities $G(m) = G(F_i m)$ and 
$\mathbf{d}\circ \iota \circ [L(m)] = \mathbf{d}\circ \iota \circ [L(F_im)]$ hold, 
as in the proof of Theorem~\ref{Thm : D_n G=dL}. 
We can check that the statement holds 
for the case $m = 3^{\pm 1}$ by considering 
the initial seed $\mathcal{S}_0$ and $\mu_3(\mathcal{S}_0)$. 

For $m \in \mathcal{M}$ and $a \in \Z$, 
we have 
\[
\mathbf{d}\circ \iota \circ [L(m \times 3^a)]
\leq \mathbf{d}\circ \iota \circ [L(m)] 
+ a \alpha_3
\]
and 
\[
\mathbf{d}\circ \iota \circ [L(m \times 3^a)]
+a(-\alpha_3)
\geq \mathbf{d}\circ \iota \circ 
[L(m F_3^{|a|})]
=\mathbf{d}\circ \iota \circ [L(m)]
\]
by Lemma~\ref{Lem : dvector inequality}. 
This shows 
\[
\mathbf{d}\circ \iota \circ [L(m \times 3^a)]
=\mathbf{d}\circ \iota \circ [L(m)]
+a \alpha_3. 
\]
By the definition of $G$, 
for $m \in \mathcal{M}$ and $a\in \Z$, 
\[G(m \times 3^a) = 
G(m) +a \alpha_3
\]
also holds. 

By the above fact, it is enough to show that 
for any $a_i, b_i, c_i \in \Z_{\geq 0}$ for $i = 1, 2, 4$, 
there exist $k \in \Z$ 
such that 
the claim holds for 
\[
m = \prod_{i=1, 2, 4}i_0^{a_i}i_2^{b_i}i_4^{c_i}\times 3^k. 
\]

For $i=1, 2, 4$, we set 
\[
(v_i, w_i) \coloneq 
\begin{cases}
    (1, 2) \quad\text{ for } i=1, \\
    (0, 4) \quad\text{ for } i=2, \\
    (6, 5) \quad\text{ for } i=4.  
\end{cases}
\]

By Lemma~\ref{Lem : mutation monomial}, we have that the following simple modules are in the 
same monoidal cluster : 

\begin{equation}\label{eq : D_4}
\begin{aligned}
\mathcal{S}_0 &\ni& &i_2,& &i_2 i_4,& &i_0i_2i_4, \\
\mu_{v_i}(\mathcal{S}_0) &\ni& &i_2,& &i_0i_2,& &i_0i_2i_4, \\
\mu_{w_i}(\mathcal{S}_0) &\ni& &i_43_1,& &i_2i_4,& &i_0i_2i_4, \\
\mu_{w_i}\mu_{v_i}(\mathcal{S}_0) &\ni& &i_0,& &i_0i_2,& &i_0i_2i_4, \\
\mu_{v_i}\mu_{w_i}(\mathcal{S}_0) &\ni& &i_43_1,& &i_0,& &i_0i_2i_4. \\
\end{aligned}
\end{equation}

For $i = 1, 2, 4$, let $\psi_i$ be an operation chosen from the set 
$\{\phi, \mu_{v_i}, \mu_{w_i}, \mu_{w_i}\mu_{v_i}, \mu_{v_i}\mu_{w_i}\}$. 
For $i = 1, 2, 4$, let $\mathscr{L}_{\psi_i}$ denote the set of three simple modules 
listed in the right-hand side of \eqref{eq : D_4} corresponding to the cluster $\psi_i(\mathcal{S}_0)$.
Since there are no arrows in $\mathcal{S}_0$ 
between any vertex in $\{v_i, w_i\}$ and any vertex in $\{v_j, w_j\}$ 
for $i \neq j$, 
the result of applying the three operations $\psi_1, \psi_2$, and $\psi_4$ to $\mathcal{S}_0$ 
is independent of the order in which they are applied. 
This implies that for any choice of $\psi_1, \psi_2$ and $\psi_4$, 
the union $\mathscr{L}_{\psi_1}\cup\mathscr{L}_{\psi_2}\cup\mathscr{L}_{\psi_4}$ is 
contained in the monoidal seed $\psi_1\psi_2\psi_4(\mathcal{S}_0)$. 



By the calculation of $\mathbf{d}$-vectors, we can check that the claim of the theorem holds for 
all the simple modules listed in \eqref{eq : D_4}. 
By the above process, for any $a_i, b_i, c_i \in \Z_{\geq 0}$ for $i = 1, 2, 4$, 
we can realize a simple module whose form is 
\[
m = \prod_{i=1, 2, 4}i_0^{a_i}i_2^{b_i}i_4^{c_i}\times 3^k
\]
for some $k \in \Z$ 
as a cluster monomial. 
It completes the proof. 
\end{proof}

\begin{Prop}\label{Prop : D_4 delta imaginary}
    For $l \in \Z_{>0}$, the irreducible module $L(l\delta)$ is imaginary. 
\end{Prop}
\begin{proof}
    Since ${G^{\prime}}^{-1}(p\delta-\alpha_1){G^{\prime}}^{-1}(l\delta)
    ={G^{\prime}}^{-1}((p+l)\delta-\alpha_1)$ for any $p, l\in \Z_{>0}$ by 
    the description of ${G^{\prime}}^{-1}$, 
    we can prove it in the same way as Proposition~\ref{Prop : D_n delta imaginary}. 
\end{proof}

\begin{Prop}\label{Prop : D_4 compatible}
    Let $\gamma_1, \gamma_2 \in \Phi_c$ be distinct roots and  
$m_1, m_2\in\Z_{\geq1}$. 
If $\gamma_1$ and $\gamma_2$ are $c$-compatible, 
then two irreducible modules $L(m_1\gamma_1)$, $L(m_2\gamma_2)$ commute strongly. 
\end{Prop}
\begin{proof}
    Let $\beta_3, \beta_4, \beta_5$ be the root as in Proposition~\ref{Prop : E_n Lambda_c} for 
    the case of $E_6^{(1)}$. 
    Since 
    \begin{align*}
        {G^{\prime}}^{-1}(\beta_3+\beta_4)^{m_1}
        {G^{\prime}}^{-1}(\delta-\beta_3-\beta_4)^{m_1}
        ={G^{\prime}}^{-1}(m_1\delta), \\
        {G^{\prime}}^{-1}(\beta_5)^{m_1}
        {G^{\prime}}^{-1}(\delta-\beta_5)^{m_1}
        ={G^{\prime}}^{-1}(m_1\delta)
    \end{align*}
    hold by the proof of Proposition~\ref{Prop : D_4 bij}, 
    we can prove it in the same way as Proposition~\ref{Prop : D_n compatible}. 
\end{proof}

Then, we can prove the following facts in the same way as 
Theorem~\ref{Thm : D_n tensor factorization}, Corollary~\ref{Cor : D_n factorization}, 
Corollary~\ref{Cor : D_n using delta}, 
and Corollary~\ref{Cor : D_n cluster monomial}. 

\begin{Thm}\label{Thm : D_4 tensor factorization}
    Let $m \in \mathcal{M}^{\prime}$ and 
    $G^{\prime}(m) = \sum_{\alpha \in \Phi_c} m_{\alpha}\alpha$ be a 
    $c$-cluster expansion. 
    Then, 
    \begin{align*}
    L(m) &\cong \bigotimes_{\alpha \in \Phi_c}L(m_{\alpha}\alpha)\\
    &\cong \bigotimes_{\alpha \in \Phi_c^{\text{re}}}
    L(\alpha)^{\otimes m_{\alpha}} \otimes L(m_{\delta}\delta). 
    \end{align*}
\end{Thm}
\begin{Cor}\label{Cor : D_4 factorization}
    Let $m \in \mathcal{M}$. 
    We can describe it as $m = \prod_{i=1}^{4}F_i^{a_i} \times m^{\prime}$ for 
    $a_i \in \Z_{\geq 0}, m^{\prime}\in \mathcal{M}^{\prime}$ uniquely. 
    Let $G^{\prime}(m^{\prime}) = \sum_{\alpha \in \Phi_c} m_{\alpha}\alpha$ be a $c$-cluster expansion. 
    Then, we have 
    \begin{align*}
    L(m)&\cong \bigotimes _{i=1}^{4}L(F_i)^{\otimes a_i} \otimes 
    \bigotimes_{\alpha \in \Phi_c}L(m_{\alpha}\alpha)\\
    &\cong \bigotimes _{i=1}^{4}L(F_i)^{\otimes a_i} \otimes
    \bigotimes_{\alpha \in \Phi_c^{\text{re}}}
    L(\alpha)^{\otimes m_{\alpha}} \otimes L(m_{\delta}\delta). 
    \end{align*}
\end{Cor}
\begin{Cor}\label{Cor : D_4 using delta}
    Let $m \in \mathcal{M}$. 
   Then, $L(m)$ is a real module if and only if 
   $G(m) \in Q$ has a real $c$-cluster expansion. 
   It is also equivalent to $G(m) \not\in \delta + \sum_{\alpha\in \Lambda_c}\Z_{\geq 0}\alpha$. 
\end{Cor}
\begin{Cor}\label{Cor : D_4 cluster monomial}
    Let $m\in \mathcal{M}$. 
    Then, $L(m)$ is real if and only if $[L(m)]$ is a cluster monomial. 
\end{Cor}

\begin{Ex}\label{Ex : D4 primeimaginary}
    By the proof of Proposition~\ref{Prop : D_4 bij}, 
    \[
    {G^{\prime}}^{-1}(\delta)= 
    1_01_42_02_44_04_4. 
    \]
    By Proposition~\ref{Prop : D_4 delta imaginary}, 
    $L({G^{\prime}}^{-1}(\delta))$ is imaginary. 
    By Corollary~\ref{Cor : D_4 using delta}, 
    if $L(m)$ is imaginary,  
    then $m$ is divisible by ${G^{\prime}}^{-1}(\delta)$. 
    So, $L(\delta)$ is a prime imaginary module. 
\end{Ex}

\subsection{$\mathscr{C}_{E_{n}}^{[1, 2n+1], \mathfrak{s}^1}(n=6,7,8)$ 
(The module category of type $E_{n}$ as a categorification of 
the $E_{n}^{(1)}$-type cluster algebra)}

Let $\mathfrak{g}$ be of type $E_n$ and $\mathcal{C}$ be the monoidal category 
$\mathscr{C}_{E_{n}}^{[1, 2n+1], \mathfrak{s}^1}$, 
where the admissible sequence $\mathfrak{s}^1$ is of Example~\ref{Ex : Cs}. 
By Theorem~\ref{Thm : affine categorification}, 
this category admits a monoidal categorification of the cluster algebra of type $E_n^{(1)}$. 

As the initial monoidal seed, 
we take the seed constructed by Theorem~\ref{Thm : affine categorification}. 
By Lemma~\ref{Lem : mutation monomial}, 
this seed (ignoring the arrows between frozen nodes) is explicitly given as follows : 

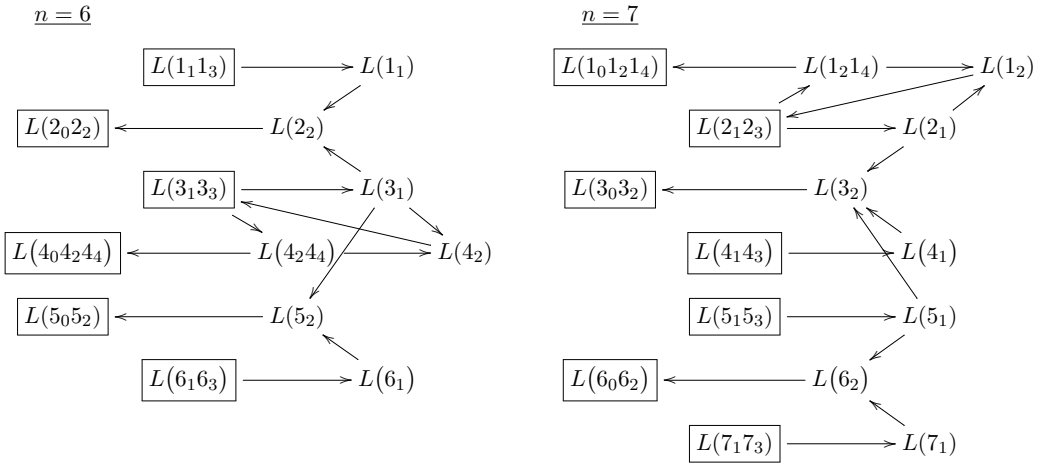
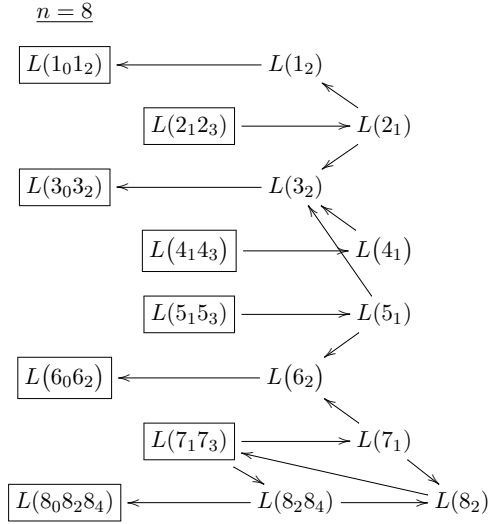
\begin{align*} 
&\raisebox{2.3em}{\scalebox{0.8}{\xymatrix@C=0.7ex@R=1ex{
&\underline{n=6}\\
 && \boxed{L(1_1 1_3)} \ar[rr] && L(1_1)\ar[dl]\\
 & \boxed{L(2_0 2_2)}  && L(2_2)\ar[ll]\\
 && \boxed{L(3_1 3_3)} \ar[rr]\ar[dr] && L(3_1)\ar[ul]\ar[dr]\ar[ddl]\\
 &\boxed{L\bigl(4_0 4_2 4_4\bigr)}
&& L\bigl( 4_2 4_4\bigr) \ar[ll] \ar[rr]
&&L( 4_2)\ar[ulll] \\
& \boxed{L(5_0 5_2)} && L(5_2)\ar[ll]\\
&& \boxed{L\bigl(6_1 6_3\bigr)} \ar[rr] && L\bigl(6_1\bigr)\ar[ul]\\
}}}
&\raisebox{2.3em}{\scalebox{0.8}{\xymatrix@C=0.7ex@R=1ex{
&\underline{n=7}\\
 & \boxed{L(1_0 1_2 1_4)} && L(1_2 1_4)\ar[ll]\ar[rr]
 &&L(1_2)\ar[dlll]\\
 && \boxed{L(2_1 2_3)} \ar[ur]\ar[rr] && L(2_1)\ar[dl]\ar[ur]\\
 & \boxed{L(3_0 3_2)} && L(3_2)\ar[ll]\\
 &&\boxed{L\bigl(4_1 4_3\bigr)} \ar[rr] 
&& L\bigl( 4_1\bigr) \ar[ul]\\
&& \boxed{L(5_1 5_3)} \ar[rr]&& L(5_1)\ar[uul]\ar[dl]\\
& \boxed{L\bigl(6_0 6_2\bigr)}  && L\bigl(6_2\bigr)\ar[ll]\\
&&\boxed{L(7_1 7_3)} \ar[rr]&& L(7_1) \ar[ul]\\
}}}
\end{align*}
\begin{align*}
&\raisebox{2.3em}{\scalebox{0.8}{\xymatrix@C=0.7ex@R=1ex{
&\underline{n=8}\\
 & \boxed{L(1_0 1_2)}  && L(1_2)\ar[ll]\\
 && \boxed{L(2_1 2_3)} \ar[rr] && L(2_1)\ar[ul]\ar[dl]\\
 & \boxed{L(3_0 3_2)}  && L(3_2)\ar[ll]\\
 &&\boxed{L\bigl(4_1 4_3\bigr)} \ar[rr] 
&& L\bigl( 4_1\bigr) \ar[ul]\\
&& \boxed{L(5_1 5_3)} \ar[rr]&& L(5_1)\ar[uul]\ar[dl]\\
& \boxed{L\bigl(6_0 6_2\bigr)}&& L\bigl(6_2\bigr)\ar[ll]\\
&&\boxed{L(7_1 7_3)} \ar[rr]\ar[dr]&& L(7_1) \ar[ul]\ar[dr]\\
&\boxed{L(8_0 8_2 8_4)}  &&L(8_2 8_4) \ar[ll]\ar[rr] &&L(8_2)\ar[ulll]\\
}}}
\end{align*}
We set $I_{\text{fin}} = \{1, \ldots, n\}$ and $I_{\text{fin}} = I_0 \sqcup I_1$, 
where 
\[I_0 \coloneq 
\begin{cases}
    \{1, 3, 6\} \quad \text{for }n=6\\
    \{2, 4, 5, 7\} \quad \text{for }n=7, 8\\
\end{cases}
\]
and 
\[
I_1 \coloneq 
\begin{cases}
    \{2, 4, 5\} \quad \text{for }n=6\\
    \{1, 3, 6\} \quad \text{for }n=7\\
    \{1, 3, 6, 8\}\quad \text{for }n=8
\end{cases}.
\]
We fix the index $k$ of the Dynkin diagrams to be 
\[
k \coloneq 
\begin{cases}
    4 \quad \text{for }n=6, \\
    1 \quad \text{for }n=7, \\
    8 \quad \text{for }n=8. \\
\end{cases}
\]

Let $\mathcal{M} \coloneq \mathcal{M}_+^{[1, 2n+1], \mathfrak{s}^1}$ be the set of monomials in 
the indeterminates
\[
\{i_1, i_3\mid i \in I_0\}\cup 
\{j_0, j_2\mid j \in I_1\}\cup 
\{k_4\}, 
\]
which parametrizes $\Irr \mathcal{C}$. 
For $i \in I_{\text{fin}}$, we define $F_i \in \mathcal{M}$ by 
\[
F_i = 
\begin{cases}
    i_1i_3 \quad &\text{if } i \in I_0, \\
    i_0i_2 \quad &\text{if } i \in I_1\setminus \{k\}, \\
    k_0 k_2 k_4 \quad &\text{if } i = k. 
\end{cases}
\]

We define the subset $\mathcal{M}^{\prime}$ of $\mathcal{M}$ by 
\[
\mathcal{M}^{\prime} \coloneq \{m \in \mathcal{M}\mid 
\text{
$m$ is not divisible by every $F_i$ for $i \in I_{\text{fin}}$
}
\}. 
\]
Let $K=[0, 2n], K^{\text{ex}}=[0,n]$, and $K^{\text{fr}}=[n+1, 2n]$. 
Let $\mathcal{S}_0 = (\{x_i\}_{i \in [0, 2n]}, \tilde{B}_0)$ 
be a seed, where $\tilde{B}_0$ is the exchange matrix associated with the above quiver. 
We write the $n$ frozen variables $x_{n+1}, \ldots, x_{2n}$ 
as $f_1, \ldots, f_{n}$, respectively.
Then by Theorem~\ref{Thm : affine categorification}, there is an isomorphism 
\begin{align*}
  \iota : K(\mathcal{C}) \overset{\sim}{\longrightarrow}\mathscr{A} \quad; 
  &\quad [L(i_1)] \mapsto x_i  &&\text{ for } i \in I_0\\
  &\quad [L(j_2)] \mapsto x_j &&\text{ for } j \in I_1\\
  &\quad [L\bigl(k_2 k_4\bigr)] \mapsto x_0 &\\
  &\quad [L(F_i)] \mapsto f_i &&\text{ for } i \in I_{\text{fin}}. 
\end{align*}
We sometimes identify $K(\mathcal{C})$ with $\mathscr{A}$ by $\iota$. 

First, we combinatorially determine the $\mathbf{d}$-vector of each simple representation 
in $\mathscr{A}$.
Let $Q = \bigoplus_{i=0}^{n}\Z\alpha_i$ be the root lattice of $E_{n}^{(1)}$. 

For each $m \in \mathcal{M}$, there is a unique expression of the form
\begin{align*}
    m = &\prod_{i\in I_{\text{fin}}}F_i^{a_i} \times \prod_{i \in I_0}i_1^{p_i}i_3^{q_i} \times 
    \prod_{j \in I_1\setminus \{k\}}j_2^{p_j}j_0^{q_j} \\
    &\times \bigl(k_2 k_4\bigr)^{p}\bigl(k_0 k_4\bigr)^{q}
    \bigl(k_0 k_2\bigr)^{r}\\
    &\times k_0^s k_2^t k_4^u, 
\end{align*}
such that each of the following tuples contains at most one nonzero entry :  

$(p_i, q_i)$ for each $i \in I_{\text{fin}}\setminus \{k\}$, 
$(p, q, r)$, $(s, t, u)$, $(p, s)$, $(q, t)$, $(r, u)$. 

Then, we define a map $G : \mathcal{M} \to Q$ by 

    \begin{align*}
    G(m) = &\sum_{i \in I_{\text{fin}}\setminus\{k\}}
    \bigl(p_i(-\alpha_i) + q_i\alpha_i\bigr) \\
    &+ p(-\alpha_0) 
    + q(\alpha_0 + 2\alpha_k + \alpha_{k+(-1)^{n+1}}) 
    + r\alpha_0\\
    &+ s(\alpha_0 + \alpha_k) + t(-\alpha_k) 
    + u(\alpha_k+\alpha_{k+(-1)^{n+1}}). 
\end{align*}

Let $G' : \mathcal{M}' \to Q$ denote the restriction of $G$ to $\mathcal{M}'$. 
By definition, for $m \in \mathcal{M}^{\prime}$, 
$G^{\prime}(m) = G(\prod_{i \in {I_{\text{fin}}}}F_i^{a_i} \times m)$ for any $a_i$. 

In this subsection, 
for $i \in I_{\text{fin}}\setminus\{k\}$ and $a \in \Z$, 
    we write 
    \[
    i^a \coloneq 
   \begin{cases}
        i_3^a \quad &\text{ if } i\in I_0, a\geq 0, \\
        i_1^{-a} \quad &\text{ if } i \in I_0, a < 0, \\
        i_0^a \quad &\text{ if } i \in I_1\setminus\{k\}, a \geq 0, \\
        i_2^{-a} \quad &\text{ if } i\in I_1\setminus\{k\}, a < 0, 
    \end{cases}
    \]
    for simplicity. 
    Then, 
    $G^{\prime}\biggl(\prod_{i \in I_{\text{fin}}\setminus\{k\}}i^{a_i}\biggr) 
    = \sum_{i \in I_{\text{fin}}\setminus\{k\}}a_i \alpha_i$ for $a_i \in \Z$. 

    For $\gamma = \sum_{i=0}^{n}a_i \alpha_i \in Q$, 
    we set a monomial $m_1 \in \mathcal{M}^{\prime}$ as 
    \[
m_1 \coloneq 
\begin{cases}
    (k_0 k_2)^{a_0-a_k}k_0^{a_k} \quad &\text{if } (a_0, a_k)\in A,\\
    (k_0 k_4)^{a_k-a_0}k_0^{2a_0-a_k} \quad &\text{if } (a_0, a_k)\in B,\\
    (k_0 k_4)^{a_0}k_4^{a_k-2a_0} \quad &\text{if } (a_0, a_k)\in C,\\
    (k_2 k_4)^{-a_0}k_4^{a_k} \quad &\text{if } (a_0, a_k)\in D,\\
    (k_2 k_4)^{-a_0}k_2^{-a_k} \quad &\text{if } (a_0, a_k)\in E,\\
    (k_0 k_2)^{a_0}k_2^{-a_k} \quad &\text{if } (a_0, a_k)\in F,\\
\end{cases}
\]
    where $A, \ldots, F$ are as follows$\colon$

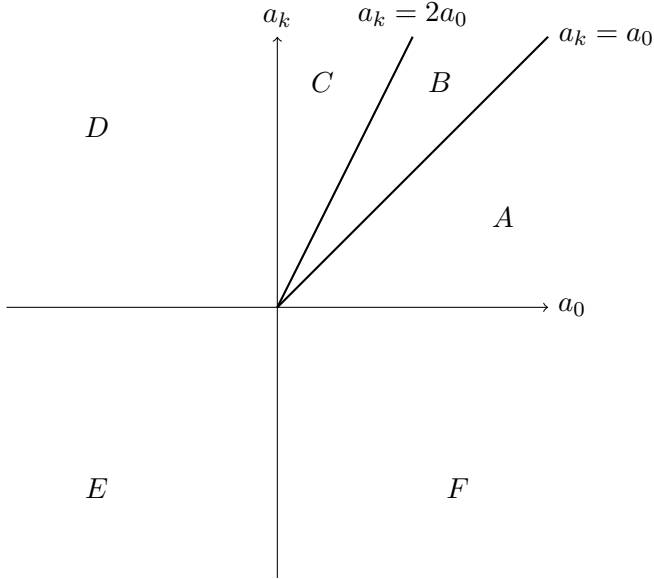
    \begin{tikzpicture}[scale=1.2]

  \draw[->] (-3,0) -- (3,0) node[right] {$a_0$};
  \draw[->] (0,-3) -- (0,3) node[above] {$a_k$};

  \draw[thick]   (0,0) -- (1.5,3) node[above] {$a_k=2a_0$};
  \draw[thick]  (0, 0)  -- (3,3) node[right] {$a_k=a_0$};

  \node at (2.5,1){$A$};
  \node at (1.8,2.5) {$B$};
  \node at (0.5,2.5) {$C$};
  \node at (-2,2) {$D$};
  \node at (-2,-2) {$E$};
  \node at (2, -2) {$F$};

\end{tikzpicture}

Then, we can check that $\gamma$ can be written in the form 
\[
\gamma = G^{\prime}(m_1) + 
\sum_{\substack{i \in \{0, 1, \ldots, n\}\\i \neq 0, k}}b_i \alpha_i 
= G^{\prime}(m_1) + 
\sum_{\substack{i \in I_{\text{fin}}\\i \neq k}}b_i \alpha_i,  
\]
where $b_i \in \Z$ for $i \in I_{\text{fin}}\setminus\{k\}$. 
Then, by the construction, the map 
\[
\sum_{i=0}^{n}a_i \alpha_i \longmapsto m_1 \times 
\prod_{\substack{i \in I_{\text{fin}}\\i \neq k}}i^{b_i}. 
\]
    is the inverse map of $G^{\prime}$. 
    As a result, the following Proposition holds. 

\begin{Prop}\label{Prop : E_n bij}
    $G' : \mathcal{M}' \to Q$ is bijective. 
\end{Prop}

\begin{Thm}\label{Thm : E_n G=dL}
    $G = \mathbf{d}\circ \iota \circ [L(-)]$. 
\end{Thm}

\begin{proof}
Note that for $m\in \mathcal{M}$ and $i \in I_{\text{fin}}$, 
the equalities $G(m) = G(F_i m)$ and 
$\mathbf{d}\circ \iota \circ [L(m)] = \mathbf{d}\circ \iota \circ [L(F_im)]$ hold, 
as in the proof of Theorem~\ref{Thm : D_n G=dL}. 
We can check that the statement holds 
for the case $m = i^{\pm 1}$ for $i \in I_{\text{fin}}\setminus\{k\}$ by considering 
the initial seed $\mathcal{S}_0$ and $\mu_i(\mathcal{S}_0)$. 

For $m \in \mathcal{M}$ and $a_i \in \Z$ for $i \in I_{\text{fin}}\setminus\{k\}$, 
we have 
\[
\mathbf{d}\circ \iota \circ [L(m \prod_{\substack{i \in I_{\text{fin}}\\i \neq k}}i^{a_i})]
\leq \mathbf{d}\circ \iota \circ [L(m)] 
+ \sum_{\substack{i \in I_{\text{fin}}\\i \neq k}}a_i\alpha_i
\]
and 
\[
\mathbf{d}\circ \iota \circ [L(m \prod_{\substack{i \in I_{\text{fin}}\\i \neq k}}i^{a_i})]
+\sum_{\substack{i \in I_{\text{fin}}\\i \neq k}}a_i(-\alpha_i)
\geq \mathbf{d}\circ \iota \circ 
[L(m \prod_{\substack{i \in I_{\text{fin}}\\i \neq k}}F_i^{a_i})]
=\mathbf{d}\circ \iota \circ [L(m)]
\]
by Lemma~\ref{Lem : dvector inequality}. 
This shows 
\[
\mathbf{d}\circ \iota \circ [L(m \prod_{\substack{i \in I_{\text{fin}}\\i \neq k}}i^{a_i})]
=\mathbf{d}\circ \iota \circ [L(m)]
+\sum_{\substack{i \in I_{\text{fin}}\\i \neq k}}a_i\alpha_i. 
\]
By the definition of $G$, 
for $m \in \mathcal{M}$ and $a_i \in \Z$ for $i \in I_{\text{fin}}\setminus\{k\}$, 
\[G(m \prod_{\substack{i \in I_{\text{fin}}\\i \neq k}}i^{a_i}) = 
G(m) +\sum_{\substack{i \in I_{\text{fin}}\\i \neq k}}a_i\alpha_i
\]
also holds. 

By the above fact, it is enough to show that 
for any $a, b, c \in \Z_{\geq 0}$, 
there exist integers $a_i \in \Z$ for $i \in I_{\text{fin}}\setminus\{k\}$ 
such that 
the claim of this theorem holds for 
\[m = k_0^{a} k_2^{b} k_4^{c}
\prod_{\substack{i \in I_{\text{fin}}\\i \neq k}}i^{a_i}. 
\]

By Lemma~\ref{Lem : mutation monomial}, we have that the following simple modules are in the 
same monoidal cluster : 
\begin{align*}
\mathcal{S}_0 &\ni& &k_2,& &k_2 k_4,& 
&k_0 k_2 k_4, \\
\mu_0(\mathcal{S}_0) &\ni& &k_2,& &k_0 k_2,& 
&k_0 k_2 k_4, \\
\mu_k(\mathcal{S}_0) &\ni& &k_4(k+(-1)^{n+1})_1,& &k_2 k_4,& 
&k_0 k_2 k_4, \\
\mu_k\mu_0(\mathcal{S}_0) &\ni& &k_0,& &k_0 k_2,& 
&k_0 k_2 k_4, \\
\mu_0\mu_k(\mathcal{S}_0) &\ni& &k_4(k+(-1)^{n+1})_1,& &k_0,& 
&k_0 k_2 k_4, \\
\end{align*}

By the calculation of $\mathbf{d}$-vectors, we can check that the claim of the theorem holds for 
the above simple modules. 
Since modules in the same cluster strongly commute mutually, 
the claim of the theorem also holds for 
the simple modules corresponding to a cluster monomial of the above clusters. 

By the above process, for any $a, b, c \in \Z_{\geq 0}$, 
we can realize a simple module whose form is 
\[
L(k_0^{a} k_2^{b} k_4^{c}
\prod_{\substack{i \in I_{\text{fin}}\\i \neq k}}i^{a_i})
\]
for some $a_i \in \Z$ as a cluster monomial. 
It completes the proof. 
\end{proof}

\begin{Prop}\label{Prop : E_n delta imaginary}
    For $l \in \Z_{>0}$, the irreducible module $L(l\delta)$ is imaginary. 
\end{Prop}
\begin{proof}
    Since ${G^{\prime}}^{-1}(p\delta-\alpha_1){G^{\prime}}^{-1}(l\delta)
    ={G^{\prime}}^{-1}((p+l)\delta-\alpha_1)$ for any $p, l\in \Z_{>0}$ by 
    the description of ${G^{\prime}}^{-1}$, 
    we can prove it in the same way as Proposition~\ref{Prop : D_n delta imaginary}. 
\end{proof}

\begin{Prop}\label{Prop : E_n compatible}
    Let $\gamma_1, \gamma_2 \in \Phi_c$ be distinct roots and  
$m_1, m_2\in\Z_{\geq1}$. 
If $\gamma_1$ and $\gamma_2$ are $c$-compatible, 
then two irreducible modules $L(m_1\gamma_1)$, $L(m_2\gamma_2)$ commute strongly. 
\end{Prop}
\begin{proof}
    Let $\beta_{n-3}, \beta_{n-2}, \beta_{n-1}$ be the root as in Proposition~\ref{Prop : E_n Lambda_c} for 
    the case of $E_{n}^{(1)}$. 
    Since 
    \begin{align*}
        {G^{\prime}}^{-1}(\beta_{n-3}+\beta_{n-2})^{m_1}
        {G^{\prime}}^{-1}(\delta-\beta_{n-3}-\beta_{n-2})^{m_1}
        ={G^{\prime}}^{-1}(m_1\delta), \\
        {G^{\prime}}^{-1}(\beta_{n-1})^{m_1}{G^{\prime}}^{-1}(\delta-\beta_{n-1})^{m_1}
        ={G^{\prime}}^{-1}(m_1\delta)
    \end{align*}
    hold by the proof of Proposition~\ref{Prop : E_n bij}, 
    we can prove it in the same way as Proposition~\ref{Prop : D_n compatible}. 
\end{proof}

Then, we can prove the following facts in the same way as 
Theorem~\ref{Thm : D_n tensor factorization}, Corollary~\ref{Cor : D_n factorization}, 
Corollary~\ref{Cor : D_n using delta}, 
and Corollary~\ref{Cor : D_n cluster monomial}. 

\begin{Thm}\label{Thm : E_n tensor factorization}
    Let $m \in \mathcal{M}^{\prime}$ and 
    $G^{\prime}(m) = \sum_{\alpha \in \Phi_c} m_{\alpha}\alpha$ be a 
    $c$-cluster expansion. 
    Then, 
    \begin{align*}
    L(m) &\cong \bigotimes_{\alpha \in \Phi_c}L(m_{\alpha}\alpha)\\
    &\cong \bigotimes_{\alpha \in \Phi_c^{\text{re}}}
    L(\alpha)^{\otimes m_{\alpha}} \otimes L(m_{\delta}\delta). 
    \end{align*}
\end{Thm}
\begin{Cor}\label{Cor : E_n factorization}
    Let $m \in \mathcal{M}$. 
    We can describe it as $m = \prod_{i=1}^{n}F_i^{a_i} \times m^{\prime}$ for 
    $a_i \in \Z_{\geq 0}, m^{\prime}\in \mathcal{M}^{\prime}$ uniquely. 
    Let $G^{\prime}(m^{\prime}) = \sum_{\alpha \in \Phi_c} m_{\alpha}\alpha$ be a $c$-cluster expansion. 
    Then, we have 
    \begin{align*}
    L(m)&\cong \bigotimes _{i=1}^{n}L(F_i)^{\otimes a_i} \otimes 
    \bigotimes_{\alpha \in \Phi_c}L(m_{\alpha}\alpha)\\
    &\cong \bigotimes _{i=1}^{n}L(F_i)^{\otimes a_i} \otimes
    \bigotimes_{\alpha \in \Phi_c^{\text{re}}}
    L(\alpha)^{\otimes m_{\alpha}} \otimes L(m_{\delta}\delta). 
    \end{align*}
\end{Cor}
\begin{Cor}\label{Cor : E_n using delta}
    Let $m \in \mathcal{M}$. 
   Then, $L(m)$ is a real module if and only if 
   $G(m) \in Q$ has a real $c$-cluster expansion. 
   It is also equivalent to $G(m) \not\in \delta + \sum_{\alpha\in \Lambda_c}\Z_{\geq 0}\alpha$. 
\end{Cor}
\begin{Cor}\label{Cor : E_n clyster monomial}
    Let $m\in \mathcal{M}$. 
    Then, $L(m)$ is real if and only if $[L(m)]$ is a cluster monomial. 
\end{Cor}

\begin{Ex}
    By the proof of Proposition~\ref{Prop : E_n bij}, 
    \[
    {G^{\prime}}^{-1}(\delta)= 
    \begin{cases}
        1_32_0^23_3^24_04_45_0^26_3 \quad &\text{ if }n=6, \\
        1_01_42_3^23_0^44_3^25_3^36_0^27_3\quad &\text{ if }n=7, \\
        1_0^22_3^43_0^64_3^35_3^56_0^47_3^28_08_4\quad &\text{ if }n=8. \\
    \end{cases}
    \]
    By Proposition~\ref{Prop : E_n delta imaginary}, 
    $L({G^{\prime}}^{-1}(\delta))$ is imaginary. 
    By Corollary~\ref{Cor : E_n using delta}, 
    if $L(m)$ is imaginary,  
    then $m$ is divisible by ${G^{\prime}}^{-1}(\delta)$. 
    So, $L(\delta)$ is a prime imaginary module. 
\end{Ex}

\subsection{$\mathscr{C}_{A_{n}}^{[1, 2n+2], \mathfrak{s}^1}(n\geq 5)$ 
(The module category of type $A_{n}$ as a categorification of 
the $D_{n+1}^{(1)}$-type cluster algebra)}

Let $\mathfrak{g}$ be of type $A_{n}$ and $\mathcal{C}$ be the monoidal category $\mathscr{C}_{A_{n}}^{[1, 2n+2], \mathfrak{s}^1}$ 
for $n \geq 5$, 
where the admissible sequence $\mathfrak{s}^1$ is of Example~\ref{Ex : Cs}. 
By Theorem~\ref{Thm : affine categorification}, 
this category admits a monoidal categorification of the cluster algebra of type $D_{n+1}^{(1)}$. 

\begin{Rem}\label{rem : newDn+1}
  In this subsection, for simplicity, 
  we change the label of simple roots of $D_{n+1}^{(1)}$ 
  from \S\ref{Sec : affinecluster} as follows : 
  
  
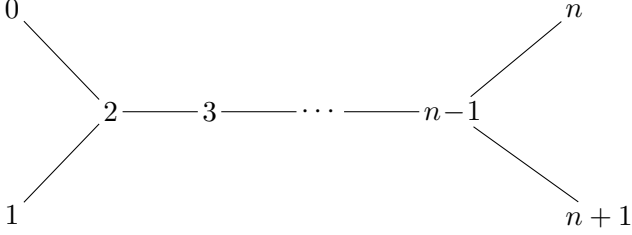
\begin{tikzpicture}[every node/.style={inner sep=1.5pt}]
  \node (1) at (-30, 0){0};
  \node (3) [below right=of 1] {2};
  \node (2) [below left=of 3] {1};
  \node (4) [right=of 3] {3};
  \node (5) [right=of 4] {$\cdots$};
  \node (6) [right=of 5] {$n\!-\!1$};
  \node (7) [above right=of 6]{$n$};
  \node (8) [below right=of 6]{$n+1$};

  \draw (1)--(3)--(4)--(5)--(6)--(7);
  \draw (2)--(3);
  \draw (6)--(8);
\end{tikzpicture}

Then, the null root is 
$\delta = \alpha_0+\alpha_1+2\sum_{i=2}^{n-1}\alpha_i + \alpha_n + \alpha_{n+1}$. 
Also, the set $\Lambda_c^{\mathrm{re}}$ of $D_{n+1}^{(1)}$ under this label is 
described as a disjoint union of three components 

\begin{align*}
        I_1 &= \{\sum_{i=k}^{l}\beta_i, \delta-\sum_{i=k}^{l}\beta_i\mid 
        1 \leq k\leq l\leq n-2\}, \\
        I_2 &= \{\beta_{n-1}, \delta-\beta_{n-1}\}, \\
        I_3 &= \{\beta_{n}, \delta-\beta_{n}\}, 
    \end{align*}
    where 
    \begin{align*}
        \beta_i &= \alpha_{2i}+\alpha_{1+2i} \quad (1 \leq i<[n/2]), \\
        \beta_{[n/2]} &= \alpha_{n-1}+\alpha_n+\alpha_{n+1}, \\
        \beta_j &= \alpha_{2n-2j}+\alpha_{2n-2j+1}\quad ([n/2]<j\leq n-2), \\
        \beta_{n-1} &= \alpha_1+\alpha_2+\cdots +\alpha_{n-2}+\alpha_{n-1}+\alpha_{n}, \\
        \beta_{n} &= \alpha_1+\alpha_2+\cdots +\alpha_{n-2}+\alpha_{n-1}+\alpha_{n+1}. \\
    \end{align*}
\end{Rem}

As the initial monoidal seed, 
we take the seed constructed by Theorem~\ref{Thm : affine categorification}. 
By Lemma~\ref{Lem : mutation monomial}, 
this seed (ignoring the arrows between frozen nodes) is explicitly given as follows : 

\begin{align*}
&n=2k-1(k\geq 3)\\
&\raisebox{2.3em}{\scalebox{0.7}
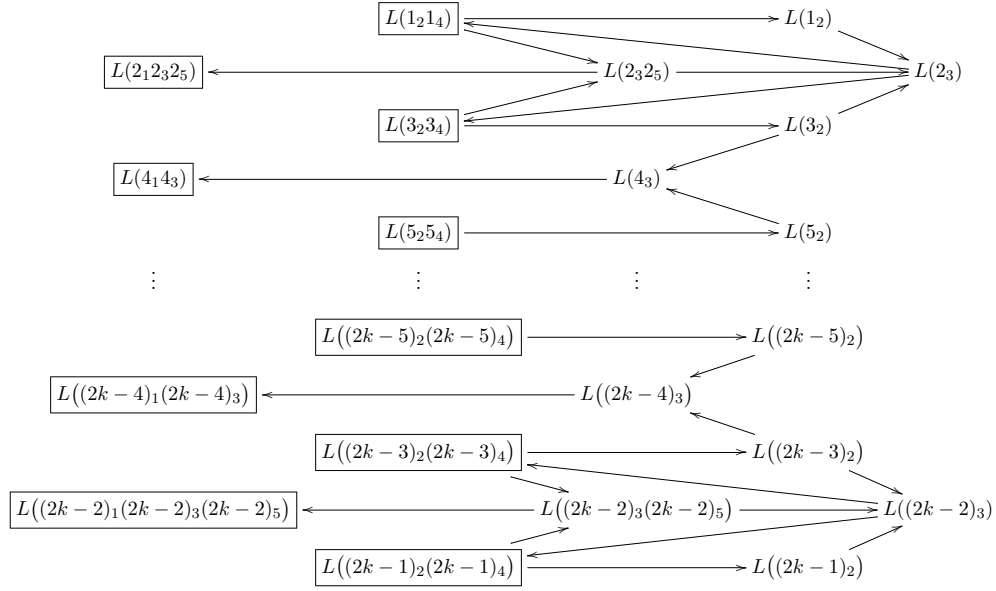
{\xymatrix@C=0.7ex@R=1ex{
 && \boxed{L(1_2 1_4)}\ar[rr]\ar[dr]  && L(1_2)\ar[dr]\\
 & \boxed{L(2_1 2_3 2_5)}  && L(2_3 2_5)\ar[ll]\ar[rr] &&L(2_3)\ar[ulll]\ar[dlll]\\
 &&\boxed{L(3_2 3_4)} \ar[rr]\ar[ur] && L(3_2) \ar[ur]\ar[dl] \\
 & \boxed{L(4_1 4_3)} && L(4_3)\ar[ll]\\
 &&\boxed{L(5_2 5_4)} \ar[rr] && L(5_2) \ar[ul]\\
 & \vdots & \vdots &\vdots&\vdots \\
 && \boxed{L\bigl((2k-5)_2 (2k-5)_4\bigr)} \ar[rr] && L\bigl((2k-5)_2\bigr)\ar[dl]\\
 &\boxed{L\bigl((2k-4)_1 (2k-4)_3 \bigr)} 
&& L\bigl((2k-4)_3\bigr) \ar[ll] \\
 && \boxed{L\bigl((2k-3)_2 (2k-3)_4\bigr)} \ar[rr]\ar[dr]
 && L\bigl((2k-3)_2\bigr)\ar[ul]\ar[dr]\\
 &\boxed{L\bigl((2k-2)_1 (2k-2)_3 (2k-2)_5\bigr)} 
&& L\bigl( (2k-2)_3 (2k-2)_5\bigr) \ar[ll] \ar[rr]
&&L( (2k-2)_3)\ar[ulll]\ar[dlll]\\
 && \boxed{L\bigl((2k-1)_2 (2k-1)_4\bigr)} \ar[ur]\ar[rr] && 
 L\bigl((2k-1)_2\bigr)\ar[ur]\\
}}}
\end{align*}
\begin{align*}
&n=2k(k\geq 3)\\
&\raisebox{2.3em}{\scalebox{0.7}
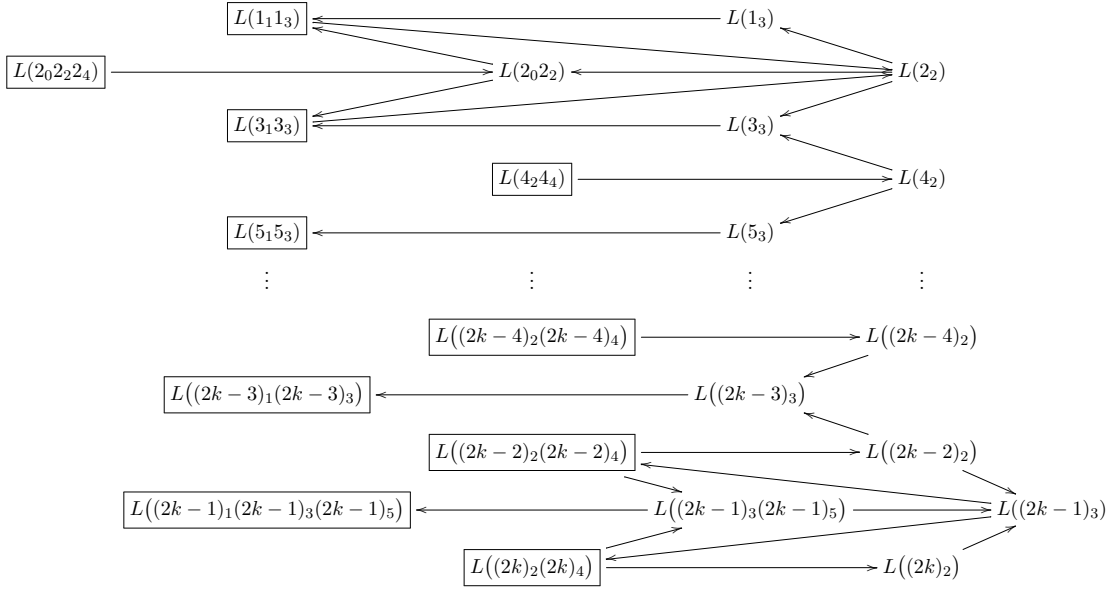
{\xymatrix@C=0.7ex@R=1ex{
 && \boxed{L(1_1 1_3)}\ar[drrr]  && L(1_3)\ar[ll]\\
 & \boxed{L(2_0 2_2 2_4)}\ar[rr]  && L(2_0 2_2)\ar[dl]\ar[ul] && 
 L(2_2) \ar[ll]\ar[ul]\ar[dl]\\
 &&\boxed{L(3_1 3_3)} \ar[urrr] && L(3_3) \ar[ll] \\
 &&& \boxed{L(4_2 4_4)} \ar[rr] && L(4_2)\ar[ul]\ar[dl]\\
 &&\boxed{L(5_1 5_3)}  && L(5_3) \ar[ll]\\
 && \vdots & \vdots &\vdots&\vdots \\
 &&& \boxed{L\bigl((2k-4)_2 (2k-4)_4\bigr)} \ar[rr] && L\bigl((2k-4)_2\bigr)\ar[dl]\\
 &&\boxed{L\bigl((2k-3)_1 (2k-3)_3 \bigr)} 
&& L\bigl((2k-3)_3\bigr) \ar[ll] \\
 &&& \boxed{L\bigl((2k-2)_2 (2k-2)_4\bigr)} \ar[rr]\ar[dr]
 && L\bigl((2k-2)_2\bigr)\ar[ul]\ar[dr]\\
 &&\boxed{L\bigl((2k-1)_1 (2k-1)_3 (2k-1)_5\bigr)} 
&& L\bigl( (2k-1)_3 (2k-1)_5\bigr) \ar[ll] \ar[rr]
&&L( (2k-1)_3)\ar[ulll]\ar[dlll]\\
 &&& \boxed{L\bigl((2k)_2 (2k)_4\bigr)} \ar[ur]\ar[rr] && 
 L\bigl((2k)_2\bigr)\ar[ur]\\
}}}
\end{align*}

We set $I_{\text{fin}} = \{1, \ldots, n\}$ and $I_{\text{fin}} = I_0 \sqcup I_1$, 
where 
\[I_0 \coloneq 
\begin{cases}
    \{2, 4, 6, \ldots, n-1\} \quad \text{for }n=2k-1\\
    \{1, 3, 5, \ldots, n-1\} \quad \text{for }n=2k
\end{cases}
\]
and 
\[
I_1 \coloneq 
\begin{cases}
    \{1, 3, 5, \ldots, n\} \quad \text{for }n=2k-1\\
    \{2, 4, 6, \ldots, n\} \quad \text{for }n=2k
\end{cases}.
\]

Let $\mathcal{M} \coloneq \mathcal{M}_+^{[1, 2n+2], \mathfrak{s}^1}$ be the set of monomials in 
the indeterminates
\[
\begin{cases}
    \{i_1, i_3\mid i \in I_0\}\cup 
\{j_2, j_4\mid j \in I_1\}\cup 
\{2_5, (n-1)_5\} \quad\text{for }n=2k-1, \\
\{i_1, i_3\mid i \in I_0\}\cup 
\{j_2, j_4\mid j \in I_1\}\cup 
\{2_0, (n-1)_5\} \quad\text{for }n=2k, 
\end{cases}
\]
which parametrizes $\Irr \mathcal{C}$. 
For $i \in I_{\text{fin}}\setminus\{2\}$, we define $F_i \in \mathcal{M}$ by 
\[
F_i = 
\begin{cases}
    i_1i_3 \quad &\text{if } i \in I_0\\
    i_2i_4 \quad &\text{if } i \in I_1\setminus \{n-1\}\\
    (n-1)_1(n-1)_3(n-1)_5 \quad &\text{if } i = n-1. 
\end{cases}
\]
Also, we define 
\[
F_2 =
\begin{cases}
    2_1 2_3 2_5 \quad \text{if }n=2k-1, \\
    2_0 2_2 2_4 \quad \text{if }n=2k. 
\end{cases}
\]

We define the subset $\mathcal{M}^{\prime}$ of $\mathcal{M}$ by 
\[
\mathcal{M}^{\prime} \coloneq \{m \in \mathcal{M}\mid 
\text{
$m$ is not divisible by every $F_i$ for $i \in I_{\text{fin}}$
}
\}. 
\]
Let $K=[0, 2n+1], K^{\text{ex}}=[0,n+1]$, and $K^{\text{fr}}=[n+2, 2n+1]$. 
Let $\mathcal{S}_0 = (\{x_i\}_{i \in [0, 2n+1]}, \tilde{B}_0)$ 
be a seed, where $\tilde{B}_0$ is the exchange matrix associated with the above quiver. 
We write the $n$ frozen variables $x_{n+2}, \ldots, x_{2n+1}$ 
as $f_1, \ldots, f_{n}$, respectively.
Then by Theorem~\ref{Thm : affine categorification}, there is an isomorphism 
\begin{align*}
  \iota : K(\mathcal{C}) \overset{\sim}{\longrightarrow}\mathscr{A} \quad; 
  &\quad [L(i_3)] \mapsto x_i  &&\text{ for } i \in I_0\\
  &\quad [L(j_2)] \mapsto x_j &&\text{ for } j \in I_1\\
  &\quad
  \begin{cases}
      [L(2_3 2_5)] \quad(n=2k-1)\\
      [L(2_0 2_2)] \quad(n=2k)
  \end{cases}
  \mapsto x_0\\
  &\quad [L\bigl((n-1)_3(n-1)_5\bigr)] \mapsto x_{n+1} &\\
  &\quad [L(F_i)] \mapsto f_i &&\text{ for } i \in I_{\text{fin}}. 
\end{align*}
We sometimes identify $K(\mathcal{C})$ with $\mathscr{A}$ by $\iota$. 

First, we combinatorially determine the $\mathbf{d}$-vector of each simple representation 
in $\mathscr{A}$.
Let $Q = \bigoplus_{i=0}^{n+1}\Z\alpha_i$ be the root lattice of $D_{n+1}^{(1)}$. 

For each $m \in \mathcal{M}$, there is a unique expression of the form as follows : 
\begin{align*}
\underline{n=2k-1}\\
    m = &\prod_{i\in I_{\text{fin}}}F_i^{a_i} \times 
    \prod_{\substack {i \in I_0\\i\neq 2, n-1}}i_3^{p_i}i_1^{q_i} \times 
    \prod_{j \in I_1}j_2^{p_j}j_4^{q_j} \\
    &\times (2_32_5)^p(2_12_5)^q(2_12_3)^r\\
    &\times 2_1^s 2_3^t 2_5^u\\
    &\times \bigl((n-1)_3(n-1)_5\bigr)^{p^{\prime}}\bigl((n-1)_1(n-1)_5\bigr)^{q^{\prime}}
    \bigl((n-1)_1(n-1)_3\bigr)^{r^{\prime}}\\
    &\times (n-1)_1^{s^{\prime}} (n-1)_3^{t^{\prime}}
    (n-1)_5^{u^{\prime}}, 
\end{align*}
\begin{align*}
\underline{n=2k}\\
    m = &\prod_{i\in I_{\text{fin}}}F_i^{a_i} \times 
    \prod_{\substack{i \in I_0\\i \neq n-1}}i_3^{p_i}i_1^{q_i} \times 
    \prod_{\substack{j \in I_1\\j\neq 2}}j_2^{p_j}j_4^{q_j} \\
    &\times (2_0 2_2)^p(2_0 2_4)^q(2_2 2_4)^r\\
    &\times 2_4^s 2_2^t 2_0^u\\
    &\times \bigl((n-1)_3(n-1)_5\bigr)^{p^{\prime}}\bigl((n-1)_1(n-1)_5\bigr)^{q^{\prime}}
    \bigl((n-1)_1(n-1)_3\bigr)^{r^{\prime}}\\
    &\times (n-1)_1^{s^{\prime}} (n-1)_3^{t^{\prime}}
    (n-1)_5^{u^{\prime}}, 
\end{align*}
such that each of the following tuples contains at most one nonzero entry :  

$(p_i, q_i)$ for each $i \in I_{\text{fin}}\setminus \{2, n\!-\!1\}$, 
$(p, q, r)$, $(s, t, u)$, $(p, s)$, $(q, t)$, $(r, u)$, 
$(p^{\prime}, q^{\prime}, r^{\prime})$, 
$(s^{\prime}, t^{\prime}, u^{\prime})$, 
$(p^{\prime}, s^{\prime})$, $(q^{\prime}, t^{\prime})$, $(r^{\prime}, u^{\prime})$. 

Then, we define a map $G : \mathcal{M} \to Q$ by 

\begin{align*}
    G(m) = &\sum_{\substack {i \in I_{\text{fin}}\\i\neq 2, n-1}}
    \bigl(p_i(-\alpha_i) + q_i\alpha_i\bigr) \\
    &+ p(-\alpha_0) + q(\alpha_0 + \alpha_1 + 2\alpha_2 + \alpha_3) 
    + r\alpha_0\\
    &+ s(\alpha_0 + \alpha_2) + t(-\alpha_2) 
    + u(\alpha_1 + \alpha_2 + \alpha_3)\\
    &+ p^{\prime}(-\alpha_{n+1}) 
    + q^{\prime}(\alpha_{n-2} + 2\alpha_{n-1} + \alpha_n + \alpha_{n+1}) 
    + r^{\prime}\alpha_{n+1}\\
    &+ s^{\prime}(\alpha_{n-1} + \alpha_{n+1}) + t^{\prime}(-\alpha_{n-1}) 
    + u^{\prime}(\alpha_{n-2} + \alpha_{n-1} + \alpha_n). 
\end{align*}

Let $G' : \mathcal{M}' \to Q$ denote the restriction of $G$ to $\mathcal{M}'$. 
By definition, for $m \in \mathcal{M}^{\prime}$, 
$G^{\prime}(m) = G(\prod_{i \in {I_{\text{fin}}}}F_i^{a_i} \times m)$ for any $a_i$. 

In this subsection, 
for $i \in I_{\text{fin}}\setminus\{2, n-1\}$ and $a \in \Z$, 
    we write 
    \[
    i^a \coloneq 
    \begin{cases}
        i_1^a \quad &\text{ if } i\in I_0, a\geq 0, \\
        i_3^{-a} \quad &\text{ if } i \in I_0, a < 0, \\
        i_4^a \quad &\text{ if } i \in I_1\setminus\{n\!-\!1\}, a \geq 0, \\
        i_2^{-a} \quad &\text{ if } i\in I_1\setminus\{n\!-\!1\}, a < 0, 
    \end{cases}
    \]
    for simplicity. 
    Then, 
    $G^{\prime}\biggl(\prod_{i \in I_{\text{fin}}\setminus\{2, n-1\}}i^{a_i}\biggr) 
    = \sum_{i \in I_{\text{fin}}\setminus\{2, n-1\}}a_i \alpha_i$ for $a_i \in \Z$. 

    For $\gamma = \sum_{i=0}^{n+1}a_i \alpha_i \in Q$, 
    we set monomials $m_1, m_2 \in \mathcal{M}^{\prime}$ as follows$\colon$

    If $n=2k-1$, then we define 
\[
m_1 \coloneq 
\begin{cases}
    (2_1 2_3)^{a_0-a_2}2_1^{a_2} \quad &\text{if } (a_0, a_2)\in A,\\
    (2_1 2_5)^{a_2-a_0}2_1^{2a_0-a_2} \quad &\text{if }(a_0, a_2)\in B,\\
    (2_1 2_5)^{a_0}2_5^{a_2-2a_0} \quad &\text{if }(a_0, a_2)\in C,\\
    (2_3 2_5)^{-a_0}2_5^{a_2} \quad &\text{if }(a_0, a_2)\in D,\\
    (2_3 2_5)^{-a_0}2_3^{-a_2} \quad &\text{if }(a_0, a_2)\in E,\\
    (2_1 2_3)^{a_0}2_3^{-a_2} \quad &\text{if }(a_0, a_2)\in F.\\
\end{cases}
\]
If $n=2k$, then we define 
\[
m_1 \coloneq 
\begin{cases}
    (2_2 2_4)^{a_0-a_2}2_4^{a_2} \quad &\text{if } (a_0, a_2)\in A,\\
    (2_0 2_4)^{a_2-a_0}2_4^{2a_0-a_2} \quad &\text{if }(a_0, a_2)\in B,\\
    (2_0 2_4)^{a_0}2_0^{a_2-2a_0} \quad &\text{if }(a_0, a_2)\in C,\\
    (2_0 2_2)^{-a_0}2_0^{a_2} \quad &\text{if }(a_0, a_2)\in D,\\
    (2_0 2_2)^{-a_0}2_2^{-a_2} \quad &\text{if }(a_0, a_2)\in E,\\
    (2_2 2_4)^{a_0}2_2^{-a_2} \quad &\text{if }(a_0, a_2)\in F.\\
\end{cases}
\]
Here, $A, \ldots, F$ are as follows. 

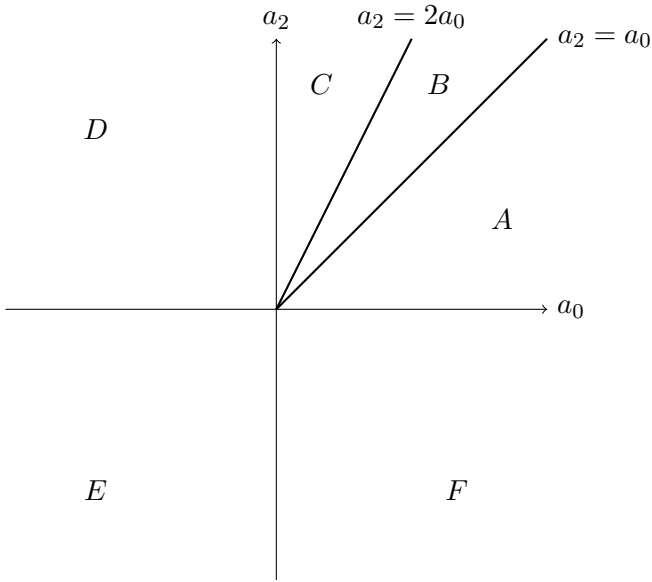
\begin{tikzpicture}[scale=1.2]

  \draw[->] (-3,0) -- (3,0) node[right] {$a_0$};
  \draw[->] (0,-3) -- (0,3) node[above] {$a_2$};

  \draw[thick]   (0,0) -- (1.5,3) node[above] {$a_2=2a_0$};
  \draw[thick]  (0, 0)  -- (3,3) node[right] {$a_2=a_0$};

  \node at (2.5,1){$A$};
  \node at (1.8,2.5) {$B$};
  \node at (0.5,2.5) {$C$};
  \node at (-2,2) {$D$};
  \node at (-2,-2) {$E$};
  \node at (2, -2) {$F$};

\end{tikzpicture}

We set 
\[
m_2 \coloneq 
\begin{cases}
    ((n-1)_1 (n-1)_3)^{a_{n+1}-a_{n-1}}(n-1)_1^{a_{n-1}} \quad 
    &\text{if } (a_{n+1}, a_{n-1})\in A^{\prime},\\
    ((n-1)_1 (n-1)_5)^{a_{n-1}-a_{n+1}}(n-1)_1^{2a_{n+1}-a_{n-1}} \quad 
    &\text{if }(a_{n+1}, a_{n-1})\in B^{\prime},\\
    ((n-1)_1 (n-1)_5)^{a_{n+1}}(n-1)_5^{a_{n-1}-2a_{n+1}} \quad 
    &\text{if }(a_{n+1}, a_{n-1})\in C^{\prime},\\
    ((n-1)_3 (n-1)_5)^{-a_{n+1}}(n-1)_5^{a_{n-1}} \quad 
    &\text{if }(a_{n+1}, a_{n-1})\in D^{\prime},\\
    ((n-1)_3 (n-1)_5)^{-a_{n+1}}(n-1)_3^{-a_{n-1}} \quad 
    &\text{if }(a_{n+1}, a_{n-1})\in E^{\prime},\\
    ((n-1)_1 (n-1)_3)^{a_{n+1}}(n-1)_3^{-a_{n-1}} \quad 
    &\text{if }(a_{n+1}, a_{n-1})\in F^{\prime}, \\
\end{cases}
\]
    where $A^{\prime}, \ldots, F^{\prime}$ are as follows. 

    
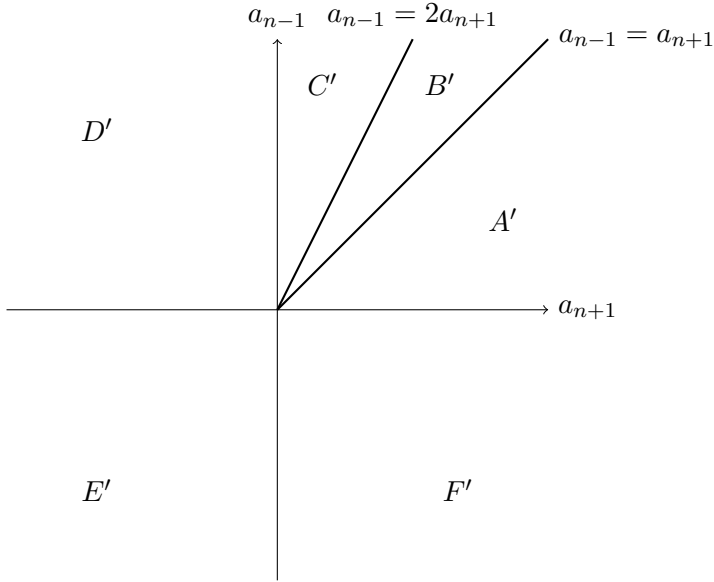
\begin{tikzpicture}[scale=1.2]

  \draw[->] (-3,0) -- (3,0) node[right] {$a_{n+1}$};
  \draw[->] (0,-3) -- (0,3) node[above] {$a_{n-1}$};

  \draw[thick]   (0,0) -- (1.5,3) node[above] {$a_{n-1}=2a_{n+1}$};
  \draw[thick]  (0, 0)  -- (3,3) node[right] {$a_{n-1}=a_{n+1}$};

  \node at (2.5,1){$A^{\prime}$};
  \node at (1.8,2.5) {$B^{\prime}$};
  \node at (0.5,2.5) {$C^{\prime}$};
  \node at (-2,2) {$D^{\prime}$};
  \node at (-2,-2) {$E^{\prime}$};
  \node at (2, -2) {$F^{\prime}$};

\end{tikzpicture}

Then, we can check that $\gamma$ can be written in the form 
\[
\gamma = G^{\prime}(m_1 m_2) + 
\sum_{\substack{i \in \{0, 1, \ldots, n+1\}\\i \neq 0, 2, n-1, n+1}}b_i \alpha_i 
= G^{\prime}(m_1 m_2) + 
\sum_{\substack{i \in I_{\text{fin}}\\i \neq 2, n-1}}b_i \alpha_i, 
\]
where $b_i \in \Z$ for $i \in I_{\text{fin}}\setminus \{2, n-1\}$. 
By the construction, the map 
\[
\sum_{i=0}^{n+1}a_i \alpha_i \longmapsto m_1 m_2 \times 
\prod_{\substack{i \in I_{\text{fin}}\\i \neq 2, n-1}}i^{b_i}
\]
    is the inverse map of $G^{\prime}$. 
    As a result, the following Proposition holds. 

\begin{Prop}\label{Prop : A_n bij}
    $G' : \mathcal{M}' \to Q$ is bijective. 
\end{Prop}

\begin{Thm}\label{Thm : A_n G=dL}
    $G = \mathbf{d}\circ \iota \circ [L(-)]$. 
\end{Thm}
\begin{proof}
Note that for $m\in \mathcal{M}$ and $i \in I_{\text{fin}}$, 
the equalities $G(m) = G(F_i m)$ and 
$\mathbf{d}\circ \iota \circ [L(m)] = \mathbf{d}\circ \iota \circ [L(F_im)]$ hold, 
as in the proof of Theorem~\ref{Thm : D_n G=dL}. 
We can check that the statement holds 
for the case $m = i^{\pm 1}$ for $i \in I_{\text{fin}}\setminus\{2, n-1\}$ by considering 
the initial seed $\mathcal{S}_0$ and $\mu_i(\mathcal{S}_0)$. 

For $m \in \mathcal{M}$ and $a_i \in \Z$ for $i \in I_{\text{fin}}\setminus\{2, n-1\}$, 
we have 
\[
\mathbf{d}\circ \iota \circ [L(m \prod_{\substack{i \in I_{\text{fin}}\\i \neq 2, n-1}}i^{a_i})]
\leq \mathbf{d}\circ \iota \circ [L(m)] 
+ \sum_{\substack{i \in I_{\text{fin}}\\i \neq 2, n-1}}a_i\alpha_i
\]
and 
\[
\mathbf{d}\circ \iota \circ [L(m \prod_{\substack{i \in I_{\text{fin}}\\i \neq 2, n-1}}i^{a_i})]
+\sum_{\substack{i \in I_{\text{fin}}\\i \neq 2, n-1}}a_i(-\alpha_i)
\geq \mathbf{d}\circ \iota \circ 
[L(m \prod_{\substack{i \in I_{\text{fin}}\\i \neq 2, n-1}}F_i^{a_i})]
=\mathbf{d}\circ \iota \circ [L(m)]
\]
by Lemma~\ref{Lem : dvector inequality}. 
This shows 
\[
\mathbf{d}\circ \iota \circ [L(m \prod_{\substack{i \in I_{\text{fin}}\\i \neq 2, n-1}}i^{a_i})]
=\mathbf{d}\circ \iota \circ [L(m)]
+\sum_{\substack{i \in I_{\text{fin}}\\i \neq 2, n-1}}a_i\alpha_i. 
\]
By the definition of $G$, 
for $m \in \mathcal{M}$ and $a_i \in \Z$ for $i \in I_{\text{fin}}\setminus\{2, n-1\}$, 
\[G(m \prod_{\substack{i \in I_{\text{fin}}\\i \neq 2, n-1}}i^{a_i}) = 
G(m) +\sum_{\substack{i \in I_{\text{fin}}\\i \neq 2, n-1}}a_i\alpha_i
\]
also holds. 

Let us consider the case $n=2k-1$. 
By the above fact, it is enough to show that 
for any $a, b, c, a^{\prime},  b^{\prime},  c^{\prime} \in \Z_{\geq 0}$, 
there exist $a_i \in \Z$ for $i \in I_{\text{fin}}\setminus\{2, n-1\}$ 
such that 
the claim holds for 
\[m = 2_1^a 2_3^b 2_5^c (n-1)_1^{a^{\prime}}(n-1)_3^{b^{\prime}}(n-1)_5^{c^{\prime}}
\prod_{\substack{i \in I_{\text{fin}}\\i \neq 2, n-1}}i^{a_i}. 
\]

By Lemma~\ref{Lem : mutation monomial}, we have that the following simple modules are in the 
same monoidal cluster : 

\begin{equation}\label{eq : A_2k-1 2}
\begin{aligned}
\mathcal{S}_0 &\ni& &2_3,& &2_3 2_5,& &2_1 2_3 2_5, \\
\mu_0(\mathcal{S}_0) &\ni& &2_3,& &2_1 2_3,& &2_1 2_3 2_5, \\
\mu_2(\mathcal{S}_0) &\ni& &2_5 1_2 3_2,& &2_3 2_5,& &2_1 2_3 2_5, \\
\mu_2\mu_0(\mathcal{S}_0) &\ni& &2_1,& &2_1 2_3,& &2_1 2_3 2_5, \\
\mu_0\mu_2(\mathcal{S}_0) &\ni& &2_5 1_2 3_2,& &2_1,& &2_1 2_3 2_5. \\
\end{aligned}
\end{equation}

Also, the following simple modules are in the same monoidal cluster : 

\begin{equation}\label{eq : A_2k-1 n-1}
\begin{aligned}
\mathcal{S}_0 &\ni& &(n\!-\!1)_3,& &(n\!-\!1)_3 (n\!-\!1)_5,& 
&(n\!-\!1)_1 (n\!-\!1)_3 (n\!-\!1)_5, \\
\mu_{n\!+\!1}(\mathcal{S}_0) &\ni& &(n\!-\!1)_3,& &(n\!-\!1)_1 (n\!-\!1)_3,& 
&(n\!-\!1)_1 (n\!-\!1)_3 (n\!-\!1)_5, \\
\mu_{n\!-\!1}(\mathcal{S}_0) &\ni& &(n\!-\!1)_5 (n\!-\!2)_2 n_2,& &(n\!-\!1)_3 (n\!-\!1)_5,& 
&(n\!-\!1)_1 (n\!-\!1)_3 (n\!-\!1)_5, \\
\mu_{n\!-\!1}\mu_{n\!+\!1}(\mathcal{S}_0) &\ni& &(n\!-\!1)_1,& &(n\!-\!1)_1 (n\!-\!1)_3,& 
&(n\!-\!1)_1 (n\!-\!1)_3 (n\!-\!1)_5, \\
\mu_{n\!+\!1}\mu_{n\!-\!1}(\mathcal{S}_0) &\ni& &(n\!-\!1)_5 (n\!-\!2)_2 n_2,& &(n\!-\!1)_1,& 
&(n\!-\!1)_1 (n\!-\!1)_3 (n\!-\!1)_5. \\
\end{aligned}
\end{equation}

Let $\psi_1, \psi_2$ be operations chosen from the set 
$\{\phi, \mu_0, \mu_2, \mu_2\mu_0, \mu_0\mu_2\}$ and 
$\{\phi, \allowbreak \mu_{n-1}, \allowbreak 
\mu_{n+1}, \allowbreak \mu_{n-1}\mu_{n+1}, \mu_{n+1}\mu_{n-1}\}$ respectively. 
Let $\mathscr{L}_{\psi_1}, \mathscr{L}_{\psi_2}$ denote the set of three simple modules 
listed in the right-hand side of \eqref{eq : A_2k-1 2}  and \eqref{eq : A_2k-1 n-1} 
corresponding to the cluster $\psi_1(\mathcal{S}_0), \psi_2(\mathcal{S}_0)$ respectively.
Since there are no arrows in $\mathcal{S}_0$ 
between any vertex in $\{0, 2\}$ and any vertex in $\{n-1, n+1\}$, 
the result of applying the operations $\psi_1, \psi_2$ to $\mathcal{S}_0$ 
is independent of the order in which they are applied. 
This implies that for any choice of $\psi_1$ and $\psi_2$, 
the union $\mathscr{L}_{\psi_1}\cup\mathscr{L}_{\psi_2}$ is 
contained in the monoidal seed $\psi_1\psi_2(\mathcal{S}_0)$. 


By the calculation of $\mathbf{d}$-vectors, we can check that the claim of the theorem holds for 
the above simple modules. 
By the above process, for any $a, b, c, a^{\prime},  b^{\prime},  c^{\prime}\in \Z_{\geq 0}$, 
we can realize a simple module whose form is 
\[
L(2_1^a 2_3^b 2_5^c (n-1)_1^{a^{\prime}}(n-1)_3^{b^{\prime}}(n-1)_5^{c^{\prime}}
\prod_{\substack{i \in I_{\text{fin}}\\i \neq 2, n-1}}i^{a_i})
\]
for some $a_i \in \Z$ as a cluster monomial. 
It completes the proof for the case $n=2k-1$. 

We can prove the case $n=2k$ by replacing $2_1, 2_3, 2_5$ with 
$2_4, 2_2, 2_0$ respectively. 
\end{proof}

\begin{Prop}\label{Prop : A_n delta imaginary}
    For $l \in \Z_{>0}$, the irreducible module $L(l\delta)$ is imaginary. 
\end{Prop}
\begin{proof}
    Since ${G^{\prime}}^{-1}(p\delta-\alpha_1){G^{\prime}}^{-1}(l\delta)
    ={G^{\prime}}^{-1}((p+l)\delta-\alpha_1)$ for any $p, l\in \Z_{>0}$ by 
    the description of ${G^{\prime}}^{-1}$, 
    we can prove it in the same way as Proposition~\ref{Prop : D_n delta imaginary}. 
\end{proof}

\begin{Prop}\label{Prop : A_n compatible}
    Let $\gamma_1, \gamma_2 \in \Phi_c$ be distinct roots and  
$m_1, m_2\in\Z_{\geq1}$. 
If $\gamma_1$ and $\gamma_2$ are $c$-compatible, 
then two irreducible modules $L(m_1\gamma_1)$, $L(m_2\gamma_2)$ commute strongly. 
\end{Prop}
\begin{proof}
    Let $\beta_{n-1}, \beta_n$ be the roots as in Remark~\ref{rem : newDn+1}. 
    Since 
    \begin{align*}
        {G^{\prime}}^{-1}(\beta_{n-1})^{m_1}{G^{\prime}}^{-1}(\delta-\beta_{n-1})^{m_1}
        ={G^{\prime}}^{-1}(m_1\delta), \\
        {G^{\prime}}^{-1}(\beta_{n})^{m_1}{G^{\prime}}^{-1}(\delta-\beta_{n})^{m_1}
        ={G^{\prime}}^{-1}(m_1\delta)
    \end{align*}
    holds by the proof of Proposition~\ref{Prop : A_n bij}, 
    we can prove it in the same way as Proposition~\ref{Prop : D_n compatible}. 
\end{proof}

Then, we can prove the following facts in the same way as 
Theorem~\ref{Thm : D_n tensor factorization}, Corollary~\ref{Cor : D_n factorization}, 
Corollary~\ref{Cor : D_n using delta}, 
and Corollary~\ref{Cor : D_n cluster monomial}. 

\begin{Thm}\label{Thm : A_n tensor factorization}
    Let $m \in \mathcal{M}^{\prime}$ and 
    $G^{\prime}(m) = \sum_{\alpha \in \Phi_c} m_{\alpha}\alpha$ be a 
    $c$-cluster expansion. 
    Then, 
    \begin{align*}
    L(m) &\cong \bigotimes_{\alpha \in \Phi_c}L(m_{\alpha}\alpha)\\
    &\cong \bigotimes_{\alpha \in \Phi_c^{\text{re}}}
    L(\alpha)^{\otimes m_{\alpha}} \otimes L(m_{\delta}\delta). 
    \end{align*}
\end{Thm}
\begin{Cor}\label{Cor : A_n factorization}
    Let $m \in \mathcal{M}$. 
    We can describe it as $m = \prod_{i=1}^{n}F_i^{a_i} \times m^{\prime}$ for 
    $a_i \in \Z_{\geq 0}, m^{\prime}\in \mathcal{M}^{\prime}$ uniquely. 
    Let $G^{\prime}(m^{\prime}) = \sum_{\alpha \in \Phi_c} m_{\alpha}\alpha$ be a $c$-cluster expansion. 
    Then, we have 
    \begin{align*}
    L(m)&\cong \bigotimes _{i=1}^{n}L(F_i)^{\otimes a_i} \otimes 
    \bigotimes_{\alpha \in \Phi_c}L(m_{\alpha}\alpha)\\
    &\cong \bigotimes _{i=1}^{n}L(F_i)^{\otimes a_i} \otimes
    \bigotimes_{\alpha \in \Phi_c^{\text{re}}}
    L(\alpha)^{\otimes m_{\alpha}} \otimes L(m_{\delta}\delta). 
    \end{align*}
\end{Cor}
\begin{Cor}\label{Cor : A_n using delta}
    Let $m \in \mathcal{M}$. 
   Then, $L(m)$ is a real module if and only if 
   $G(m) \in Q$ has a real $c$-cluster expansion. 
   It is also equivalent to $G(m) \not\in \delta + \sum_{\alpha\in \Lambda_c}\Z_{\geq 0}\alpha$. 
\end{Cor}
\begin{Cor}\label{Cor : A_n clyster monomial}
    Let $m\in \mathcal{M}$. 
    Then, $L(m)$ is real if and only if $[L(m)]$ is a cluster monomial. 
\end{Cor}

\begin{Ex}
    By the proof of Proposition~\ref{Prop : A_n bij}, 
    \[
    {G^{\prime}}^{-1}(\delta)= 
    \begin{cases}
    2_12_54_14_5 \quad &\text{if }n=5, \\
    2_12_53_4(n\!-\!2)_4(n\!-\!1)_1(n\!-\!1)_5
    \displaystyle\prod_{i=4, 6, \ldots, n-3}i_1^2
    \prod_{i=5, 7, \ldots, n-4}i_4^2&\text{if }n=2k-1(k\geq4), \\
    2_02_43_1(n\!-\!2)_4(n\!-\!1)_1(n\!-\!1)_5
    \displaystyle\prod_{i=4, 6, \ldots, n-4}i_4^2
    \prod_{i=5, 7, \ldots, n-3}i_1^2 &\text{if }n=2k(k\geq3). \\
    \end{cases}
    \]
    By Proposition~\ref{Prop : A_n delta imaginary}, 
    $L({G^{\prime}}^{-1}(\delta))$ is imaginary. 
    By Corollary~\ref{Cor : A_n using delta}, 
    if $L(m)$ is imaginary,  
    then $m$ is divisible by ${G^{\prime}}^{-1}(\delta)$. 
    So, $L(\delta)$ is a prime imaginary module. 
\end{Ex}

\section{Replacement of Q-datum}\label{Sec : Qdatum}
Let $\mathfrak{g}$ be a complex finite-dimensional simple Lie algebra 
of type ADE. 
The monoidal categories $\mathscr{C}$ argued in the above section 
are subcategories of $\mathscr{C}_{\mathscr{Q}_0}$ 
for some appropriate Q-data $\mathscr{Q}_0$. 
However, by Proposition~\ref{Prop : C_Q equiv}, 
for any Q-datum $\mathscr{Q}$ of $\mathfrak{g}$, 
$\mathscr{C}_{\mathscr{Q}}$ is equivalent to $\mathscr{C}_{\mathscr{Q}_0}$ 
as monoidal categories. 
In this section,  for $\mathscr{C}_{\mathscr{Q}}$ with some 
Q-datum $\mathscr{Q}$, 
we argue what the subcategory in $\mathscr{C}_{\mathscr{Q}}$ 
corresponding to $\mathscr{C} \subset \mathscr{C}_{\mathscr{Q}_0}$ is, 
and find imaginary modules in the corresponding category. 

\subsection{Quiver Hecke algebras}
Let $\mathfrak{g}$ be a finite-dimensional simple Lie algebra 
of type ADE and $I$ be the set of Dynkin indices. 
Let $\Bbbk$ be a field of characteristic $0$. 
We choose polynomials 
\[
Q_{i,j}(u, v) = \delta(i\neq j)
\sum_{\substack{(p, q)\in \Z_{\geq0}^2\\(\alpha_i, 
\alpha_i)p+(\alpha_j, \alpha_j)q=-2(\alpha_i, \alpha_j)}}
t_{i, j; p, q}u^pv^q \in \Bbbk[u, v]
\]
with $t_{i, j; p, q} \in \Bbbk$, 
$t_{i, j; p, q} = t_{j, i; q, p}$ and $t_{i, j; -a_{i, j}, 0}
\in \Bbbk^{\times}$. 
Let $S_n=\langle s_1, \ldots, s_{n-1}\rangle$ be the symmetric group 
on $n$ letters with the action of $S_n$ on $I^n$ by place permutation. 
Let $Q$ be the root lattice of $\mathfrak{g}$ and 
$Q^{+}\coloneq \sum_{i\in I}\Z_{\geq 0}\alpha_i$. 
For $\beta \in Q^{+}$ with $\mathrm{ht}(\beta)=n$, we set 
\[
I^{\beta} \coloneq \{\nu = (\nu_1, \ldots, \nu_n)\in I^n\mid 
\alpha_{\nu_1}+\cdots+\alpha_{\nu_n}=\beta\}. 
\]
\begin{Def}[{\cite{Ro08}, \cite{KL09}}]\label{Def : KLR}
    Let $\beta \in Q^{+}$ with $\mathrm{ht}(\beta)=n$. 
    The \emph{quiver Hecke algebra} $R(\beta)$ associated with the parameters 
    $\{Q_{i, j}(u, v)\}$ is the $\Bbbk$-algebra given by the following 
    generators and relations : 

    generators : $e(\nu) \ (\nu \in I^{\beta}),\quad x_k\ (k=1, \ldots, n),
    \quad \tau_m\ (m=1, \ldots, n-1)$

    relations : 
    \begin{itemize}
        \item $e(\nu)e(\nu^{\prime}) = \delta_{\nu, \nu^{\prime}}e(\nu), 
        \quad \sum_{\nu\in I^{\beta}}e(\nu) = 1, $
        \item $x_kx_m = x_mx_k, \quad x_ke(\nu)=e(\nu)x_k, $
        \item $\tau_me(\nu)=e(s_m\nu)\tau_m$, 
        \item $\tau_k\tau_m=\tau_m\tau_k \quad \mathrm{ if } | k-m | >1$, 
        \item $(\tau_{k+1}\tau_k\tau_{k+1}-\tau_k\tau_{k+1}\tau_k)e(\nu)$
        \[
        =\begin{cases}
            \dfrac{Q_{\nu_k, \nu_{k+1}}(x_k, x_{k+1})-Q_{\nu_{k+2}, \nu_{k+1}}(x_{k+2}, x_{k+1})}
            {x_k-x_{k+2}} \quad &\mathrm{if } \nu_k = \nu_{k+2}, \\
            0  &\mathrm{otherwise}. 
        \end{cases}
        \]
        \item $\tau_k^2e(\nu) = Q_{\nu_k, \nu_{k+1}}(x_k, x_{k+1})e(\nu)$, 
        \item $(\tau_kx_m -x_{s_k(m)}\tau_k)e(\nu) = 
        \begin{cases}
            -e(\nu) \quad &\mathrm{if }m=k, \nu_k = \nu_{k+1}, \\
            e(\nu) \quad &\mathrm{if } m=k+1, \nu_k = \nu_{k+1}, \\
            0 &\mathrm{otherwise. }
        \end{cases}$
    \end{itemize}

    The algebra $R(\beta)$ has the $\Z$-graded algebra structure defined by 
    \[
    \mathrm{deg }e(\nu) = 0, \mathrm{deg }x_ke(\nu) = (\alpha_{\nu_k}, \alpha_{\nu_k}), 
    \mathrm{deg }\tau_le(\nu) = -(\alpha_{\nu_l}, \alpha_{\nu_{l+1}}),  
    \]
    where $(\alpha_i, \alpha_i)=2$ for all $i \in I$. 
\end{Def}

Let $R$-gmod $\coloneq \bigoplus_{\beta\in Q^{+}}R(\beta)$-gmod, 
where $R(\beta)$-gmod is the category of finite-dimensional graded $R(\beta)$-modules. 
For $M \in R(\beta)$-gmod and $N\in R(\gamma)$-gmod, 
we can define their convolution product $M\circ N \in R(\beta + \gamma)$-gmod by 
\[
M\circ N \coloneq R(\beta+\gamma)e(\beta, \gamma)
\underset{R(\beta)\otimes R(\gamma)}{\otimes} (M\otimes N), 
\]
where $e(\beta, \gamma) = \sum_{\nu_1\in I^{\beta}, \nu_2\in I^{\gamma}}e(\nu_1, \nu_2)$. 
Using the $\Z$-grading structure and the convolution product, 
the Grothendieck group $K(R\mathrm{-gmod})$ has a $\Z[q, q^{-1}]$-algebra structure. 

The quiver Hecke algebra $R(\beta)$ is said to be \emph{symmetric} if $Q_{i,j}(u,v)$ is a polynomial 
in $u-v$ for any $i, j \in I$. 
In what follows, we assume that $R$ is symmetric. 

Let $U_q(\mathfrak{g})$ be the quantum group, a $\Q(q)$-algebra generated by 
$e_i, f_i \ (i \in I), \ q^h\ (h \in P^{\vee})$ with relations. 
We define $U_q^{+}(\mathfrak{g})$ to be the sub $\Q(q)$-algebra of $U_q(\mathfrak{g})$ 
generated by $e_i$'s. 
We define $(U_q^{+}(\mathfrak{g}))_{\Z}$ to be the sub $\Z[q, q^{-1}]$-algebra 
generated by $e_i^{(n)}$ for $i \in I$ and $n \in \Z_{\geq 0}$, 
where $e_i^{(n)} = e_i^n/[n]_q!$. 
These algebras have $Q^{+}$-grading where $e_i$ is homogeneous of degree $\alpha_i$. 

We define $A_q(\mathfrak{n}) \coloneq \bigoplus_{\mu \in Q^+}
\mathrm{Hom}_{\Q(q)}((U_q^+(\mathfrak{g})_{\mu}, \Q(q))$ 
to be the graded dual of $U_q^{+}(\mathfrak{g})$. 

There is a symmetric non-degenerate bilinear form 
\[
( \quad , \quad) : U_q^{+}(\mathfrak{g})\times U_q^{+}(\mathfrak{g})\to \Q(q), 
\]
so it induces a bijection $U_q^{+}(\mathfrak{g}) \to A_q(\mathfrak{n})$ (see \cite{KKKO18}). 
Via this bijection, we can endow $A_q(\mathfrak{n})$ with a $\Q(q)$-algebra structure 
isomorphic to $U_q^{+}(\mathfrak{g})$. 
Furthermore, we define $A_q(\mathfrak{n})_{\Z}\coloneq \bigoplus_{\mu \in Q^+}
\mathrm{Hom}_{\Z[q, q^{-1}]}((U_q^+(\mathfrak{g})_{\Z, \mu}, \Z[q, q^{-1}])\subset A_q(\mathfrak{n})$, 
which is a sub $\Z[q, q^{-1}]$-algebra of $A_q(\mathfrak{n})$. 
(For details, see \cite{KKKO18}. Note that the argument there is given for 
$U_q^{-}(\mathfrak{g})$, with $e_i$ in $U_q^{+}(\mathfrak{g})$ 
replaced by $f_i$ in $U_q^{-}(\mathfrak{g})$. )

\begin{Prop}[{\cite{KL09}, \cite{Ro08}}]\label{Prop : KLR iso}
    There exists a $\Z[q, q^{-1}]$-algebra isomorphism 
    \[
    ch : K(R\mathrm{-gmod}) \to A_q(\mathfrak{n})_{\Z}. 
    \]
\end{Prop}

\subsection{Dual PBW bases and dual canonical bases}
Consider a reduced expression $\underline{w_0} = i_1\ldots i_l$ 
of the longest element $w_0$ of the Weyl group. 
For $i \in I$, there is an automorphism of 
$T_i : U_q(\mathfrak{g}) \to U_q(\mathfrak{g})$. 
(See \cite{Lusbook}. Note that, in this paper, we write $v, K_h$ in \cite{Lusbook} 
as $q^{-1}, q^{-h}$ and adopt $T_{i, 1}^{\prime\prime}$ for $T_i$. )

For $\mathbf{c} \in \Z_{\geq 0}^l$, 
we define 
\[
E_{\underline{w_0}}({\mathbf{c}})\coloneq 
T_{i_1}T_{i_2}\cdots T_{i_{l-1}}(e_{i_l}^{(c_l)})\cdots
T_{i_1}(e_{i_2}^{(c_2)})e_{i_1}^{(c_1)}. 
\]
\begin{Prop}[{\cite{Lusbook}}]\label{PBWbasis}
    For any reduced expression $\underline{w_0}$, 
    the set 
    \[
    \{E_{\underline{w_0}}({\mathbf{c}}) \mid \mathbf{c} \in \Z_{\geq 0}^l\}
    \]
    forms a $\Z[q, q^{-1}]$-basis of $(U_q^{+}(\mathfrak{g}))_{\Z}$. 
\end{Prop}
We call the basis \emph{PBW-basis}. 
Then, using the non-degenerate bilinear form $(\quad, \quad)$, 
we have the dual basis $\{E_{\underline{w_0}}^*(\mathbf{c})\}$ of $A_q(\mathfrak{n})_{\Z}$, 
called \emph{dual PBW-basis}. 

Let $\ \bar{ }\  : U_q(\mathfrak{g})\to U_q(\mathfrak{g})$ be the $\Q$-algebra involution defined by 
\[
e_i \mapsto e_i, \quad f_i \mapsto f_i, \quad q^h \mapsto q^{-h}
\]
for all $i \in I, h \in P^{\vee}$ and $q \mapsto q^{-1}$. 
Also, let $\bar{ } : A_q(\mathfrak{n})\to A_q(\mathfrak{n})$ be the involution defined by 
\[
\bar{f}(x) = \overline{f(\bar{x})}. 
\]
\begin{Prop}[{\cite{Lus90}}]\label{dualcanonical}
    There is a unique $\Z[q, q^{-1}]$-basis $\{B_{\underline{w_0}}^{\mathrm{up}}(\mathbf{c})\mid \mathbf{c}\in \Z_{\geq 0}^l\}$ 
    of $A_q(\mathfrak{n})_{\Z}$ which satisfies the following conditions : 
\begin{itemize}
    \item $B_{\underline{w_0}}^{\mathrm{up}}(\mathbf{c})$ are $\ \bar{ }\ $-invariant, 
    \item $B_{\underline{w_0}}^{\mathrm{up}}(\mathbf{c}) \in E_{\underline{w_0}}^*(\mathbf{c}) + 
    \sum_{\mathbf{c}^{\prime}<\mathbf{c}}q\Z[q] E_{\underline{w_0}}^*(\mathbf{c}^{\prime})$, 
    where $<$ denote the lexicographical order. 
\end{itemize}
    Furthermore, the set $\{B_{\underline{w_0}}^{\mathrm{up}}(\mathbf{c})\mid 
    \mathbf{c}\in \Z_{\geq 0}^l\}$ is independent of 
    the choice of reduced words. 
\end{Prop}
The basis $\{B_{\underline{w_0}}^{\mathrm{up}}(\mathbf{c})\mid \mathbf{c}\in \Z_{\geq 0}^l\}$ is called \emph{dual canonical basis} 
or \emph{upper global basis}. 

\begin{Prop}[{\cite{Lus90}}]\label{Prop : Rww}
    Let $\underline{w_0}={i_1}\cdots {i_l}, 
    \underline{w_0}^{\prime}={i^{\prime}_1}\cdots {i^{\prime}_l}$ be reduced words. 
    Suppose that $\underline{w_0}^{\prime}$ is obtained from $\underline{w_0}$ by 
    \begin{enumerate}
        \item[(a)] replacing three consecutive entries $s_i, s_j, s_i$ in $\underline{w_0}$ with $c_{ij}=-1$ 
        by $s_j, s_i, s_j$, or by 
        \item[(b)] replacing two consecutive entries $s_i, s_j$ in $\underline{w_0}$ with $c_{i,j}=0$ 
        by $s_j, s_i$. 
    \end{enumerate}
    We define a map $R_{\underline{w_0}}^{\underline{w_0}^{\prime}} : \Z_{\geq 0}^l \to \Z_{\geq 0}^l$ 
    as follows : 
    \begin{enumerate}
        \item[(a)] It takes $\mathbf{c}\in \Z_{\geq 0}^l$ to $\mathbf{c}^{\prime}$, 
    which has the same coordinate as $\mathbf{c}$ except in the three consecutive position at which 
    $\underline{w_0}, \underline{w_0}^{\prime}$ differ. 
    If $(a, b, c)$ are the coordinates of $\mathbf{c}$ at those three position, 
    the coordinates of $\mathbf{c}^{\prime}$ at those positions are 
    \[
    (b+c-\mathrm{min}(a, c), \mathrm{min}(a, c), a+b-\mathrm{min}(a, c)). 
    \]
    \item[(b)] It takes $\mathbf{c}\in \Z_{\geq 0}^l$ to $\mathbf{c}^{\prime}$, 
    which has the same coordinate as $\mathbf{c}$ except in the two consecutive position at which 
    $\underline{w_0}, \underline{w_0}^{\prime}$ differ. 
    If $(a, b)$ are the coordinates of $\mathbf{c}$ at those two position, 
    the coordinates of $\mathbf{c}^{\prime}$ at those positions are $(b, a)$. 
    \end{enumerate}
    Then, the map $R_{\underline{w_0}}^{\underline{w_0}^{\prime}}$ is bijective and 
    \[
    B_{\underline{w_0}}^{\mathrm{up}}(\mathbf{c}) = 
    B_{\underline{w_0}^{\prime}}^{\mathrm{up}}
    (R_{\underline{w_0}}^{\underline{w_0}^{\prime}}(\mathbf{c})). 
    \]
\end{Prop}
Let $\underline{w_0}, \underline{w_0}^{\prime}$ be reduced words. 
Since $\underline{w_0}^{\prime}$ can be obtained from $\underline{w_0}$
by repeating (a) and (b) 
in Proposition~\ref{Prop : Rww}, 
by composing the bijection $R$'s above, 
we can obtain the bijection $R_{\underline{w_0}}^{\underline{w_0}^{\prime}} : \Z_{\geq 0}^l \to \Z_{\geq 0}^l$ such that 
\[
    B_{\underline{w_0}}^{\mathrm{up}}(\mathbf{c}) = 
    B_{\underline{w_0}^{\prime}}^{\mathrm{up}}
    (R_{\underline{w_0}}^{\underline{w_0}^{\prime}}(\mathbf{c})). 
    \]

\begin{Prop}[{\cite{VV11}}]
    Under the isomorphism in Proposition~\ref{Prop : KLR iso}, 
    the set of isomorphism classes of simple $R$-modules corresponds to the dual canonical basis 
    up to grading shift. 
\end{Prop}

\subsection{Cuspidal modules and Quantum affine Schur-Weyl duality}
Let $\underline{w_0}=s_{i_1}\cdots s_{i_l}$ be a reduced word. 
We define $\beta_k \coloneq s_{i_1}\cdots s_{i_{k-1}}(\alpha_{i_k})$ for $1 \leq k \leq l$. 
Then, $\beta_1, \ldots, \beta_l$ runs through all positive roots, each occurring exactly once. 
\begin{Prop}[{\cite{Mc15},\cite{Ka14}}]
Let $\underline{w_0}=s_{i_1}\cdots s_{i_l}$ be a reduced word. 
    For each $1 \leq k \leq l$, 
    there exists a simple module $S_k$ of $R(\beta_k)$ such that 
    \begin{itemize}
        \item for each $\mathbf{c}=(c_1, \ldots, c_l)\in \Z_{\geq 0}^l$, the image of the class of 
        $S_l^{\circ c_l}\circ \cdots \circ S_1^{\circ c_1}$ 
        under the isomorphism in Proposition~\ref{Prop : KLR iso} is 
        the dual PBW element $E_{\underline{w_0}}^*(\mathbf{c})$, 
        \item for each $\mathbf{c}\in \Z_{\geq 0}^l$, 
        the head of $S_l^{\circ c_l}\circ \cdots \circ S_1^{\circ c_1}$ is a simple module and 
        the image of the class of $\mathrm{hd}(S_l^{\circ c_l}\circ \cdots \circ S_1^{\circ c_1})$ 
        under the isomorphism in Proposition~\ref{Prop : KLR iso} is 
        the dual canonical element $B_{\underline{w_0}}^{\mathrm{up}}(\mathbf{c})$ 
        up to a grading shift. 
    \end{itemize}
\end{Prop}
We call these simple modules $\{S_1, \ldots, S_l\}$ \emph{cuspidal modules}. 
\begin{Prop}[{\cite{KKKII15}}]
    Let $(\mathscr{Q}, \underline{w_0})$ be a pair of $Q$-datum $\mathscr{Q}$ and $\mathscr{Q}$-adapted 
    reduced word $\underline{w_0}$. 
    Let $((i_s, p_s))_{s\in \Z}$ be a corresponding admissible sequence. 
    Then, there is an exact monoidal functor $\mathscr{F}_{\mathscr{Q}} : 
    R\mathrm{-gmod} \to \mathscr{C}_{\mathscr{Q}}$ 
    such that the cuspidal module $S_k$ is sent to the fundamental module $L(Y_{i_k, p_k})$. 
    Also, this functor sends a simple module to a simple module, and it induces the algebra isomorphism 
    \[
    \overline{\mathscr{F}_{\mathscr{Q}}} : 
    K(R\mathrm{-gmod})/(q-1)K(R\mathrm{-gmod}) \overset{\sim}{\to} K(\mathscr{C}_{\mathscr{Q}}). 
    \]
\end{Prop}

\begin{Prop}[{\cite[Proposition 6.5]{PBW24}}]\label{Prop : monomial corresponding to canbase}
Let $(\mathscr{Q}, \underline{w_0})$ be a pair of $Q$-datum $\mathscr{Q}$ and $\mathscr{Q}$-adapted 
    reduced word $\underline{w_0}$. 
    Let $((i_s, p_s))_{s\in \Z}$ be a corresponding admissible sequence. 
    Then, for 
    \[A_q(\mathfrak{n})_{\Z} \overset{\sim}{\longrightarrow} K(R\mathrm{-gmod}) 
    \overset{\mathscr{F}_{\mathscr{Q}}}{\longrightarrow} K(\mathscr{C}_{\mathscr{Q}}), 
    \]
    a dual canonical element $B_{\underline{w_0}}^{\mathrm{up}}(\mathbf{c})$ for 
    $\mathbf{c} = (c_1, \ldots, c_l)$ is sent to $[L(\prod_{k=1}^{l}Y_{i_k, p_k}^{c_k})]$. 
\end{Prop}

\begin{Prop}
    Let $\mathscr{Q}, \mathscr{Q}^{\prime}$ be $Q$-data. 
    Then, isomorphism 
    \[
    \overline{\mathscr{F}_{\mathscr{Q}^{\prime}}}\circ\overline{\mathscr{F}_{\mathscr{Q}}}^{-1} : 
    K(\mathscr{C}_{\mathscr{Q}}) \overset{\sim}{\to} K(\mathscr{C}_{\mathscr{Q}^{\prime}})
    \]
    send classes of real modules to classes of real modules, 
    classes of imaginary modules to classes of imaginary modules, 
    and classes of prime modules to classes of prime modules. 
\end{Prop}
\begin{proof}\label{Prop : real to real}
    It follows from the fact that 
    $\overline{\mathscr{F}_{\mathscr{Q}}}, 
    \overline{\mathscr{F}_{\mathscr{Q}}^{\prime}}$ are ring homomorphisms and 
    send simple modules to simple modules. 
\end{proof}

\subsection{Examples of Q-data replacements}
In this subsection we use the results above to compute 
which imaginary modules arise, after changing the $Q$-datum, 
from the imaginary modules constructed in Section~\ref{sec : classification}. 
Although the computation can be carried out in principle for any type, 
we work out the case of type $D_4$ as a concrete example. 

Up to translation and permutation of nodes $1$, $2$, and $4$, 
there are four kinds of $Q$-data of type $D_4$ as follows:
\begin{align*}
    \mathscr{Q}_0 &: \xi(1)=1, \xi(2)=1, \xi(3)=0, \xi(4)=1, \\
    \mathscr{Q}_1 &: \xi(1)=-1, \xi(2)=1, \xi(3)=0, \xi(4)=1, \\
    \mathscr{Q}_2 &: \xi(1)=-1, \xi(2)=-1, \xi(3)=0, \xi(4)=1, \\
    \mathscr{Q}_3 &: \xi(1)=-1, \xi(2)=-1, \xi(3)=0, \xi(4)=-1. 
\end{align*}
Therefore, up to spectrum parameter shift and permutation of the nodes $1, 2, 4 \in I$, 
type $D_4$ admits four kinds of $\mathscr{C}_{\mathscr{Q}}$ :

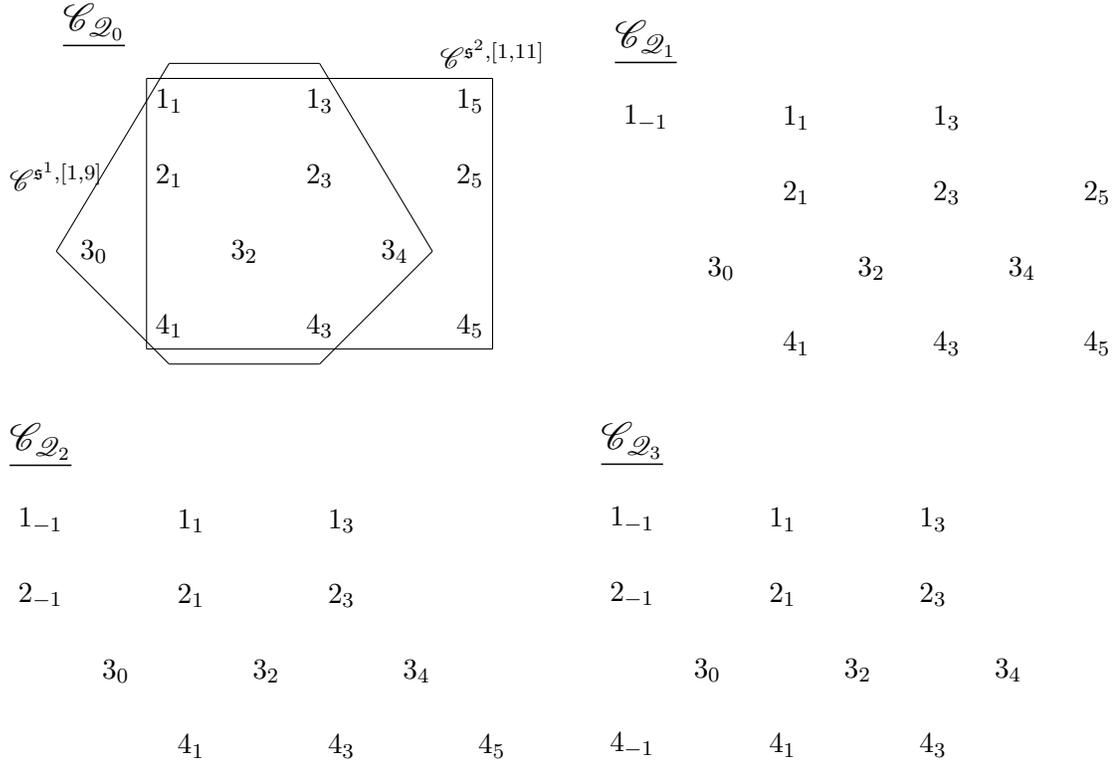
\begin{tikzpicture}[scale=1]

\draw (-0.5,-3)--(1,-0.5) ;
\draw (-0.5,-3)--(1,-4.5);
\draw (3,-0.5)--(1,-0.5);
\draw (3,-4.5)--(1,-4.5);
\draw (3,-0.5)--(4.5,-3);
\draw (3,-4.5)--(4.5,-3);
\draw(0.7,-0.7)--(5.3,-0.7) node[above] 
{$\mathscr{C}^{\mathfrak{s}^2,[1,11]}$};
\draw(0.7,-0.7)--(0.7,-4.3);
\draw(5.3,-4.3)--(0.7,-4.3);
\draw(5.3,-4.3)--(5.3,-0.7);

  \node[font=\Large] at (0,0){$\underline{\mathscr{C}_{\mathscr{Q}_0}}$};
  \node at (0,-3){$3_0$};
  \node at (1,-1) {$1_1$};
  \node at (1,-2) {$2_1$};
  \node at (1,-4) {$4_1$};
  \node at (2,-3) {$3_2$};
  \node at (3,-1) {$1_3$};
  \node at (3,-2) {$2_3$};
  \node at (3,-4) {$4_3$};
  \node at (4,-3) {$3_4$};
  \node at (5,-1) {$1_5$};
  \node at (5,-2) {$2_5$};
  \node at (5,-4) {$4_5$};
  \node at (-0.5, -2.0) {$\mathscr{C}^{\mathfrak{s}^1,[1,9]}$};

\end{tikzpicture}
\begin{tikzpicture}[scale=1]
  \node[font=\Large] at (-1,0){${\underline{\mathscr{C}_{\mathscr{Q}_1}}}$};
  \node at (-2,0){ };
  \node at (-1,-1){$1_{-1}$};
  \node at (0,-3){$3_0$};
  \node at (1,-1) {$1_1$};
  \node at (1,-2) {$2_1$};
  \node at (1,-4) {$4_1$};
  \node at (2,-3) {$3_2$};
  \node at (3,-1) {$1_3$};
  \node at (3,-2) {$2_3$};
  \node at (3,-4) {$4_3$};
  \node at (4,-3) {$3_4$};
  \node at (5,-2) {$2_5$};
  \node at (5,-4) {$4_5$};

\end{tikzpicture}

\begin{tikzpicture}[scale=1]
  \node[font=\Large] at (-1,0){${\underline{\mathscr{C}_{\mathscr{Q}_2}}}$};
  \node at (2,0.9) { };
  \node at (-1,-1){$1_{-1}$};
  \node at (-1,-2){$2_{-1}$};
  \node at (0,-3){$3_0$};
  \node at (1,-1) {$1_1$};
  \node at (1,-2) {$2_1$};
  \node at (1,-4) {$4_1$};
  \node at (2,-3) {$3_2$};
  \node at (3,-1) {$1_3$};
  \node at (3,-2) {$2_3$};
  \node at (3,-4) {$4_3$};
  \node at (4,-3) {$3_4$};
  \node at (5,-4) {$4_5$};

\end{tikzpicture}
\begin{tikzpicture}[scale=1]
  \node[font=\Large] at (-1,0){${\underline{\mathscr{C}_{\mathscr{Q}_3}}}$};
  \node at (-2,0.9) { };
  \node at (-2.3,0){ };
  \node at (-1,-1){$1_{-1}$};
  \node at (-1,-2) {$2_{-1}$};
  \node at (-1,-4) {$4_{-1}$};
  \node at (0,-3){$3_0$};
  \node at (1,-1) {$1_1$};
  \node at (1,-2) {$2_1$};
  \node at (1,-4) {$4_1$};
  \node at (2,-3) {$3_2$};
  \node at (3,-1) {$1_3$};
  \node at (3,-2) {$2_3$};
  \node at (3,-4) {$4_3$};
  \node at (4,-3) {$3_4$};
\end{tikzpicture}

where $\mathscr{C}_{\mathscr{Q}_i}$ is the smallest monoidal category 
which is closed under tensor products, subquotients, and extensions, 
and containing fundamental modules of the above and a trivial module. 
(See Definition~\ref{Def : category}. )

Up to spectral parameter shift, 
the categories $\mathscr{C}^{\mathfrak{s}^1,[1,9]}$ and $\mathscr{C}^{\mathfrak{s}^2,[1,11]}$ 
can be seen as subcategories of $\mathscr{C}_{\mathscr{Q}_0}$. 
For each $\mathscr{Q}_i$, 
we fix $\mathscr{Q}_i$-adapted reduced words as follows : 
\begin{align*}
    i=0 : \mathbf{i}_0 = 312431243124, \\
    i=1 : \mathbf{i}_1 = 132413241324, \\
    i=2 : \mathbf{i}_2 = 123412341234, \\
    i=3 : \mathbf{i}_3 = 124312431243. 
\end{align*}

In $\mathscr{C}^{\mathfrak{s}^1,[1,9]}, \mathscr{C}^{\mathfrak{s}^2,[1,11]} \subset 
\mathscr{C}_{\mathscr{Q}_0}$, 
there are prime imaginary modules $L(3_03_4), L(1_11_52_12_54_14_5)$. 
(See Example~\ref{Ex : Dn primeimaginary}, \ref{Ex : D4 primeimaginary}. )
By the map 
\[
\mathscr{F}_{\mathscr{Q}_0} : A_q(\mathfrak{n})_{\Z} \to K(\mathscr{C}_{\mathscr{Q}_0})
\]
in Proposition~\ref{Prop : monomial corresponding to canbase}, 
the classes of $L(3_03_4), L(1_12_14_11_52_54_5)$ come from 
$B^{up}_{\mathbf{i}_0}(1,0,0,0,0,0,0,0,1,0,0,0)$ and 
$B^{up}_{\mathbf{i}_0}(0,1,1,1,0,0,0,0,0,1,1,1)$. 

By considering how to obtain $\mathbf{i}_1, \mathbf{i}_2, \mathbf{i}_3$ from $\mathbf{i}_0$ by 
using braid relations, we can calculate 
\begin{align*}
    B^{up}_{\mathbf{i}_0}(1,0,0,0,0,0,0,0,1,0,0,0) &= 
    B^{up}_{\mathbf{i}_1}(0,1,0,0,0,0,0,0,0,1,0,0) \\
    &= B^{up}_{\mathbf{i}_2}(0,0,1,0,0,0,0,0,0,0,1,0)\\
    &= B^{up}_{\mathbf{i}_3}(0,0,0,1,0,0,0,0,0,0,0,1), 
\end{align*}
\begin{align*}
    B^{up}_{\mathbf{i}_0}(0,1,1,1,0,0,0,0,0,1,1,1) &= 
    B^{up}_{\mathbf{i}_1}(1,0,0,0,1,1,0,0,0,0,1,1) \\
    &= B^{up}_{\mathbf{i}_2}(1,1,0,0,0,0,1,1,0,0,0,1)\\
    &= B^{up}_{\mathbf{i}_3}(1,1,1,0,0,0,0,0,1,1,1,0), 
\end{align*}
by Proposition~\ref{Prop : Rww}. 
Then, by proposition~\ref{Prop : monomial corresponding to canbase}, 
we find new prime imaginary modules 
\begin{align*}
&L(3_03_4), L(1_{-1}1_13_22_54_5) \in \mathscr{C}_{\mathscr{Q}_1}, \\
&L(3_03_4), L(1_{-1}2_{-1}3_24_34_5)\in \mathscr{C}_{\mathscr{Q}_2}, \\
&L(3_03_4), L(1_{-1}2_{-1}4_{-1}1_32_34_3)\in \mathscr{C}_{\mathscr{Q}_3}. 
\end{align*}

Also when we set $\mathbf{i}^{\prime} = 132431432434$ as \cite{Le03}, 
similarly we can get 
\[
B^{up}_{\mathbf{i}_0}(1,0,0,0,0,0,0,0,1,0,0,0) = 
B^{up}_{\mathbf{i}^{\prime}}(0,1,0,0,0,0,1,0,0,0,1,0). 
\]
So, the prime imaginary module $L(3_03_4)$ corresponds to the imaginary vector 
found in \cite{Le03}.

\bibliographystyle{amsplain}
\bibliography{refs}

\end{document}